\documentclass[11pt]{amsart}
\usepackage[mathscr]{euscript}
\usepackage{marginnote,bm,amscd}
\usepackage{amssymb}
\usepackage[foot]{amsaddr}
\usepackage{color}
\usepackage{graphicx}
\usepackage{hyperref}
\usepackage{varwidth}

\usepackage[most]{tcolorbox}

\tcbset{colback=yellow!10!white, colframe=red!50!black, 
        highlight math style= {enhanced, 
            colframe=red,colback=red!10!white,boxsep=0pt}
        }

\newtheorem*{thm*}{Theorem}

\newtheorem*{prop*}{Proposition}

\newtheorem*{lem*}{Lemma}

\newtheorem{exercise}{Exercise}
\newtheorem{theorem}{Theorem}[section]
\newtheorem{lemma}[theorem]{Lemma}

\newtheorem{proposition}[theorem]{Proposition}
\newtheorem{corollary}[theorem]{Corollary}
\newtheorem{definition}[theorem]{Definition}
\newtheorem{remark}[theorem]{Remark}

\newcommand\jwf[1]{\textcolor{red}{#1}}

\numberwithin{equation}{section}

\newcounter{CurrentSection}
\newcounter{CurrentTheorem}
\newcounter{CounterSectionImprintingMethod}
\newcounter{CounterTheoremImprintingMethod}
\newcounter{CounterSectionMetrizabilityConditions}
\newcounter{CounterTheoremMetrizabilityConditions}

\renewcommand{\epsilon}{\varepsilon}

\newcommand{\wt}{\widetilde}

\newcommand{\cl}{\operatorname{cl}}

\providecommand{\ker}[1]{$\text{ker}\ {#1}$}

\makeindex

\title[Polyfold Constructions]{Polyfold and SFT Notes I:\\
A Primer on Polyfolds and Construction Tools}
\author{ J. W. Fish \and H. Hofer}

\email{hofer@ias.edu}
\email{joel.fish@umb.edu}

\date{\today}                                           

\begin{document}
\maketitle
\tableofcontents

\newpage
These notes are essentially the first few chapters from the upcoming book
  \cite{FH-poly}.
As such, the authors would request that citations and references to these
  notes also be directed at the forthcoming book.

\vspace{1cm}
\begin{center}
{\bf J. W. Fish and H. Hofer, \\
Polyfold Constructions: Tools, Techniques, and Functors}\\
\end{center}
\vspace{1cm}
\noindent which we make  available for the upcoming workshop\\

\begin{center}
{\bf Workshop on Symplectic Field Theory IX:}\\
\vspace{0.5cm}
{\bf POLYFOLDS FOR SFT}\\
\vspace{0.5cm}
Augsburg, Germany\\
 Monday, 27 August 2018 - Friday, 31 August 2018\\
\vspace{0.5cm}
A Pre-course takes place on the preceding weekend:\\

Saturday, 25 August 2018 - Sunday, 26 August 2018
\end{center} 
\vspace{9cm}
\pagebreak

\section{Introduction}

\subsection{General Context and Goals}
The  so-called polyfold theory introduced in a series of papers,
  \cite{HWZ2,HWZ3,HWZ3.5,H2014,HWZ5}, describes a pairing of a
  generalization of differential geometry and a generalized nonlinear
  functional analysis.
The aim of the overall theory is to provide a framework in which  moduli
  spaces can be constructed as they arise naturally within specific fields
  of mathematics, for example in symplectic geometry.
In this context and from an analytical viewpoint, the study of moduli
  spaces is the study of solutions of families of nonlinear
  elliptic differential equations on varying domains with possibly even
  varying targets up to a notion of isomorphism.
Varying domains arise through bubbling-off phenomena which implies
  compactness problems. 
Moreover, the fact that a family can be isomorphic to itself in different
  ways implies in general the occurrence of transversality issues.
The reference volume \cite{HWZ2017} gives a comprehensive description of
  the abstract theory which can be employed once a problem has been lifted
  into the abstract framework.
In the current note  we describe the underlying (abstract) theory for the
  construction concrete  M-polyfolds by using  a building-block type system,
  \cite{Lego}.

\subsection{Warm-up Exercises for the Reader}
The first three exercises test the knowledge about basic manifold theory
  including the implicit function theorem. It helps to focus the ideas in
  the direction we are going to exploit further.
The first exercise gives a different approach to smooth manifolds.\\

\begin{exercise} 
  \label{EXERC1}\label{EX_1}
  \hfill\\
Assume that $U\subset {\mathbb R}^n$ is an open subset, $Y$ a set, and
  $\oplus:U\rightarrow Y$ a surjective map.
Associated to \(\oplus\), we define \(\mathcal{T}_{\oplus}\) to be the
  finest topology on \(Y\) for which \(\oplus\) is continuous; that is,
  \(\mathcal{T}_{\oplus}\) is the quotient topology on \(Y\) associated to
  \(\oplus\).
Assume that \(\mathcal{T}_\oplus\) is metrizable. 
Suppose further that for every point $y\in Y$ there exists $V_y\in
  {\mathcal T}_\oplus$ and a map $H_y:V_y\rightarrow U$ such that 
  \begin{itemize}
    \item[(i)] $\oplus\circ H_y=Id_{V_y}$.
    \item[(ii)] $H_y\circ\oplus:\oplus^{-1}(V_y)\rightarrow U$ is a smooth
      map.
    \end{itemize}
Show that the metrizable space $(Y,{\mathcal T}_\oplus)$ has the structure
  of a smooth manifold uniquely characterized by the following two
  properties.
\begin{itemize}
  \item[(1)] $\oplus:U\rightarrow Y$ is a smooth map.
  \item[(2)] Every map  $H:V\rightarrow U$, $V\in {\mathcal T}_{\oplus}$,
  satisfying (i) and (ii) is smooth.
  \end{itemize}
\end{exercise}
%

If stuck with Exercise \ref{EX_1}, try to find inspiration in H. Cartan's
  last mathematical theorem, see \cite{Cartan}.  
As a consequence of the exercise we obtain a method to define new
  manifolds from old ones.\\

\begin{exercise}
  \label{EX_2}
  \hfill\\
Assume that $M$ is a smooth manifold, $Y$ a set, and $\oplus:M\rightarrow
  Y$ a surjective map.
Assume that the associated quotient topology ${\mathcal T}_\oplus$ on $Y$
  is metrizable.
Suppose further that for every point $y\in Y$ there exists $V_y\in
    {\mathcal T}_\oplus$ and a map $H_y:V_y\rightarrow M$ such that
  \begin{itemize}
    \item[(i)] $\oplus\circ H_y=Id_{V_y}$.
    \item[(ii)] $H_y\circ\oplus:\oplus^{-1}(V_y)\rightarrow M$ is a smooth
      map.
    \end{itemize}
Show that the metrizable space $(Y,{\mathcal T}_\oplus)$ has the structure
  of a smooth manifold uniquely characterized by the following two
  properties.
\begin{itemize}
  \item[(1)]
  $\oplus:M\rightarrow Y$ is a smooth map.
  \item[(2)] 
  The maps $H:V\rightarrow M$, $V\in {\mathcal T}_\oplus$, satisfying (i)
  and (ii) are smooth.
  \end{itemize}
\end{exercise}
%

These two exercises prompt the following definition.\\

\begin{definition}[manifold imprinting]
  \label{DEF_manifold_imprinting}
  \hfill\\
A surjective map $\oplus:M\rightarrow Y$ from a smooth manifold \(M\) to a
  set $Y$ is called a \emph{manifold imprinting}\index{manifold
  imprinting}, provided the quotient topology ${\mathcal T}_\oplus$ on $Y$
  is metrizable and for every $y\in Y$ there exists $V_y\in {\mathcal
  T}_\oplus$ and a map $H_y:V_y\rightarrow M$ such that 
  \begin{itemize}
    \item[(1)] 
      $\oplus\circ H_y=Id_{V_y}$
    \item[(2)] 
      The map $H_y\circ\oplus:\oplus^{-1}(V_y)\rightarrow M$ is smooth.
    \end{itemize}
\end{definition}
%
\begin{remark}
Henceforth, given \(\oplus:M\to Y\) we will let \(\mathcal{T}_\oplus\)
  denote the associated quotient topology on \(Y\).
\end{remark}
%

The following exercise shows the naturality  of manifold imprintings.\\

\begin{exercise}
  \label{EX_3}
  \hfill \\
Assume that $\oplus_1:M\rightarrow X$ is a manifold imprinting. 
Let $\oplus_2:X\rightarrow Y$ be a surjective map between two sets. 
The following two statements are equivalent:
  \begin{itemize}
  \item[(1)] 
    $\oplus_2\circ\oplus_1:M\rightarrow Y$  is a manifold imprinting.
  \item[(2)] 
    With $X$ equipped with the $\oplus_1$-structure the map
    $\oplus_2:X\rightarrow Y$ is a manifold imprinting.
  \end{itemize}
Moreover in the case that  (1) or (2), and therefore both hold, the two
  structures induced on $Y$ are the same.
\end{exercise}
%

The next exercise relates the previous discussion to the standard
  treatment of manifolds.\\

\begin{exercise}
  \label{EX_4}
  \hfill\\
Let $M$ be a smooth connected manifold. 
Then there exists an open subset $O$ of some ${\mathbb R}^N$ and an
  $\oplus$-construction $\oplus:O\rightarrow M$, where $M$ is just
  considered as a set, with the following properties.
\begin{itemize}
  \item[(1)] 
  A map $f:M\rightarrow N$  to a smooth manifold $N$ is smooth if and only
  if $f\circ\oplus:O\rightarrow N$ is smooth.
  \item[(2)] 
  A map $g:N\rightarrow M$ from the smooth manifold $N$ to the topological
  space \((M, \mathcal{T}_\oplus)\)  is smooth if and only if  it is
  continuous and for each $m\in M$ with $U=U(m)\in {\mathcal T}_\oplus$
  and  map $H_m:U\rightarrow O$  as given in Definition
  \ref{DEF_manifold_imprinting}, the map $H_m\circ g: g^{-1}(U)\rightarrow
  O$ is smooth.
  \end{itemize}
\end{exercise}
%

Next we transfer the knowledge gained through the four exercises to the
  case of Banach spaces and Banach manifolds.
This follows from the familiar fact that finite-dimensional calculus can
  be generalized to Banach spaces based on the notion of Fr\'echet
  differentiability.\\

\begin{exercise}
  \label{EX_5}
  \hfill \\
Show that Exercises 1--3 hold if we replace ${\mathbb R}^n$ by a Banach
  space, $M$ by a Banach manifold, and generalize smoothness of a map as
  smooth Fr\'echet differentiability.
We leave the precise formulations to the reader.
Generalizing Exercise 4 might be somewhat tricky and perhaps not even
  true. 
Questions about the (non-)existence of smooth partitions of unity on
  Banach spaces enter the picture.
See \cite{Fry,GTWZ}  for the problems associated to questions about smooth
  functions on Banach spaces.
A generalization to Hilbert spaces might be possible.
\end{exercise}
%

In the context of Banach spaces there are other notions of
  differentiability, particularly if we allow additional structures on
  them.\\

\begin{definition}[sc-sctructure]
  \label{DEF_sc_structure}
  \hfill\\
Let $E$ be a Banach space. 
A \emph{sc-structure} on $E$ is a nested sequence of linear subspaces
  $\ldots \subset E_{i+1}\subset \ldots \subset E_1\subset E_0=E$, such that
  each $E_i$ is equipped with a Banach space structure so that the following
  holds.
\begin{itemize}
  \item[(i)] 
  The inclusion operator $E_{i+1}\rightarrow E_i$ is a compact operator
  for every $i=0,1,...$
  \item[(ii)] 
  $E_{\infty}=\bigcap_{i=0}^{\infty} E_i$ is dense in every $E_i$.
  \end{itemize}
\end{definition}
%

We note that a finite-dimensional vector space has a unique sc-structure.
We can define a new notion of smoothness, called sc-smoothness using
  this auxiliary structure.\\

\begin{definition}[sc-differentiable]
  \label{DEF_sc_differentiable}
  \hfill\\
Assume the Banach spaces $E$ and $F$ are equipped with sc-structures and
  $U\subset E$ is an open subset.
A map $f:U\rightarrow F$ is said to be sc$^1$, i.e. one times
  sc-differentiable, provided:
  \begin{itemize}
    \item[(1)] 
    For every $i$ the map $f:U\cap E_i\rightarrow F_i$ is well-defined and
    continuous, where $U\cap E_i$ is equipped with the topology induced
    from $E_i$.
    \item[(2)] 
    For every $x\in U\cap E_1$ there exists a bounded linear operator
    denoted $Df(x):E_0\rightarrow F_0$ such that
    \begin{align*}                                                        
      \lim_{\substack{|h|_1\rightarrow 0\\ x+h\in U\cap E_1}}
      \frac{|f(x+h)-f(x)-Df(x)h|_0}{|h|_1} =0.
      \end{align*}
    \item[(3)] 
    The map $Tf$ defined for $x\in U\cap E_1$, $h\in E_0$ by
    \begin{align*}                                                        
      Tf(x,h):=(f(x),Df(x)(h))
      \end{align*}
    maps $U_{1+i}\times E_i$ to $F_{1+i}\times
    F_i$, and for each \(i\in \mathbb{N}\) the associated map is
    continuous as a map $U_{1+i}\times E_i\rightarrow F_{1+i}\times F_i$.
    \end{itemize}
\end{definition}
%

Note that a map only satisfying (1) above is called sc\(^0\) or {\bf
  sc-continuous}.\index{sc-continuous}\index{sc$^0$}
Additionally, We note that we can view $Tf :U_1\times E\rightarrow
  F_1\times F$.
We can consider $U_1\times E$ as an open subset of $E_1\times E$ equipped
  with the sc-structure given by $E_{1+i}\times E_i$ and similarly for
  $F_1\times F$.
We shall define $TU:= U_1\times E \subset E_1\times E$.
Then if $f$ is sc$^1$ we obtain an sc$^0$ map $Tf:TU\rightarrow TF$. 
We say that $f$ is sc$^2$ provided $Tf$ is sc$^1$. 
This way we can define inductively what it means that a map is  sc$^k$ or
  even sc-smooth. 
In the finite-dimensional case sc-differentiability is precisely classical
  differentiability.
Looking at property (2) in the previous definition we see that we view the
  maps as going from level $1$ to level $0$.
Of course, this makes it a priori doubtful if the notion of
  sc-differentiability allows for the chain rule. \\

\begin{exercise}
  \label{EOOPI6}\label{EX_6}
  \hfill\\
Assume that $E,F$ and $G$ are are Banach spaces equipped with
  sc-structures and $U\subset E$, $V\subset F$ are open subsets.
Assume that $f:U\rightarrow F$ and $g:V\rightarrow G$ are sc$^1$ such that
  $f(U)\subset V$.
Show that $g\circ f:U\rightarrow G$ is sc$^1$ and $T(g\circ f) =(Tg)\circ
  (Tf)$.
With other words the chain rule holds.
The same conclusion holds if the maps are sc$^k$ or sc-smooth.
\end{exercise}
%

Exercise \ref{EX_6} is nontrivial and the proof utilizes strongly the fact
  that the inclusion operators are compact operators.\\
 
Consider the Banach space $E:=C^0(S^1,{\mathbb R}^n)$ of continuous maps
  defined on the circle with image in ${\mathbb R}^n$.
Here $S^1={\mathbb R}/{\mathbb Z}$ with the usual smooth manifold structure. 
We have a canonical smooth map ${\mathbb R}\rightarrow S^1$.\\
 
\begin{exercise}
  \label{EX_7}
  \hfill \\
The Banach space map $\Phi:{\mathbb R}\times E\rightarrow E$ defined by $
  \Phi(c,u)(t):= u(t+c)$ is nowhere Fr\'echet differentiable. 
The same holds if instead of $E=C^0(S^1,{\mathbb R}^n)$ we take
  $E=C^m(S^1,{\mathbb R}^n)$.
\end{exercise}
%
 
The standard sc-structure on $E=C^0(S^1,{\mathbb R}^n)$ is given by
  $E_i=C^i(S^1,{\mathbb R}^n)$, the Banach space of  $i$-times
  continuously differentiable maps.
We equip the Banach space ${\mathbb R}\times E$ with the sc-structure
  $({\mathbb R}\times E_i)_{i\in \mathbb{N}}$. \\

\begin{exercise}
  \label{EX_8}
  \hfill\\
The map $\Phi:{\mathbb R}\times E\rightarrow E$ defined by $
  \Phi(c,u)(t):= u(t+c)$ is sc-smooth.
\end{exercise}
%

Now we are ready for the final two exercises. 
First though, we note that by making use of the novel notion of
  sc-differentiability, we can generalize the above notion of a manifold
  imprinting as follows.\\
It is worth mentioning that the following definition will be fundamental
  idea to be further developed and employed throughout these notes.

\begin{definition}[preliminary M-polyfold imprinting]
  \hfill\\
A surjective map \(\oplus:U\to Y\) defined on an open subset \(U\) of a
  Banach space \(E\) equipped with  an sc-structure, with target being a
  set \(Y\) will be called an M-polyfold imprinting, provided the quotient
  topology \(\mathcal{T}_\oplus\) on \(Y\) is metrizable and for every
  \(y\in Y\) there exists \(V_y \in \mathcal{T}_\oplus\) and a map
  \(H_y:V_y\to U\) such that
  \begin{enumerate}                                                       
    \item \(\oplus\circ H_y = Id_y.\)
    \item \(H_y\circ \oplus: \oplus^{-1}(V_y)\to U\) is an sc-smooth map.
    \end{enumerate}
%
%
%
\end{definition}
%

The following exercise shows that something new happens with an unlikely
  space having some kind of smooth structure.\\

\begin{exercise}
  \label{EX_9}\label{EXC8}
  \hfill\\
Consider the subset $\Sigma$ of ${\mathbb R}^2$ defined by
  \begin{align*}                                                          
    \Sigma :=\left \{(x,y)\in {\mathbb R}^2\ |\ x^2+y^2<1\right \} \bigcup
    \left\{(x,y)\in {\mathbb R}^2\ |\ (x-1)^2 +y^2 =1\right\},
    \end{align*}
  see Figure \ref{FIG100}.
There exists an infinite-dimensional  Hilbert space $H$,  equipped with an
  sc-structure using Hilbert spaces $H_i$, and an $\oplus$-construction
\begin{align*}                                                            
  \oplus:U\rightarrow \Sigma, 
  \end{align*}
   where $U$ is an open subset of $H$, such that the quotient topology is
   precisely the metrizable subspace topology induced from ${\mathbb
   R}^2$, and moreover $\oplus(U\cap H_i)=\Sigma$ for every $i$.   

\end{exercise}
%

\begin{figure}[h!]
  \label{FIG100}\label{FIG_01}
  \includegraphics[width=6.5cm]{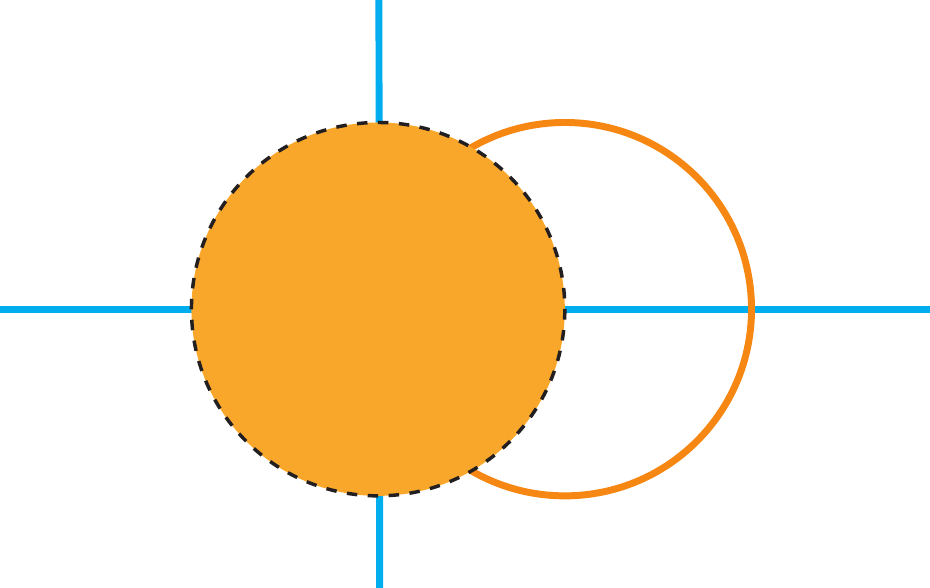}
  \caption{The metrizable space $\Sigma$.}
  \label{FIG_mylabel}
\end{figure}

From Exercise \ref{EX_9} we derive that $\oplus$ equips $\Sigma$ with some
  kind of smooth structure. This can be viewed as the starting point of a
  systematic study of a new kind of smooth structure on topological spaces. 
We shall explain this shortly, however for the moment we simply mention
  that the function spaces that contain the breaking/bubbling-off phenomena
  that arise naturally in moduli problems in symplectic geometry can be
  given precisely this sc-smooth structure.
This is the contents of these notes, as well as the book \cite{FH-poly}. 

\begin{exercise}
  \label{EX_10}
  \hfill\\
Show that the Hilbert space $H$ in Exercise \ref{EXC8} can never be picked
  finite-dimensional.
\end{exercise}
%

Even with sufficient mathematical background it presumably will take a
  while to do the exercises without consulting the literature or starting to
  read further in this book.
Trying to do them might, however, be ultimately the fastest way for
  entering the field.

\section{A Primer on M-Polyfolds}
  \label{APP7}\label{SEC_primer_on_m_polyfolds}
The theory of polyfolds has been developed in a series of papers,
  \cite{HWZ2,HWZ3,HWZ3.5,HWZ8.7}, and its main purpose is to generalize
  differential geometry and nonlinear functional analysis in a manner
  suitable for studying families of nonlinear elliptic differential
  equations, particularly in the case in which solutions are isomorphism
  classes of maps between manifolds, in which the domains are allowed to
  vary geometrically and change topology abruptly.
In particular, key components of the theory are a nonlinear Fredholm
  theory, an implicit function theorem, and an abstract perturbation
  algorithm robust enough to guarantee that after generic perturbation,
  certain compactified moduli spaces of solutions are cut out transversely,
  and are akin to smooth submanifolds with their topology and smooth
  structure inherited from the polyfold structure on the ambient space of
  maps.
In this section, we will rapidly recall the most salient definitions and
  results of the polyfold theory upon which we will build in later sections.

%
\subsection{Sc-Structures and Sc-Calculus}  \label{qsec1.1}
  \label{SEC_sc_structures_calculus}
An {\bf sc-structure}\index{sc-structure} on a Banach space \(E\) is a
  sequence of nested linear subspaces
  \begin{align*}                                                          
    \cdots \subset E_{i+1}\subset E_i\subset \cdots \subset E_1 \subset
    E_0=E,
    \end{align*}
  each equipped with a Banach space structure, and  with the
  property that for each \(i\in \mathbb{N}=\{0,1,2, \ldots\}\) the
  inclusion operator \(E_{i+1}\rightarrow E_i\) is compact and the
  intersection \(E_\infty:= \bigcap_{i=0}^\infty E_i\) is dense in every
  \(E_i\).
If \(E\) is equipped with such a structure we shall refer to it as an
  {\bf sc-Banach space}\index{sc-Banach space}.
Note that as a consequence of the requirements a finite-dimensional vector
  space can have precisely one sc-structure, namely the constant
  sc-structure:  \mbox{\(E_0 = E_1 = E_2 = \cdots\)}.

It is worth mentioning that scales of Banach spaces are familiar
  objects in interpolation theory, see \cite{Tr}, however our perspective
  is somewhat different.
Indeed, we regard \((E_i)_{i=0}^\infty\) as a given structure on \(E\),
  and it will become clear that from this perspective the sc-structure
  \((E_i)_{i=0}^\infty\) can be understood as a type of smooth structure on
  \(E\) for which a seemingly novel generalized calculus can be developed
  which has surprising properties.
It should be noted that scales have been used in geometric settings
  before; see for example the work by D. Ebin,  \cite{Ebin},  and
  H. Omori, \cite{Omori}.
The latter of these is somewhat closer to the viewpoint in our paper,
  however Omori does not use \emph{compact} scales, and this turns out
  to be a crucial condition for our applications.
For example, without the compactness assumption one does not obtain the
  new local models for a generalized differential geometry in which 
  bubbling-off and trajectory-breaking is a smooth phenomenon.

Returning to our definitions, we note that elementary concepts from
  linear algebra, such as direct sums, subspaces, and linear complements,
  have natural extension to sc-Banach spaces.
For example, given two sc-Banach spaces \(E\) and \(F\) we can form
  their {\bf sc-direct sum}\index{sc-direct sum} which is defined to be
  the direct sum \(E\oplus F\) equipped with sc-structure  \((E\oplus
  F)_i:=E_i\oplus F_i\).
Similarly, a subspace \(F\) of  an sc-Banach space \(E\) is called an {\bf
  sc-subspace} \index{sc-subspace} provided \(F\) is closed and the sequence
  \((F_i)_{i=0}^\infty\) given by \(F_i = F\cap E_i\) defines an
  sc-structure on \(F\).
An sc-subspace \(F\) of an sc-Banach space \(E\) has an {\bf
  sc-complement}\index{sc-complement} provided there exists an algebraic
  complement \(G\) of \(F\) in \(E\) which is an sc-subspace for which
  \(E_i = F_i\oplus G_i\) for each \(i\in \mathbb{N}\); we call \(G\) the
  sc-complement of \(F\) in \(E\).

Given sc-Banach spaces \(E\) and \(F\), a linear operator
  \mbox{\(T:E\rightarrow F\)} is called an {\bf
  sc-operator}\index{sc-operator} provided that for each \(i\in
  \mathbb{N}\) we have \mbox{\(T(E_i)\subset F_i\)} and \(T:E_i\rightarrow
  F_i\) is continuous.
A {\bf sc-isomorphism}\index{sc-isomorphism} is a linear bijection
  \(T:E\rightarrow F\) for which \(T\) and \(T^{-1}\) are sc-operators.
Additionally, there is a very important class of so-called
  sc\(^+\)-operators, which play the role of compact operators in the
  polyfold theory.
To define them we say a linear map \(S:E\rightarrow F\) is an {\bf
  sc\(^+\)-operator} \index{sc\(^+\)-operator} provided that for each \(i\in
  \mathbb{N}\), we have \(S(E_i)\subset F_{i+1}\) and \(S:E_i\rightarrow
  F_{i+1}\) is continuous.
Observe then that a consequence of requiring the inclusion operators,
  \(F_{i+1}\rightarrow F_i\), to be compact is that the level-wise
  restriction of an sc\(^+\)-operator \(S\) to a map \(S:E_i\rightarrow
  F_i\) is a compact operator for each \(i\in \mathbb{N}\).
Finally, a {\bf linear sc-Fredholm operator}\index{sc-Fredholm operator}
  is a linear sc-operator \(T:E\rightarrow F\) for which \(\ker(T)\) and
  \(R(T)\) are sc-subspaces of \(E\) and \(F\) respectively, such that each
  has an sc-compliment, and \(\ker(T)\) and \(F/R(T)\) are each
  finite-dimensional.

Given one sc-Banach space \(E\), we can construct another in the
  following manner.
For each \(j\in \mathbb{N}\), we define the sc-Banach space \(E^j\) to
  be the Banach space \(E_j\) equipped with the sc-structure \({(E^j)}_i:=
  E_{j+i}\), and we say that \(E^j\) is obtained from \(E\) with the {\bf
  index raised}\index{index raising} by \(j\).
Note that \(E^0=E\). 
Also note that \(S:E\rightarrow F\) is an sc\(^+\)-operator provided it
  induces an sc-operator \(E\rightarrow F^1\).

In order to lay the groundwork for manifolds (or more generally
  M-polyfolds) with boundary and corners, we must first establish their
  model analogues in sc-Banach spaces.
To that end, we define a {\bf partial quadrant}\index{partial quadrant}
  (or sector)\index{sector} \(C\) in an sc-Banach space \(E\) to be a closed
  convex subset with the property that there exists an sc-Banach space
  \(W\), a suitable \(n\in \mathbb{N}\), and an sc-isomorphism \(T:{\mathbb
  R}^n \oplus W\rightarrow E\) for which \(T([0,\infty)^n\oplus W)=C\).
In order to define the degree of a corner, we recall that there is
  a well-defined map \(d_C:C\rightarrow {\mathbb N}\), called the {\bf
  degeneracy index}\index{degeneracy index}, which is defined by
  \(d_C(e)=k\), where \(k\) is the number of indices \(i\in \{1,..,n\}\)
  for which \(r_i=0\), where \(T(r,w)=e\) is as described above.
Note that this definition does not depend on the actual choice of \(T\).

Next we move on to review a calculus on sc-Banach spaces and on partial
  quadrants contained therein; in other words, the sc-calculus.
Let \(E\) and \(F\) be sc-Banach spaces and \(C\subset E\) a partial quadrant. 
Assume that \(U\subset C\) is a relatively open subset and \(f:U\rightarrow
  F\) a map.
We say that \(f\) is {\bf sc-continuous}\index{sc-continuous} provided
  that  for each \(i\in \mathbb{N}\), both \(f(U_i)\subset F_i\) and the map
  \(f:U_i\rightarrow F_i\) is continuous; here \(U_i=E_i\cap U\).
Alternatively, we shall also say that \(f\) is an {\bf
  sc\(^0\)-map}\index{sc\(^0\)-map}.

Given a relatively open subset \(U\subset C\), we call the filtration
  \((U_i)_{i=0}^\infty\) the sc-structure on \(U\).
In the same way that we raised the index on sc-Banach spaces, so too
  can we raise the index of \(U\) by \(j\) to obtain \(U^j\).
The {\bf tangent}\index{tangent} of \(U\) is the relatively open subset
  \(TU:=U_1\oplus E_0\) of \(C_1\oplus E_0\subset E_1\oplus E_0\) and it
  shall always be equipped with the filtration \((U_{1+i}\oplus
  E_i)_{i=0}^\infty\).
Alternatively we can define  \(TU:=U^1\oplus E\subset  E^1\oplus E\).
The first crucial definition is in regards to differentiability, and
  is as follows.

\begin{definition}[sc differentiable]
  \label{def:sc1}\label{DEF_sc_diff}\index{sc\(^1\)-map}
  \hfill\\
Let \(E\) and \(F\) be sc-Banach spaces and \(C\subset E\) a partial
  quadrant, and \(U\subset C\) a relatively open subset.
A map \mbox{\(f:U\rightarrow F\)} is said to be {\bf sc}\(^{\bm{1}}\)
  provided the following holds.
\begin{itemize}
  \item[(1)] 
    \(f\) is sc\(^0\).
  \item[(2)] 
    For each \(x\in U_1\) there exists a bounded linear operator
  \begin{align*}                                                          
    Df(x):E_0\rightarrow F_0
    \end{align*}
  such that
  \begin{equation*}                                                       
    \lim_{\substack{|h|_1\rightarrow 0\\ x+h\in U_1}}
    \frac{\big|f(x+h)-f(x)-Df(x)h\big|_0}{|h|_1}=0
    \end{equation*}
  \item[(3)] 
    The map \(Tf:TU\rightarrow TF\) given by \(
    Tf(x,h):=(f(x),Df(x)(h))\) is sc\(^0\).
  \end{itemize}
\end{definition}
%
In the case that \(f\) is sc\(^1\), the associated map \(Tf:TU\rightarrow
  TF\), defined in (3) above, is called the {\bf tangent} of \(f\).
The following non-trivial result essentially states that the chain
  rule\index{chain rule} holds for compositions of sc\(^1\) maps.

\begin{theorem}[sc chain rule]
  \label{thm_sc_chain_rule}
  \hfill\\
Let \(E\), \(F\), and \(G\) be sc-Banach spaces and \(C\subset E\) and
  \(D\subset F\) partial quadrants.
Suppose \(U\subset C\) and \(V\subset D\) are relatively open and
  \(f:U\rightarrow F\) and \(g:V\rightarrow G\) are sc\(^1\) with the
  property that \(f(U)\subset V\).
Then \(g\circ f:U\rightarrow G\) is sc\(^1\) and \(T(g\circ f) =(Tg)\circ
  (Tf)\).
\end{theorem}
%
\begin{proof}
See \cite{HWZ2,HWZ2017}.
\end{proof}
For each fixed \(i\geq 2\), a map \(f:U\rightarrow F\) is said to be of
  class sc\(^i\), provided \(Tf\) is of class \(sc^{i-1}\).
If \(f\) is of class sc\(^i\) for all \(i\in \mathbb{N}\), then we say that
  \(f\) is  sc-smooth or alternatively of class sc\(^\infty\).
\begin{remark}\hfill\\
Consider a vector \(e\in E\) and the associated affine map
  \(\phi:E\rightarrow E:x\rightarrow x+e\).
We note that this map is sc\(^0\) if and only if \(e\in E_\infty\), however
  in this case the map is also sc\(^\infty\).
  
\end{remark}
%

%
\subsection{Sc-Smooth Models and M-Polyfolds}
  \label{sec:models and M-polyfolds}\label{qsec1.2}
  \label{SEC_sc_smooth_models}
Of particular interest are sc-smooth maps \(r:U\rightarrow U\), where
  \(U\subset C\) is relatively open in a partial quadrant \(C\subset E\),
  and which satisfy \(r\circ r=r\).
Such a map is called an {\bf sc-smooth retraction}\index{sc-smooth
  retraction}, and the associated image, \(O=r(U)\), is called an {\bf
  sc-smooth retract}\index{sc-smooth retract} (with respect to \((E,C)\)).
\begin{definition}[sc smooth model]
  \label{def:sc smooth model}\label{DEF_sc_smooth_model}
  \hfill\\
A \emph{sc-smooth model} \index{sc-smooth model} is a tuple \((O,C,E)\),
  where \(C\subset E\) is a partial quadrant in the sc-Banach space \(E\),
  and \(O=r(U)\) for some sc-smooth retraction \(r:U\rightarrow U\) where
  \(U\subset C\) is relatively open and \(O\subset U\).
\end{definition}
%
Given an sc-smooth model \((O,C,E)\) it may happen that \(O=r(U)\)
  and \(O=s(V)\) for relatively open subsets \(U,V\subset C\) and two
  different sc-smooth retractions \(r\) and \(s\).
At the level of point-set topology, the fact that a local model can
  be defined by two different retractions is a non-issue, since both
  retractions preserve levels of the sc-structure and have identical images.
However, since sc-smooth models are meant to provide the local models for
  a differential geometry, there may be preliminary concern that different
  sc-smooth retractions may yield different sc-differentiable structures
  on \((O, C, E)\); we address this at present. 
First, we consider two sc-smooth models \((O,C,E)\) and \((O',C',E')\),
  and then we say that a map \mbox{\(f:O\rightarrow O'\)} is {\bf
  sc-smooth}\index{sc-smooth map} provided \(f\circ r:U\rightarrow E'\) is
  sc-smooth for a suitable \(r:U\rightarrow U\) with \(r(U)=O\).
Next we note that it is not difficult to verify that this definition
  does not depend on the choice of sc-smooth retraction \(r\).
Furthermore, one can easily verify that \(Tr(TU)=Ts(TV)\) so that one
  can define the {\bf tangent}\index{tangent of a sc-smooth model}
  \(T(O,C,E)\) of \((O,C,E)\) by
  \begin{equation*}                                                       
    T(O,C,E) := (TO,TC,TE):= (Tr(TU),C^1\oplus E,E^1\oplus E).
    \end{equation*}
From this we see that neither the notion of sc-differentiability of
  maps between models nor the notion of a tangent depends on the choice
  of retraction which defines a given sc-smooth model.
Consequently, one can use these sc-smooth models as the local models
  for a generalized differential geometry.

Given a topological space \(X\) and a point \(x\in X\) a {\bf M-polyfold
  chart}\index{M-polyfold chart} around \(x\) is given by a tuple
  \((U,\phi,(O,C,E))\), where \((O,C,E)\) is an sc-smooth model, \(U\) an
  open neighborhood of \(x\) and \( \phi:U\rightarrow O\) is a
  homeomorphism.
We say that two charts \((U,\phi,(O,C,E))\) and \((U',\phi',(O',C',E'))\)
  are {\bf sc-smoothly compatible}\index{sc-smooth compatibility} provided
  \begin{equation*}                                                       
    \phi'\circ \phi^{-1}:\phi(U\cap U')\rightarrow \phi'(U\cap U') 
    \end{equation*}
  is sc-smooth, and
  \begin{equation*}                                                       
    \phi\circ \phi'^{-1}:\phi'(U\cap U')\rightarrow \phi(U\cap U')
    \end{equation*}
  is sc-smooth. 
We note that if \(Q\subset O\) is an open subset then also \((Q,C,E)\) is
  an sc-smooth model, and hence \((\phi(U\cap U'),C,E)\) and \((\phi'(U\cap
  U'),C',E')\) are sc-smooth models.

Analogous to the classical differential geometry, we define an {\bf
  sc-smooth atlas}\index{sc-smooth atlas} \({\mathcal A}\) for a Hausdorff
  paracompact topological space \(X\) to consist of a set of sc-smoothly
  compatible charts for which the domains cover \(X\).
Two atlases are said to be {\bf equivalent}\index{equivalent atlases} if
  their union is a sc-smooth atlas.

\begin{definition}[M-polyfold]
  \label{DEF_m_polyfold}
  \hfill\\
A M-polyfold\index{M-polyfold} is a Hausdorff paracompact topological
  space \(X\) equipped with an equivalence class of sc-smooth atlases.
\end{definition}
%

\begin{remark}\hfill\\
It is straightforward to show that an M-polyfold has an underlying
  metrizable topology; see \cite{HWZ2017}.
\end{remark}
%
With the above definitions established, one can develop a
  generalized differential geometry in which M-polyfolds are
  generalizations of manifolds with boundary and corners; see
  \cite{HWZ2017}.
This definition is too broad however, as it allows for M-polyfolds with a
  boundary which is too badly behaved for our applications.
For example, even locally the boundary of an M-polyfold (defined
  as above) does not inherit an M-polyfold structure from the ambient
  M-polyfold, which as we shall see later is an important property for the
  examples we have in mind.
Consequently, one must impose additional assumptions on retractions
  defined on relatively open subsets in partial quadrants. 
In generality this will be dealt with by introducing the notion of tame
  retractions and the associated class of M-polyfolds, see \cite{HWZ2017}.
If \(C\subset E\) is a partial quadrant in a sc-Banach space \(E\) and
  \(x\in C\)  we can consider closed linear subspaces \(F\) of \(E\) so that
  \(x + B_{\varepsilon}^F(0)\subset C\) for a suitable \(\varepsilon>0\).
One can show that there is a maximal closed subspace of this kind which is
  denoted by \(E_x\)\index{\(E_x\)}.
For example, if \(x\) is an interior point of \(X\) we have that \(E_x=E\).

\begin{definition}[$E_x$]
  \label{E-xxxx-EE}\label{DEF_Ex}
  \hfill\\
Consider the sc-Banach space \(E:={\mathbb R}^n\oplus W\) with partial
  quadrant \(C=[0,\infty)^n \oplus W\).
For each \(x\in C\) we define the subset \((x)\subset \{1,....,n\}\) to consist
  of all indices \(i\) such that \(x_i>0\).
Then we denote by \({\mathbb R}^{(x)}\) the linear subspace of \({\mathbb
  R}^n\) consisting of vectors \(h\) with \(h_i=0\) for \(i\in
  \{1,\ldots,n\}\setminus(x)\) and define
  \begin{align*}                                                          
    E_x={\mathbb R}^{(x)}\oplus W. 
    \end{align*}
\end{definition}
%

Next we are in the position to define a sc-smooth retract as well as a
  sc-smooth retraction.

\begin{definition}[tame sc-retraction]
  \label{DEF_tame_sc_retraction}
  \hfill\\
Let  \(r :U \rightarrow U\) be a sc-smooth retraction defined on a
  relatively open subset  \(U\) of a partial quadrant \(C\) in the sc-Banach
  space \(E\). 
The sc-smooth retraction \(r\)  is called a {\bf tame
  sc-retraction}\index{tame sc-retraction}, if the following two conditions
  are satisfied.
\begin{itemize}
  \item[(1)]  
  \(d_C(r(x)) = d_C(x)\) for all \(x \in U\).
  \item[(2)]  
  At every smooth point \( x\in  O:= r(U)\), there exists a sc-subspace
  \(A\subset E\), such that  \(E=T_0O \oplus A\) and \(A\subset E_x\).
  \end{itemize}
A sc-smooth retract \((O,C,E)\) is called a {\bf tame sc-smooth
  retract}\index{tame sc-smooth retract}, if \(O\) is the image of a
  sc-smooth tame retraction.
\end{definition}
%

A {\bf tame M-polyfold}\index{tame M-polyfold} is a metrizable space
  equipped with a sc-smooth atlas consisting of charts using tame models
  \((O,C,E)\).
 
The degeneracy index, \(d_C\), defined above on partial quadrants,
  has an analogue on M-polyfolds.  
A precise notion is provided by Definition 2.13 in Section 2.3 of
  \cite{HWZ2017}, however the essential idea is as follows.
If \(X\) is an M-polyfold, then \(X\) is locally modeled on sc-smooth
  retracts.
Consequently, for each fixed \(x\in X\), there exists a class
  \(\mathcal{L}\) of local models of the form \(\Phi : U\to  O\subset C
  \subset E \), where  \((O, C, E)\) is an sc-smooth retract, \(U\subset X\)
  is open and contains \(x\), and \(\Phi\) is an sc-smooth diffeomorphism.
We then define the degeneracy index for M-polyfolds by
  \begin{align*}                                                          
    d_X(x) := \min_{(\Phi, U, (O,C,E))\in \mathcal{L}} d_C(\Phi(x)).
    \end{align*}
We see immediately that degeneracy index\index{degeneracy index
  \(d_X\)} is a well defined map \(d_X:X\rightarrow {\mathbb N}\) which
  again measures the degree of the corner.
The following two results are will be important, and proofs can
  be found in \cite{HWZ2017}.

\begin{proposition}[sc-diffeomorphisms preserve degeneracy index]
  \label{PROP_sc_diffeo_preserves_degen_ind}
  \hfill\\
Let \(X\) and \(Y\) be M-polyfolds, and let \(U\subset X\) and \(V\subset
  Y\) be open subsets.
If \(f:U\to V \) is an sc-diffeomorphism and \(x\in U\), then 
  \begin{align*}                                                          
    d_X(x) = d_Y \big(f(x)\big).
    \end{align*}
\end{proposition}
%
\begin{proof}
This is a restatement of Proposition 2.7 in Section 2.3 of \cite{HWZ2017}.
\end{proof}
%

\begin{proposition}[equality of degeneracy indices]
  \label{HOFER11X}\label{PROP_eq_of_degen_index}
  \hfill\\
Consider the tame sc-smooth retract \((O,C,E)\) and view \(O\) as an
  abstract M-polyfold with degeneracy index \(d_O\). 
Then the following identity holds.
\begin{align*}                                                            
  d_C(x)=d_O(x)\ \ \text{for all}\ x\in O. 
  \end{align*}
\end{proposition}
%

With the degeneracy index established, we can then define the {\bf boundary
  points}\index{boundary points} of \(X\) precisely as set of the points
  \(x\in X\) with \(d_X(x)>0\).
Finally, we note that a partial quadrant \(C\) can be viewed as an
  M-polyfold and that the previously defined \(d_C\) coincides with the
  degeneracy index of \(C\) viewed as a M-polyfold, however this is a
  nontrivial result, even if easier than Proposition
  \ref{PROP_eq_of_degen_index}.
For details, we refer the reader to \cite{HWZ2017}. 
Another useful result proved in \cite{HWZ2017} is Proposition
  \ref{PROP_consequences_of_tame_m_polyfolds} below.

\begin{proposition}[consequences of tame M-polyfolds]
  \label{PROP_consequences_of_tame_m_polyfolds}
  \hfill\\
The following results hold.
\begin{itemize}
  \item[(1)]  
  Let \(X\) and \(Y\) be two tame M-polyfolds.  Then \(X\times Y\) is a tame
  M-polyfold and \(d_{X\times Y}(x,y)=d_X(x)+d_Y(y)\).
  \item[(2)]  
  If \(X\) is a ssc-manifold\footnote{The notion of an ssc-manifold is
  provided in Definition \ref{DEF_ssc_manifold_atlas} below.}  with
  boundary with corners then the underlying M-polyfold is tame.
  \end{itemize}
\end{proposition}
%

Given an M-polyfold \(X\), we then have a filtration \(X_{i+1}\subset
  X_i\subset\ldots\subset X_0=X\), and one can show that \(X_1\) has an
  M-polyfold structure induced from \(X\).
We denote this M-polyfold by \(X^1\), and we can define \(X^j\)
  similarly.\index{\(X^1\) and \(X^j\)}
A notion which will be important to us is that of a sub-M-polyfold,
  given below.
\begin{definition}[sub M-polyfold]
  \label{def:sub M polyfold}\label{DEF_sub_m_polyfold}
  \hfill\\
Let \(X\) be a M-polyfold and \(A\) a subset. 
We say \(A\) is a \emph{sub-M-polyfold} \index{sub-M-polyfold}provided that
  for every point \(a\in A\) there exists an open neighborhood
  \(U=U(a)\subset X\) and an sc-smooth map \(r:U\rightarrow U\) such that
  \(r\circ r=r\) and \(r(U)=U\cap A\).
\end{definition}
%
It is elementary to prove the following lemma; see \cite{HWZ2017}.
\begin{lemma}[structures sub-M-polyfolds inherit]
  \label{lem:sub poly}\label{LEM_structures_sub_poly_inherit}
  \hfill\\
A sub-M-polyfold \(A\) of \(X\) inherits a natural M-polyfold structure
  for which the inclusion \(j:A\rightarrow X\) is sc-smooth, and the local
  maps \(U\rightarrow A\) given by  \( u\rightarrow j^{-1}\circ r(u)\) are
  sc-smooth as well, where \(A\) is equipped with this M-polyfold structure.
\end{lemma}
%

This result has the following corollary, which can be viewed as the
  starting point for quite far-reaching generalizations, which are of utmost
  importance later on.
For example the \(\oplus\)-method introduced in Section
  \ref{SEC_imprinting_method} can be viewed as having its roots in this
  corollary.

\begin{corollary}[induced sub-M-polyfold structures]
  \label{corr18}\label{COR_induced_sub_m_poly}
  \hfill\\
Assume \(X\) and \(Y\) are M-polyfolds and \(g:Y\rightarrow X\) is an
  sc-smooth map so that there exists an open neighborhood \(U\) of \(g(Y)\)
  and an sc-smooth map \(f:U\rightarrow Y\) satisfying \(f\circ g = Id_Y\).
Then \(A:=g(Y)\) is a sub-M-polyfold of \(X\) and for the induced structure
  the map \(g:Y\rightarrow A\) is an sc-diffeomorphism.
\end{corollary}
%
\begin{proof}
Define \(r:U\rightarrow U\) by \(r(u)= g\circ f(u)\). This is an sc-smooth
  map and \(r\circ r = g\circ f\circ g \circ f = g\circ Id_Y\circ f= r\).
Using that \(f\) is surjective it follows that  \(r(U) = g(Y)=A\).
Hence \(g(Y)\) is a sub-M-polyfold. 
Moreover \(g^{-1}\circ r:U\rightarrow Y\) is precisely the sc-smooth
  map \(f\).
If \(a\in A=g(Y)\) and \(s:V\rightarrow V\) a local sc-smooth retraction with
  \(s(V)=A\cap V\) we observe that on \(V\cap U\) it holds that \(g^{-1}\circ s
  = g^{-1}\circ r\circ s\) which shows that \(g^{-1}\circ s\) is sc-smooth.
Hence  \(g^{-1}\) is sc-smooth for the induced structure on \(A\).
\end{proof}

\begin{proposition}[degeneracy index inequality of sub-M-polyfolds]
  \label{PROP_degen_ind_ineq_sub_m_poly}
  \hfill\\
If \(X\) is an M-polyfold and \(A\subset X\) is a sub-M-polyfold of \(X\),
  then
  \begin{align*}                                                          
    d_A(a) \leq d_X(a)
    \end{align*}
  for all \(a\in A\).
\end{proposition}
%
\begin{proof}
This is a restatement of Lemma 2.3 in Section 2.3 of \cite{HWZ2017}.
\end{proof}

A M-polyfold \(X\) sometimes admits sub-M-polyfolds for which their
  induced M-polyfold structure admits an equivalent finite-dimensional
  manifold atlas.

\begin{definition}[finite dimensional submanifold]
  \label{def:fin dim submanifold}\label{DEF_fin_dim_submanifold}
  \hfill\\
A subset \(A\) of a  M-polyfold \(X\) is called a \emph{finite-dimensional
  submanifold} \index{finite-dimensional submanifold} provided for every
  point \(a\in A\) there exists an open neighborhood \(U=U(a)\) in \(X\) and
  an sc-smooth retraction \(r:U\rightarrow U\) having the following
  properties:
\begin{itemize}
  \item[(1)] 
    \(r(U)=U\cap A\).
  \item[(2)] 
    \(r:U\rightarrow U^1\) is well-defined, i.e. \(r(U_i)\subset U_{i+1}\),
    and sc-smooth.
  \end{itemize}
\end{definition}
%
We note that (2) implies that \(A\subset X_\infty\) so that \(A\) has at a
  every point a tangent space. The retraction satisfying (1) and (2) is
  called a {\bf sc\(^+\)-retraction}\index{sc\(^+\)-retraction}.
The following holds, see \cite{HWZ2017}.

\begin{lemma}[submanifolds inherit manifold structure]
  \label{lem:submanifold}\label{LEM_submanifold_inherit}
  \hfill\\
Let \(A\) be a finite-dimensional submanifold of an M-polyfold \(X\). 
Then for each point \(a\in A\) which satisfies \(d_X(a)=0\), there exists
  a neighborhood \(V\subset X\) such that \(V\cap A\) is equipped with
  a natural smooth manifold structure.
Moreover every point  \(a\in A\) has a well-defined  tangent
  space.
\end{lemma}
%

In light of Lemma \ref{LEM_submanifold_inherit}, it is natural to ask about
  points \(a\) in a finite dimensional submanifold \(A\) of an M-polyfold
  \(X\) for which \(d_X(a)>0\).
In this case, whether or not a neighborhood of such an \(a\in A\)
  locally has the structure of a smooth manifold with boundary and corners
  depends on the position of \(T_aA\) with respect  to \(\partial X\), and
  the properties of \(\partial X\).
This question has been studied in \cite{HWZ2017}. 
Of related importance, is the discussion of \emph{tame} M-polyfolds, laid
  out in depth in \cite{HWZ2017}, which is the relevant generalization
  of a manifold with boundary with corners to an M-polyfold with boundary
  and corners.

Here we will provide the following important example of a tame local model
  which is featured prominently in applications.
Define \(\wt{C}=[0,\infty)^k\times {\mathbb R}^{n-k}\), and let \(E\)
  be an sc-Banach space.
Consider a family \(a\mapsto \pi_a\) of linear sc-projections
  parameterized by \(a\in V\), where \(V\subset \wt{C}\) is open, so that
  the map \(\pi_{(\cdot)} :V\times E\rightarrow E\) given by
  \((a,e)\rightarrow \pi_a(e)\) is sc-smooth.
Then an important type of sc-smooth retraction \(r:V\times E\rightarrow
  V\times E\) is given by
  \begin{equation*}                                                       
    r(a,e)=(a,\pi_a(e)),
    \end{equation*}
  where the local sc-model is given by the tuple \mbox{\((O,\wt{C}\times
  E,{\mathbb R}^n\times E)\)}, where \(r(V\times E)=O\), and
  of course \((O',\wt{C}\times E,{\mathbb R}^n\times E)\) for open subsets
  \(O'\) of \(O\). 
These specific retractions are called {\bf splicings}\index{splicings} and
  were introduced in \cite{HWZ2}.
As mentioned previously, the tuples \((O',\wt{C}\times E,{\mathbb
  R}^n\times E)\) are examples of tame local models.
Furthermore, if \((a,e)\in O'\) and \(a=(a_1,..,a_n)\), then the degeneracy
  index \(d_{O'}(a,e)\) is precisely the number of indices \(i\) such that
  \(a_i=0\).
This is not true in general for an sc-smooth model \((O,C,E)\), however
  tameness is a key condition that guarantees that \(d_O=d_C|O\).   
  We leave it to the reader to consult the general theory of tame local
  models in \cite{HWZ2017} for additional details and we just point out that
  M-polyfolds built on splicings have nice boundaries which can be stratified
  by M-polyfolds. 

We take a moment to discuss several classes of sc-smooth objects.
If \(E\) is an sc-Banach space, and \(C\subset E\) is a partial quadrant,
  and \(U\subset C\) is a relatively open subset, then we can define an
  sc-manifold (with boundary and corners) to be a metrizable topological
  space with charts which are homeomorphisms \(\psi:V\rightarrow U\),
  for which compatibility and an atlas is defined as before.
Hence we can have the following kind of charts:
\begin{itemize}
  \item[(1)]  
    {\bf M-polyfold charts} \((V, \psi,(O,C,E))\), where we require \(O\)
    to be an sc-smooth retract.
      \item[(2)] 
    {\bf Tame M-polyfold charts} \((V, \psi,(O,C,E))\), where we require
    in addition that \((O,C,E)\) is an sc-smooth, tame retract (e.g. the
    image of a splicing).
  \item[(3)] 
    {\bf Sc-Manifold charts} \((V, \psi,(O,C,E))\), where \(O\subset C\)
    is open.
  \end{itemize}
The resulting atlases respectively define M-polyfolds, tame M-polyfolds,
  and sc-manifolds, provided we require the transition maps to be
  sc-smooth.
When considering sc-manifolds, it may occasionally happen that we
  find a compatible sc-manifold atlas for which the transition maps are
  level-wise \emph{classically} smooth.
We shall call this an ssc-manifold structure; here the additional `s'
  stands for strong.
\begin{definition}[ssc-manifold atlas and manifold]
  \label{def:ssc-manifold atlas}\label{DEF_ssc_manifold_atlas}
  \hfill\\
An {atlas} \({\mathcal A}\), for the metrizable space \(X\), consisting of
  sc-manifold charts which are level-wise classically smooth, is said to
  define a {\bf strong sc-manifold structure}\index{strong sc-manifold
  structure} on \(X\).
We shall call it a {\bf ssc-manifold atlas}\index{ssc-atlas}.
A topological space equipped with an ssc-manifold atlas is an
  ssc-manifold.\index{ssc-manifold}
\end{definition}
%
We observe that a finite-dimensional ssc-manifold is the same as a
  finite-dimensional smooth manifold possibly with boundary with corners.
Note that there is a forgetful chain of properties:
\begin{eqnarray*}                                                         
 & \text{\bf ssc-manifold}  \Rightarrow \text{\bf sc-manifold}
  \Rightarrow \text{\bf tame M-polyfold}&\\
 &  \Rightarrow \text{\bf M-polyfold}.&
  \end{eqnarray*}
Given that M-polyfolds are locally modelled on the image of an sc-smooth
  retraction, it is natural to ask what sort of structure the image
  of a classically smooth retraction has, or, more importantly for
  our applications, what sort of structure the image of an ssc-smooth
  retraction has.
As such, we present the following result, which can be regarded as a
  version of Cartan's last mathematical theorem, \cite{Cartan}.
The proof is left to the reader.

\begin{proposition}[ssc retracts yield ssc-submanifolds]
  \label{PROPD1.13}\label{PROP_ssc_retracts}
  \hfill\\
Let \(E\) be an sc-Banach space and \(U\) an open subset. 
Assume that \(r:U\rightarrow U\) is ssc-smooth and \(r\circ r=r\).
Then \(O= r(U)\) is an ssc-submanifold of \(E\),  i.e. every point \(o\in O\)
  has an open neighborhood which is ssc-diffeomorphic to a product.
\end{proposition}
%

\begin{exercise}
  \label{EX_11}
Prove Proposition \ref{PROPD1.13}.
\end{exercise}
%

%
\subsection{Strong Bundles}\label{qsec1.3}
The above constructions can be adapted to define a generalized notion
  of a vector bundle over an M-polyfold, however we will be specifically
  interested in the class of so-called strong bundles.
Given two sc-Banach spaces \(E\) and \(F\), we can define \(E\triangleleft
  F\)\index{\(E\triangleleft F\)} to be the space \(E\times F\) equipped
  with the double filtration
  \begin{equation*}                                                       
    (E\triangleleft F)_{m,k}:= E_m\times F_k \qquad \text{for}\qquad
    0\leq k\leq m+1.
    \end{equation*}
 By forgetting some of the structure we can consider \(E\oplus F\)
  with the diagonal filtration and \(E\oplus F^1\) also with the diagonal
  filtration, where we note that \((E\triangleleft F)_{m,m}=(E\oplus F)_m\)
  and \((E\triangleleft F)_{m,m+1}=(E\oplus F^1)_m\).

We shall require that maps \(\Phi:E\triangleleft F\rightarrow
  G\triangleleft H\) respect the double filtration and will be linear on
  the second factor; in other words for
  \begin{equation*}                                                       
    \Phi(u,h)=(f(u),\phi(u)(h)),
    \end{equation*}
  we require that \(\phi(u):F\rightarrow H\) is linear.
In particular, such a map induces maps \(E_m\oplus F_{m+i}\rightarrow
  G_m\oplus H_{m+i}\) for each \(m\in \mathbb{N}\) and \(i\in \{0, 1\}\);
  or more succinctly, it induces sc-continuous maps \(E\oplus F^i\rightarrow
  G\oplus H^i\) for \(i\in \{0, 1\}\) provided the level-wise maps are all continuous.
\begin{remark}\hfill\\
Observe that a map \(\Phi:E\oplus F\rightarrow G\oplus H\) which
  sends \(E\oplus F^1\) to \(G\oplus H^1\)  will be a map \(E\triangleleft
  F\rightarrow G\triangleleft H\); in other words, if \(\Phi\) preserves
  the above two diagonal filtrations, then it also preserves the above
  double filtration.
Indeed, to prove this, first assume \((u,h)\in E_m\oplus F_k\) and \(0\leq
  k\leq m\); we will discuss the \(k=m+1\) case momentarily.
Regard \((u,h)\) as an element in \(E_k\oplus F_k\).
Hence \(\Phi(u,h)\in G_k\oplus H_k\) and therefore \(\phi(u)(h)\in H_k\).
On the other hand \((u,0)\) is in \(E_m\oplus F_{m}\) implying \(\Phi(u,0)\in
  G_m\oplus H_m\), and therefore \(f(u)\in G_m\).
This implies \((f(u),\phi(u)h)\in G_m\oplus F_k\).
If \(k=m+1\) we can draw a similar conclusion by using \(E\oplus F^1\).

\end{remark}

We shall call a map \(\Phi:E\triangleleft F\rightarrow G\triangleleft
  H\) {\bf sc\(_\triangleleft\)-continuous}\index{sc\(_\triangleleft\)
  -continuity} if the induced maps \(E\oplus F^i\rightarrow G\triangleleft
  H^i\) are sc\(^0\) for \(i=\{0, 1\}\).
Similarly we call such a map {\bf
  sc\(_\triangleleft\)-smooth}\index{sc\(_\triangleleft\)-smoothness}
  provided the induced maps \(E\oplus F^i\rightarrow G\oplus H^i\) are
  sc-smooth.
The ideas can be immediately extended to the situation where \(C\subset
  E\) is a partial quadrant and \(U\subset C\) is relatively open, so that
  we can consider \(\Phi:U\triangleleft F\rightarrow V\triangleleft H\).

Of particular interest are the sc\(_\triangleleft\)-smooth maps
  \(R:U\triangleleft F\rightarrow U\triangleleft F\) satisfying \(R\circ R=R\),
  where \(U\subset C\) is open.
Observe that such maps cover an sc-smooth retraction \(r:U\rightarrow U\),
  so that by taking \(K=R(U\triangleleft F)\) and \(O=r(U)\) we obtain
  \begin{equation*}                                                       
    p:K\rightarrow O.
    \end{equation*}
This is the {\bf local model for a strong bundle}\index{local strong
  bundle model} and we refer the reader to \cite{HWZ2017} for more
  details.
Note that the fibers \(p^{-1}(o)\) are Banach spaces, but not necessarily
  sc-Banach spaces.
Indeed, a fiber only has the structure of an sc-Banach space if \(o\in
  O_\infty\).
On the other hand, if \(o\in O_m\) then the fiber \(K_o=p^{-1}(o)\) only
  has a well-defined finite grading \(k=0,\ldots,m+1\).

Having defined local strong bundles, which we denote by \(p:K\rightarrow
  O\), we can define strong bundles \(P:W\rightarrow X\).
Here \(X\) and \(W\) are paracompact Hausdorff spaces, where \(P\) is a
  surjective map and the fibers are equipped with vector space structures.
We can define strong bundle charts and strong bundle atlases in the
  obvious way and leave the details to the reader.
Precise definitions can be found in \cite{HWZ2017}.

Observe that a strong bundle \(P:W\rightarrow X\) has a double filtration
  of \(W\) and underlying M-polyfolds \(W[0]\) and \(W[1]\) where the
  filtration is given by \(W[i]_m:=W_{m,m+i}\) for \(i\in \{0, 1\}\).
An {\bf sc-smooth  section}\index{sc-smooth section} \(s\) of \(P\) is a
  map \(s:X\rightarrow W\) for which \(P\circ s=Id_X\) and
  \(s:X\rightarrow W[0]\) is sc-smooth.
We denote the vector space of sc-smooth sections by \(\Gamma(P)\).
A {\bf sc\(^+\)-section}\index{sc\(^+\)-section} \(s\) of \(P\) is an
  sc-smooth section of \(P\) which is also sc-smooth as a map
  \(s:X\rightarrow W[1]\).
The vector space of sc\(^+\)-sections is denoted by \(\Gamma^+(P)\).

We note that we can have various types of strong bundles related to
  the forgetful chain of properties previously exhibited. 
We leave the details to the reader. 
In our applications we shall use strong bundles over tame M-polyfolds
  and strong ssc-bundles over ssc-manifolds.

%
\subsection{Submersion Property}\label{APPA3}\label{SEC_submsersion_property}
The following result is very useful in concrete constructions which
  frequently involve fibered products.

\begin{definition}[submersion property]
  \label{DEFNXA.12}\label{DEF_submersion_property}
  \hfill\\
Assume that \(p:X\rightarrow Y\) is an sc-smooth map between M-polyfolds.
We say that \(p\) has the \emph{submersion property}\index{submersion
  property} provided that \(p\) is surjective and the following holds for
  the graph of \(p\) denoted by \(\text{Gr}(p)=\{(x,p(x))\ |\ x\in X\}\).
For \((x_0,p(x_0))\in \text{Gr}(p)\) there exists an open neighborhood
  \(U=U(x_0,p(x_0))\) in \(X\times Y\) and an sc-smooth map
  \(\rho:U\rightarrow U\) of the form \(\rho(x,y)=(\bar{\rho}(x,y),y)\) and
  \(\rho\circ\rho=\rho\) such that \(\rho(U)  = U\cap \text{Gr}(p)\).
\end{definition}
%

\begin{remark}
  \label{REM_submersion_retract}
  \hfill\\
Because \(\rho\circ \rho = \rho\) and \(\rho(x,y) = (\bar{\rho}(x,y),
  y)\), we have
  \begin{align*}                                                          
    \bar{\rho}(x,y) = \bar{\rho}\big(\bar{\rho}(x,y), y\big) \in X
    \end{align*}
Because \(\rho\) retracts onto \({\rm Gr}(p)\) and \(\rho(x,y) =
  \big(\bar{\rho}(x,y), y\big)\), we must have
  \begin{align*}                                                          
    \bar{\rho}(x,y) \in p^{-1}(y),
    \end{align*}
  and of course \(\bar{\rho}(x, p(x)) = x\).
\end{remark}
%

The obvious example of a submersive sc-smooth map is the projection onto a
  factor in  a product.

\begin{proposition}[natural projection of product is submersive]
  \label{PROP_natural_projection_submersive}
  \hfill\\
Let \(X\) and \(Y\) be M-polyfolds and \(p:X\times Y\rightarrow X\) the
  projection onto the first factor.
Then \(p\) is submersive.
\end{proposition}
%
\begin{proof}
Define \(\rho: (X\times Y) \times X\to (X\times Y)\times X\) by 
  \begin{align*}                                                          
    \rho\big((x,y), x' \big) = \big(\bar{\rho}((x,y), x'), x'\big)
    \end{align*}
  where
  \begin{align*}                                                          
    \bar{\rho}((x,y), x')  = (x', y).
    \end{align*}
In other words, \(\rho((x,y), x') = ((x', y), x')\).
Then \(\rho\) is a (global) sc-smooth retraction and \(\rho((X\times
  Y)\times X) = \text{Gr}(p)\).
\end{proof}

The following result is easy to prove and is left as an exercise.

\begin{lemma}[strong bundles are submersive]
  \label{LEMM1.16}\label{LEM_strong_bundles_submersive}
  \hfill\\
The projection to the base in a strong bundle always have the submersion
  property.
\end{lemma}
%

For the notion of bundle see \cite{HWZ2017}. 
These spaces are similar to strong bundles and have only the diagonal
  filtration.\\

\begin{exercise}\hfill\\
Prove Lemma \ref{LEM_strong_bundles_submersive}, and extend the result to
  the case that the strong bundle is in fact an ssc-bundle.
\end{exercise}
%

An immediate consequence of the definition is given by the following
  straight forward proposition.

\begin{proposition}[submersions and implied diffeomorphisms]
  \label{FirstCOnc}\label{PROP_submersions_implied_diffeo}
  \hfill\\ 
Assume that \(f:X\rightarrow Y\) is a sc-smooth submersive map between
  M-polyfolds.
Given a smooth point \(x_0\in X\) there exists an open neighborhood
  \(V_{y_0}\subset Y\), \(y_0=f(x_0)\), a sub-M-polyfold
  \(\Sigma_{x_0}\subset X\) containing \(x_0\), and  a sc-diffeomorphism
  \begin{align*}                                                          
    \psi:V_{y_0}\rightarrow \Sigma_{x_0}
    \end{align*}
  such that 
\begin{enumerate}                                                         
  \item 
  \(f\circ \psi: V_{y_0}\rightarrow Y\) is the identity map on \(V_{y_0}\)
  \item 
  \(\psi\circ f :f^{-1}(V_{y_0})\rightarrow Y\) has its image in
  \(f^{-1}(V_{y_0})\)
  \item 
  \(\psi\circ f :f^{-1}(V_{y_0})\rightarrow Y\) is a sc-smooth retraction
  defining \(\Sigma_{x_0}\).
  \end{enumerate}
\end{proposition}
%
\begin{proof}
The submersion property guarantees the existence of an open neighborhood
  \(W\) of \((x_0,y_0)\in X\times Y\) and a sc-smooth retraction
  \(\rho:W\rightarrow W\) of the form \(\rho(x,y)=(\bar{\rho}(x,y),y)\)
  satisfying \(\rho(W)=W\cap \text{Gr}(f)\).
Pick an open neighborhood \(V_{y_0}\subset Y\) satisfying \(\{x_0\}\times
  V_{y_0}\subset W\).
Then we define the sc-smooth 
  \begin{align}
    &\psi: V_{y_0}\rightarrow X\\
    &\psi(y)=\bar{\rho}(x_0,y).
    \end{align}
Since \(\rho\) retracts onto the graph it follows that \(f\circ\psi(y)=y\)
  for \(y\in V_{y_0}\).
Define \(U_{x_0}=f^{-1}(V_{y_0})\) and consider \(\psi\circ f:
  U_{x_0}\rightarrow X\).
Let us first show that the image is contained in \(U_{x_0}\). Define
  \(x_1=\psi\circ f(x)\) for \(x\in U_{x_0}\).
Then \(x_1=\bar{\rho}(x_0,f(x))\) and by the properties of \(\rho\) it
  follows that \(f(x_1)=f(x)\). Hence \(x_1\in f^{-1}(V_{y_0})=U_{x_0}\).
Define \(\tau:U_{x_0}\rightarrow U_{x_0}\) by \(\tau(x)=\psi\circ f(x)\).  
Then \(\tau\circ \tau=\tau\) and 
  \begin{align}
    \tau(U_{x_0})= \psi(V_{y_0}).
    \end{align}
We define \(\Sigma_{x_0}=\psi(V_{y_0})\) and by the previous discussion
  this is a sub-M-polyfold.
This completes the proof.
\end{proof}

The next proposition studies pull-back diagrams involving submersive
  diagrams.

\begin{proposition}[fibered products: projections and submersions]
  \label{PROP_fibered_products_projection_submersions}
  \hfill\\
Suppose \(p:X\rightarrow Y\) has the submersion property and
  \(f:Z\rightarrow Y\) is an sc-smooth map.
Then the  fibered product \index{fibered product} 
\begin{align*}                                                            
  X{_{p}\times_f}Z= \{(x,z)\in X\times Z\ |\ p(x)=f(z)\}
  \end{align*}
  is a sub-M-polyfold of \(X\times Z\) and
  \(X{_{p}\times_f}Z\xrightarrow{{\rm pr}_2} Z\) has the submersion
  property.
\end{proposition}
%
\begin{proof}
Let \((x_0,z_0)\) with \(p(x_0)=f(z_0)\). 
Because \(p\) has the submersion property, there exists an open
  neighborhood \(U\) of \((x_0,p(x_0))\) in \(X\times Y\) and an sc-smooth
  map \(\rho:U\rightarrow U\) with the form
  \(\rho(x,y)=(\bar{\rho}(x,y),y)\), with \(\rho\circ\rho=\rho\) and
  \(\rho(U)=\text{Gr}(p)\cap U\).
Define an open neighborhood \(V=V(x_0,z_0)\) in \(X\times Z\) by
  \begin{align}
    V=\{(x,z)\in X\times Z\ |\ (x,f(z))\in U\}.
    \end{align}
We define an sc-smooth map 
  \begin{align*}                                                          
    &\sigma:V\rightarrow X\times Z\\
    &\sigma(x,z)=\big(\bar{\rho}(x,f(z)),z\big)
    \end{align*}
First we want to show that \(\sigma:V\to V\). 
To that end, observe that if \((x,z)\in V\), then \((x, f(z))\in U\).
Also observe that 
  \begin{align*}                                                          
    \big(\bar{\rho}(x, f(z)), f(z)\big)=\rho(x, f(z)\big)\in U
    \end{align*}
  which guarantees that \(\sigma(x, z)=\big(\bar{\rho}(x, f(z)), x\big)
  \in V\), and thus indeed we have \(\sigma:V\to V\).
To see that \(\sigma\circ\sigma = \sigma\), we compute
  \begin{align*}                                                          
    \sigma\big(\sigma(x,z)\big) &= \sigma\big(\bar{\rho}(x, f(z)), z\big)
    \\
    &=\big(\bar{\rho}\big(\bar{\rho}(x, f(z)), f(z)\big), z \big)
    \\
    &=\big(\bar{\rho}(x, f(z)), z\big)
    \\
    &=\sigma(x, z).
    \end{align*}
Next we want to show \(\sigma(V) = V\cap X{_p\times_f} Z\).
We begin by showing the containment \(\sigma(V)\subset V\cap X{_p\times_f}
  Z \).
To that end, let \((x, z)\in \sigma(V)\subset V\).
Then \(\sigma(x, z) = (x, z) = (\bar{\rho}(x, f(z)), z)\), and thus
  \(x=\bar{\rho}(x, f(z))\).
However, because \(\rho\circ \rho = \rho\), and \(\rho(x,y) =
  (\bar{\rho}(x, y), y)\), and \(\rho(U) = {\rm Gr}(p)\cap U\), it follows
  from Remark \ref{REM_submersion_retract}, that
  \begin{align}\label{EQ_rho_p}                                           
    \bar{\rho}(x, y) \in p^{-1}(y).
    \end{align}
Combining this with the fact that \(x=\bar{\rho}(x, f(z))\), we have
  \(x\in p^{-1}(f(z))\), or in other words \(p(x)=f(z)\), so that
  \((x,z)\in X_{p\times_f} Z\), and hence \(\sigma(V)\subset V\cap
  X{_p\times_f} Z\).

Next we aim to establish that \(V\cap X{_p\times_f}Z \subset \sigma(V)\).
To that end, we let \((x, z)\in V\cap X{_p\times_f} Z\).
From this, it follows that \(p(x)=f(z)\).
Recalling that \(\rho\) is a retraction onto \({\rm Gr}(p)\), we then have
  \begin{align*}                                                          
    (x, f(z)) &= (x, p(x)) = \rho(x, p(x)) = \rho(x, f(x))\\
    &=(\bar{\rho}(x, f(z)), f(z)), 
    \end{align*}
so that \(x=\bar{\rho}(x, f(z))\).
But then, \(\sigma(x, z) = \big(\bar{\rho}(x, f(z)), z\big) = (x, z)\) and
  hence \((x, z)\in \sigma(V)\), and thus \(V\cap X{_p\times_f} Z\subset
  \sigma(V)\).  We conclude that \(X{_p\times_f}Z\) is a sub-M-polyfold.

Note that \({\rm pr}_2: X\times Z\) is sc\(^\infty\), so that \({\rm
  pr}_2: X{_p\times_f} Z\to Z\) is sc\(^\infty\) as well.
Surjectivity of the latter follows from surjectivity of \(p:X \to Y\).

To complete the proof of Proposition
  \ref{PROP_fibered_products_projection_submersions}, all that remains is to
  prove that the map \({\rm pr}_2:X{_p\times_f} Z\to Z\) has the submersion
  property.
We have already established that \({\rm pr}_2: X{_p\times_f} Z\to Z\) is
  sc\(^\infty\) and surjective, and thus it remains to show that for each
  \(((x_0, z_0), z_0)\in {\rm Gr}({\rm pr}_2)\subset(X{_p\times_f} Z)\times
  Z\) there exists an open set \(W \subset (X{_p\times_f} Z)\times Z\) and
  an sc\(^\infty\) retraction \(\delta\) satisfying
  \begin{align*}                                                          
    \delta:W\to W \qquad\text{and}\qquad \delta(W) = {\rm Gr}({\rm pr}_2).
    \end{align*}
To that end, we define
  \begin{align*}                                                          
    W = \big\{ \big((x, z), z'\big) \in (X{_p\times_f} Z)\times Z: (x,
    f(z')) \in U \big\}
    \end{align*}
  and
  \begin{align*}                                                          
    &\delta:W \to (X{_p\times_f} Z)\times Z\\
    &\delta\big((x, z), z'\big) = \big( (\bar{\rho}(x, f(z')), z'), z').
    \end{align*}
To see that \(\delta\) is well defined, we need
\begin{align*}                                                            
  p\big(\bar{\rho}(x, f(z'))\big) = f(z').
  \end{align*}
However, again by equation (\ref{EQ_rho_p}) we have \(\bar{\rho}(x, f(z')
  )\in p^{-1}(f(z'))\) which shows that \(\delta\) is indeed well-defined.

Next we show \(\delta:W\to W\).  To that end, let \(((x, z), z'\big)\in
  W\).
But then \((x, f(z'))\in U\).
But this implies \(\rho(x, f(z'))\in U\) since \(\rho:U\to U\).
But this in turn implies
\begin{align*}                                                            
  \big(\bar{\rho}(x, f(z')), f(z')\big) \in U 
  \end{align*}
  by definition of \(\rho\).
Consequently, we have shown that the tuple given by
  \(\big(\big(\bar{\rho}(x, f(z')), z'\big), z'\big)\) then satisfies
  \begin{enumerate}                                                       
    \item \(p\big(\bar{\rho}(x, f(z'))\big)=f(z')\)
    \item \(\big(\bar{\rho}(x, f(z')), f(z')\big)\in U\)
    \end{enumerate}
  so that
  \begin{align*}                                                          
    \big(\big(\bar{\rho}(x, f(z')), z'\big), z'\big) \in W.
    \end{align*}
Thus we have shown that \(\delta:W\to W\).
Moreover
  \begin{align*}                                                          
    \big(\big(\bar{\rho}(x, f(z')), z'\big), z'\big) \in {\rm Gr}({\rm
    pr}_2),
    \end{align*}
  so that \(\delta: W\to W\cap {\rm Gr}({\rm pr}_2)\).

Note that \(\delta\circ \delta = \delta\) follows from a straightforward
  computation which makes use of the fact that
  \begin{align*}                                                          
    \bar{\rho}\big(\bar{\rho}(x, f(z')\big), f(z')\big) = \bar{\rho}(x,
    f(z')\big).
    \end{align*}
It then follows that \(\delta(W) = W\cap {\rm Gr}({\rm pr}_2)\).
Finally, we note that sc-smoothness of \(\delta\) is self-evident from its
  definition.
Thus we have shown that \(\delta:W\to W\) is an sc\(^\infty\) retraction
  for which \(\delta(W)= W\cap {\rm Gr}({\rm pr}_2).\)
This establishes the submersion property, and completes the proof of
  Proposition \ref{PROP_fibered_products_projection_submersions}.
\end{proof}

There are other consequences of the submersion property.

\begin{proposition}[fibers of submersions are sub-M-polyfolds]
  \label{PROP_fibers_submersions_are_sub_m_polyfolds}
  \hfill\\
Assume that \(f:X\rightarrow Y\) is a submersive sc-smooth map between
  M-polyfolds.
Then for every smooth point  \(y_0\) the fiber \(f^{-1}(y_0)\) is a
  sub-M-polyfold and consequently the fiber-wise degeneracy index
  \(d_{f^{-1}(y_0)}:f^{-1}(y_0)\rightarrow {\mathbb N}\) is well-defined for
  every \(y_0\in Y_\infty\).
\end{proposition}
%
\begin{proof}
Fix \(y_0\in Y_\infty\) and pick \(x_0\in f^{-1}(y_0)\).  
Because \(f\) is submersive, there exists an open neighborhood
  \(W=W(x_0,y_0)\) in \(X\times Y\) and a sc-smooth retraction
  \(\rho:W\rightarrow W\) of the form
  \(\rho(x,y)=(\bar{\rho}(x,y),y)\) satisfying
  \begin{align}
    \rho(W) = W\cap \text{Gr}(f).
    \end{align}
Denote by \(V=V(x_0)\subset X\) the open neighborhood consisting of all
  points \(x\) such that \((x,y_0)\in W\).
We define the sc-smooth map \(\tau:V\rightarrow X\) by
  \(\tau(x):=\bar{\rho}(x,y_0)\).
Observe that with \((x,y_0)\in W\) we also have \(\rho(x,y_0)\in W\) and
  \(\rho(x, y_0) = (\bar{\rho}(x,y_0), y_0)\) implying that
  \(\tau(V)\subset V\).
Since \(\rho(x,y_0) =(\tau(x),y_0)\) belongs to \(\text{Gr}(f)\) it
  follows that \(f(\tau(x))=y_0\).
Hence
  \begin{align}
    \tau(V)\subset V\cap f^{-1}(y_0).
    \end{align}
Moreover, if \(x\in V\cap f^{-1}(y_0)\) then \((x, y_0)\in W\) and
  \(f(x)=y_0\), by definition of \(V\).
Equivalently, \((x, y_0)\in W\cap {\rm Gr}(f)\).
But the retraction \(\rho\) satisfies \(\rho(W)=W \cap {\rm Gr}(f)\), so
  \(\rho(x, y_0) = (x, y_0)\).
But then we have
  \begin{align*}                                                          
    \rho(x, y_0) = \big(\bar{\rho}(x, y_0), y_0\big) = \big(\tau(x),
    y_0\big)
    \end{align*}
  so \(x=\tau(x)\), and thus \(x\in \tau(V)\).
Hence we have shown that \(\tau(V)=V\cap f^{-1}(y_0)\). 
This shows that the fibers \(f^{-1}(y_0)\) are sub-M-polyfolds, and
  therefore \(d_{f^{-1}(y)}\) is defined for every \(y\in Y\).
\end{proof}
In the case discussed in the previous proposition it is also possible to
  estimate the degeneracy indices.

\begin{proposition}[degeneracy inequality and submersions]
  \label{propller1}\label{PROP_degeneracy_inequality_submersions}
  \hfill\\
Assume that \(f:X\rightarrow Y\) is a submersive sc-smooth map between
  M-polyfolds.
Then we have the inequality \(d_X(x)\geq d_Y(f(x))\) for all smooth points
  \(x_0\in X\).
\end{proposition}
%
\begin{proof}
In view of Proposition \ref{PROP_submersions_implied_diffeo} we find a
  sc-diffeomorphism of the form \(\psi:U(y_0)\rightarrow \Sigma_{x_0}\),
  where \(\Sigma_{x_0}\subset X\) is a sub-M-polyfold.
We must have the identity
  \begin{align}
    d_{Y}(y)=d_{\Sigma_{x_0}}(\psi(y))\ \ \text{for}\ y\in U(y_0).
    \end{align}
It is trivially true that \(d_{X}(x)\geq d_{\Sigma_{x_0}}(x)\) for all
  \(x\in\Sigma_{x_0}\).
Hence 
  \begin{align}
    d_X(x_0)\geq d_{\Sigma_{x_0}}(x_0)=d_{Y}(y_0).
    \end{align}
Hence we have proved for all \(x\in X_\infty\) the inequality
  \(d_Y(f(x))\leq d_X(x)\).
\end{proof}

There is an obvious corollary which shall be useful later on.

\begin{corollary}[degeneracy inequality and submersions over manifolds]
  \label{COR_degen_inequality_submersions_manifolds}
  \hfill\\
Let \(X\) be a M-polyfold and \(V\) a smooth finite-dimensional manifold
  with boundary with corners and \(f:X\rightarrow V\) an sc-smooth
  submersive map.
Then we have for all \(x\in X\) the inequality \(d_X(x)\geq d_{V}(f(x))\).
\end{corollary}
%
\begin{proof}
We already know that the inequality holds for all \(x\in X_\infty\) in
  view of Proposition \ref{propller1}.
Take \(x_0\in X\) and define \(v_0=f(x_0)\). 
The latter is a smooth point.
Since \(f\) is submersive we find an sc-smooth retraction
  \(\rho:W\rightarrow W\), where \(W\) is an open neighborhood of
  \((x_0,v_0)\in X\times V\) such that \(\rho(W)= W\cap \text{Gr}(f)\) and
  \(\rho(x,v)=(\bar{\rho}(x,v),v)\).
Take a sequence of smooth points \(x_k\in X\) with \(x_k\rightarrow x_0\)
  and consider the sequence of smooth points \(\bar{\rho}(x_k,v_0)\).
Because 
  \begin{align*}                                                          
    \big(\bar{\rho}(x_k, v_0), v_0\big)=\rho(x_k, v_0) \in {\rm Gr}(f)
    \end{align*}
  we must have \(f(\bar{\rho}(x_k, v_0))=v_0\) and hence
  \(d_X( \bar{\rho}(x_k,v_0))\geq d_V(v_0)\).
It is straightforward and elementary to show that  there exists an open
  neighborhood \(U=U(x_0)\) such that
  \(d_X(x)\leq d_X(x_0)\) for all \(x\in U(x_0)\).
Consequently we infer for large \(k\)  that \(d_X(x_0)\geq d_X(
  \bar{\rho}(x_k,v_0))\geq d_V(v_0)\).
The proof is complete.
\end{proof}
%

%
\section{The Imprinting Method}\label{APP777}\label{SEC_imprinting_method}
We derive several results which are extremely useful for constructing
  M-polyfolds and which will be used throughout this text.

%
\subsection{Basic Results}\label{qsec2.1}\label{SEC_basic_results}

We being with a restatement of Definition \ref{DEF_imprinting_m_poly},
  which will be our basic tool for constructing polyfolds.
\begin{definition}[imprinting]
  \label{DEFNR1.1}\label{DEF_imprinting_m_poly}
  \hfill\\
Let \(Y\) be a set, let \(X\) be an M-polyfold, and suppose \(\oplus:X\to
  Y\) is a surjective map with the following additional properties.
  \begin{itemize}
    \item[(1)] 
    The quotient topology\index{quotient topology} \({\mathcal T}_\oplus\) on
    \(Y\), i.e. the finest topology for which \(\oplus\) is continuous, is
    metrizable.
    \item[(2)] 
    For every \(y\in Y\) there exists an open \({\mathcal
    T_\oplus}\)-neighborhood \(V =V(y)\) and a map \(H:V\rightarrow X\)
    with the following properties.
    \begin{itemize}
      \item[(a)]  
      \(\oplus\circ H=Id_{V}\).
      \item[(b)] 
      \(H\circ\oplus :\oplus^{-1}(V)\rightarrow X\) is sc-smooth.
      \end{itemize}
    \end{itemize}
We shall refer to  \(\oplus:X\rightarrow Y \) as an (M-polyfold) imprinting,
  or as an M-polyfold construction for \(Y\).
\end{definition}
%

The following theorem holds.

\setcounter{CounterSectionImprintingMethod}{\value{section}}
\setcounter{CounterTheoremImprintingMethod}{\value{theorem}}
\begin{theorem}[imprinting method]
  \label{oplusmethod}\label{THM_imprinting_method}
  \hfill\\
Let \(\oplus:X\rightarrow Y\) be an imprinting.
Then there exists a unique M-polyfold structure on \(Y\) characterized by
  the following properties.
\begin{itemize}
  \item[(1)] 
  \(\oplus:X\rightarrow Y\) is sc-smooth.
  \item[(2)] 
  For every \(y\in Y\) there exists an open neighborhood \(V=V(y)\) and a
  sc-smooth map \(H:V\rightarrow X\)
  satisfying \(\oplus\circ H=Id_{V}\).
  \end{itemize}
Moreover the M-polyfold structure has the following additional properties.
\begin{itemize}
  \item[(a)] 
  A  map \(f:Z\rightarrow Y\), where \(Z\) is an M-polyfold is sc-smooth
  if and only if \(f\) is continuous and for every \(z\in Z\) there exists
  an open neighborhood \(V=V(y)\), \(y=f(z)\),  and a sc-smooth map
  \(H:V\rightarrow X\) such that \(H\circ f:f^{-1}(V)\rightarrow X\)
  is sc-smooth and \(\oplus\circ H=Id_{V}\).
  \item[(b)] 
  If \(g:Y\rightarrow Z\) is a map into a M-polyfold, then \(g\) is
  sc-smooth if and only if \(g\circ \oplus:X\rightarrow Z\) is sc-smooth.
  \end{itemize}
\end{theorem}
%
\begin{proof}
The proof is given in Section \ref{SEC_proofs_of_theorems}.
\end{proof}
%

\begin{remark}
  \hfill\\
Given an imprinting \(\oplus:X\rightarrow Y\), the theorem
  says that a map denoted by \(h:Y\rightarrow Z\), where \(Z\) is a
  M-polyfold, is sc-smooth if and only if  \(h\circ\oplus\) is sc-smooth.
We also gave another criterion which gives a characterization of
  sc-smoothness for a map \(g:Z\rightarrow X\). 
Namely the (local) maps \(H\circ g\) have to be sc-smooth. One should view
  these compositions as local lifts with respect the structure map
  \(\oplus:X\rightarrow Y\).
If we are only interested in verifying that \(g:Z\rightarrow Y\) is
  sc-smooth the following criterion is very often more practical.
Namely a map \(g:Z\rightarrow Y\) is sc-smooth, where \(Z\) is a
  M-polyfold, provided for every \(z\in Z\) there exists an open
  neighborhood \(V=V(z)\) and a sc-smooth map \(\wt{g}:V\rightarrow X\)
  such that \(\oplus\circ \wt{g} = g|V\).
Since \(\oplus\) is sc-smooth for the M-polyfold structure on \(Y\) the
  statement is trivially true by the chain rule.
In many of our applications \(X\) will be a ssc-smanifold and therefore
  has a locally simple structure.
In these applications the local lifts \(\wt{g}\) are generally sc-smooth
  but not ssc-smooth.
\end{remark}
%

A result along the line of thought as described in the remark is given in
  the following  proposition.

\begin{proposition}[sc-smooth maps between imprintings]
  \label{prop-oplusx}\label{PROP_sc_smooth_between_imprintings}
  \hfill\\
Assume that \(\oplus:X\rightarrow Y\) and \(\oplus':X'\rightarrow Y'\) are
  two M-polyfold imprintings and assume that
  \(h:Y\rightarrow Y'\) is a map.
Then \(h\) is sc-smooth provided for each \(x\in X\) there exists an open
  neighborhood \(U=U(x)\subset X\), and a sc-smooth map
  \(\wt{h}:U\rightarrow X'\) such that the following diagram is
  commutative
\begin{align}
  \begin{CD}
  U@>\oplus >> Y\\
  @V \wt{h} VV   @V h VV\\
  X' @>\oplus' >> Y'.
  \end{CD}
  \end{align}
\end{proposition}
%
\begin{proof}
With \(Y\) and \(Y'\) equipped with the M-polyfold structures determined
  by \(\oplus\) and \(\oplus'\) we know that \(\oplus\) and \(\oplus'\) are
  sc-smooth.
By assumption it follows that \(\oplus'\circ\wt{h}\big|_U\) is sc-smooth
  and consequently
  \begin{align*}                                                          
    h\circ \oplus\big|_U :U\rightarrow Y' 
    \end{align*}
  is sc-smooth.  
Hence \(h\circ \oplus:X\rightarrow Y'\) is sc-smooth by commutativity, and
  thus by Theorem \ref{THM_imprinting_method}(b) it follows that \(h\) is
  sc-smooth.
\end{proof}
%

\begin{remark}
The drawback is that we have to run a test for every \(x\in X\) in order
  to check the smoothness property of \(h\), despite the fact that there are
  usually many points \(x\) mapped to the same point in \(Y\).
Under an additional property less tests are necessary, see Exercise
  \ref{EX_homogeneity}.
\end{remark}
%

The above Theorem \ref{THM_imprinting_method} has a refinement which will
  be important to us.
Note that in this theorem we have as an assumption that the quotient
  topology \({\mathcal T}_\oplus\) on \(Y\) is metrizable.
Of course, whenever we have  a map of the form \(f:X\rightarrow Y\), where
  \(X\) is a topological space and \(Y\) is a set, the quotient topology
  \({\mathcal T}_f\) on \(Y\)  is defined.
We shall give a criterion which automatically implies that \({\mathcal
  T}_f\) is metrizable.

\setcounter{CounterSectionMetrizabilityConditions}{\value{section}}
\setcounter{CounterTheoremMetrizabilityConditions}{\value{theorem}}
\begin{theorem}[metrizablity conditions]
  \label{THMKLP1.3}\label{THM_metrizability_conditions}
  \hfill\\
Let \(X\) be a M-polyfold and \(\oplus:X\rightarrow Y\) a surjective map
  onto a set. 
Assume the following.
\begin{itemize}
  \item[(1)] 
  The topology \({\mathcal T}_X\) on \(X\)  is second countable.
  \item[(2)] 
  The quotient topology \({\mathcal T}_\oplus\) has the following
    property.
  For each open set \(W\in \mathcal{T}_\oplus\) and each \(z\in W\) there
    exists an open \({\mathcal T}_\oplus\)-neighborhood \(L\) of \(z\) such
    that \(\cl_Y(L)\subset W\)
  \item[(3)] 
  For every \(x\in X\) there exists an open neighborhood
    \(U(\oplus(x))\in{\mathcal T}_\oplus\) and a map
    \(H:(U(\oplus(x)),\oplus(x))\rightarrow (X,x)\) such that
  \begin{itemize}
    \item[(a)] 
    \(\oplus\circ H=Id_{U(\oplus(x))}\).
    \item[(b)] 
    \(H\circ\oplus: \oplus^{-1}(U(\oplus(x)))\rightarrow X\) is
    continuous.
    \end{itemize}
  \item[(4)]
  For every \(y\in Y\) there exists \(U=U(y)\in {\mathcal T}_\oplus\) and
    a map \(H:U\rightarrow X\) such that
  \begin{itemize}
    \item[(a)] 
    \(\oplus\circ H=Id_U\).
    \item[(b)] 
    The map \(H\circ\oplus:\oplus^{-1}(U)\rightarrow X\) is sc-smooth.
    \end{itemize}
  \end{itemize}
Then the quotient topology \({\mathcal T}_\oplus\) is metrizable and consequently
  \(\oplus:X\rightarrow Y\) is an imprinting.
\end{theorem}
%
\begin{proof}
The proof is given later in Subsection \ref{SSSEC2.3}.\end{proof}
%

\begin{remark}
  \hfill\\
Conditions (1)-(3) imply that the quotient topology is metrizable.  
Then  together with (4) we see that we have an imprinting.
\end{remark}
%

\begin{exercise}
  \label{homogeneity}\label{EX_homogeneity}
  \hfill\\
Consider an imprinting \(\oplus:X\rightarrow Y\). 
We say that \(\oplus\) has the \emph{homogeneity
  property}\index{homogeneity property} provided that for each pair of
  points \(x_1, x_2\) with \(\oplus(x_1)=\oplus(x_2)\), there exist open
  neighborhoods \(U_1 = U_1(x_1)\) and \(U_2=U_2(x_2)\) of \(x_1\) and
  \(x_2\) respectively, and there exists an
  sc-diffeomorphism \(K:U_1\rightarrow U_2\) such that \(\oplus\circ
  K=\oplus \) on \(U_1\).
Assume that \(\oplus':X'\rightarrow Y'\) is a second
  imprinting and that \(h:Y\rightarrow Y'\) is  a map.
Show that \(h\) is sc-smooth for the defined structures provided for each
  \(y\in Y\) there exists a point \(x\in X\) with \(\oplus(x)=y\), an open
  neighborhood \(U(x)\) and a sc-smooth map \(\wt{h}:U(x)\rightarrow X'\)
such that \(\oplus'\circ \wt{h}(z)=h\circ\oplus(z)\) for \(z\in U(x)\).
Compare this with Proposition \ref{PROP_sc_smooth_between_imprintings}.
\end{exercise}
%

%
\subsection{The Example of Gluing}
  \label{GLUINg}\label{qsec2.2}\label{SEC_example_gluing}
Here and throughout, we will use the following standard notation
  \begin{align*}                                                          
    \mathbb{R}^+ = [0, \infty), \qquad \mathbb{R}^- = [-\infty,
    0), \qquad\text{and}\qquad S^1 = \mathbb{R}/\mathbb{Z}.
    \end{align*}
For each fixed \(\delta_0>0\), we let \(H^{3,\delta_0}({\mathbb R}^+\times
  S^1,{\mathbb R}^N)\) and \(H^{3,\delta_0}({\mathbb R}^-\times S^1,{\mathbb
  R}^N)\) respectively denote the Hilbert spaces of functions \(u^\pm:
  \mathbb{R}^\pm \times S^1\to \mathbb{R}^N\) determined by the
  property that for each multi-index \(\alpha\) with \(|\alpha|\leq 2\) the
  function \((s,t)\mapsto e^{\delta_0 |s|}u^\pm(s,t)\) is in  \(L^2({\mathbb
  R}^\pm\times S^1,{\mathbb R}^N)\).
We then define the Hilbert spaces
  \begin{align*}                                                          
    H_c^{3, \delta_0}(\mathbb{R}^\pm \times S^1, \mathbb{R}^N) =
    \mathbb{R}^N \oplus H^{3, \delta_0}(\mathbb{R}^\pm \times S^1,
    \mathbb{R}^N);
    \end{align*}
   that is, each \(u^\pm \in H_c^{3, \delta_0}(\mathbb{R}^\pm \times S^1,
   \mathbb{R}^N)\) can be written as \(u^\pm = c^\pm + r^\pm\) where
   \(c^\pm \in \mathbb{R}^N\) and \(r^\pm\in H^{3,
   \delta_0}(\mathbb{R}^\pm\times S^1, \mathbb{R}^N)\).
Given a weight sequence \(\delta = (\delta_0, \delta_1, \delta_2, \ldots
  )\) with
  \begin{align}
    0<\delta_0<\delta_1<....<\delta_i<\delta_{i+1}<\ldots
    \end{align}
  we can equip \(H_c^{3, \delta_0}(\mathbb{R}^\pm \times S^1,
  \mathbb{R}^N)\) with an sc-structure where level \(m\)
  corresponds to \(H^{3+m,\delta_m}_c\). 
We shall let \(H^{3,\delta}_c({\mathbb R}^\pm\times S^1,{\mathbb R}^N)\)
  denote these affine sc-Hilbert spaces.
We denote by 
  \begin{align}
    E\subset H^{3,\delta}_c({\mathbb R}^+\times S^1,{\mathbb R}^N)\oplus
    H^{3,\delta}_c({\mathbb R}^-\times S^1,{\mathbb R}^N)
    \end{align}
  the sc-subspace consisting of all \((u^+,u^-)=(c^+ +r^+, c^-+r^-) \) with
  \(c^+=c^-\).

%
\subsubsection{Cylinder Gluing}\label{SEC_cylinder_gluing}
Denote by \(\varphi:(0,1]\rightarrow [0,\infty)\) the diffeomorphism
  defined by
  \begin{align}
    \varphi(s)=e^{\frac{1}{s}}-e.
    \end{align}
It is called the {\bf exponential gluing profile}\index{exponential gluing
  profile}.
We define \({\mathbb B}\) by
  \begin{align*}                                                          
    \mathbb{B} = \big\{ z\in \mathbb{C}: |z|<\textstyle{ \frac{1}{4}}\big\}
    \end{align*}
  and call them {\bf gluing parameters}. \index{gluing parameters}
We define \(Z_0\) by 
  \begin{align}
    Z_0= ({\mathbb R}^+\times S^1)\sqcup ({\mathbb R}^-\times S^1)
    \end{align}
  and for \(a\neq 0\) in \({\mathbb B}\) written as \(a=|a|\cdot e^{2\pi
  i\theta}\) we define with \(R=\varphi(|a|)\)
  \begin{align*}
    Z_a=\left\{\{(s,t),(s',t')\}\subset [-R, R]\times S^1 \ | \  s=s'+R,\
    t=t'+\theta \right\};
    \end{align*}
  observe that as a consequence of this definition, for any \(\{(s,t),
  (s',t')\}\in Z_a\) we necessarily have \(s\in [0, R]\) and \(s'\in [-R,
  0]\).
We consider two natural bijections for \(a\in {\mathbb B}\), \(a\neq 0\)
  \begin{align}
    [0,R]\times S^1\xleftarrow{c^+_a} Z_a \xrightarrow{c^-_a} [-R,0]\times
    S^1,
    \end{align}
  defined by 
\begin{align*}                                                            
  &c_a^+\big(\{(s, t), (s', t')\}\big) = (s,t)\\
  &c_a^-\big(\{(s, t), (s', t')\}\big) = (s', t').
  \end{align*}
We note the following trivial fact.

\begin{lemma}[Smooth structures on $Z_a$]
  \label{LEM_smooth_structure_Za}
  \hfill\\
The smooth structures on \(Z_a\) making \(c^+_a\) or \(c^-_a\)
  diffeomorphisms coincide.
\end{lemma}
%

%
\subsubsection{Gluing}
Pick a smooth map, called a {\bf cut-off model}\index{cut-off model},
  \(\beta:{\mathbb R}\rightarrow [0,1]\) satisfying
  \begin{align*}                                                          
     \beta(s)=1 &\ \text{for} \ s\leq -1.\\
    \beta(s)=0& \ \text{for}\ s\geq 1\nonumber \\
    \beta(s)+\beta(-s)=1 &\ \text{for all}\ s\in{\mathbb R}.\nonumber
    \end{align*}
With  \(u=(u_x,u_y)\in E\) we can define a glued map 
  \begin{align}
    \oplus(a,u):Z_a\rightarrow {\mathbb R}^N
    \end{align}
  as follows.
If \(a=0\) we just recover \(u\),
  i.e. \(\oplus(0,u)=u\).
If \(0<|a|<1/4\) write \(a=|a|\cdot e^{2\pi i\theta}\) and define with
  \(R=\varphi(|a|)\)
  \begin{align*}                                                            
    & \oplus(a,u)(\{(s,t),(s',t')\}_a)&\\
    &\beta(|s|-R/2)\cdot u_x(s,t)+\beta(|s'|-R/2)\cdot u_y(s',t')&\nonumber
    \end{align*}
Here \(t\in S^1={\mathbb R}/{\mathbb Z}\) and \(t\in [0,R]\).
We shall refer to \(\oplus(a,u)\) as a {\bf glued map}.\index{glued map}

Here is an example of an imprinting-construction. 
Introduce the set
  \(X^{3,\delta_0}_{\varphi}:=X^{3,\delta_0}_{\varphi}({\mathbb R}^N)\) as
  follows
  \begin{align}\label{EQ_example_construction}
    X^{3,\delta_0}_{\varphi}= E\coprod \left(\coprod_{a\in {\mathbb
    B}\setminus\{0\}} H^3(Z_a,{\mathbb R}^N)\right).
    \end{align}
Here \(H^3(Z_a,{\mathbb R}^M)\) for \(0<|a|<1/4\) is the usual Sobolev
  space.
We see immediately that 
  \begin{align}
    \oplus:{\mathbb B}\times E \rightarrow X^{3,\delta_0}_{\varphi}
    \end{align}
  is a surjective map. 
\begin{theorem}[\(\oplus\) is an imprinting]
  \label{THM_oplus_is_imprinting}
  \hfill\\
The map \(\oplus\) is an imprinting.
The M-polyfold  structure on \(X^{3,\delta_0}_{\varphi}\) does not depend
  on the specific choice of \(\beta\).
\end{theorem}
%
We shall denote the set \(X^{3,\delta_0}_{\varphi}\) equipped with this
  M-polyfold structure by \(X^{3,\delta}_{\varphi}({\mathbb R}^N)\).
A proof is given in \cite{FH-poly}. Using results in \cite{HWZ8.7} it is
  not difficult to prove the theorem. In fact one can construct a global
  map \(H: X^{3,\delta_0}_{\varphi}\rightarrow {\mathbb B}\times E \) such
  that
  \begin{itemize}
    \item[(1)] 
    \(\oplus\circ H=\text{id}_{X^{3,\delta_0}_{\varphi}}\).
    \item[(2)] 
    \(r:=H\circ \oplus: X^{3,\delta_0}_{\varphi}\rightarrow
    X^{3,\delta_0}_{\varphi}\) is sc-smooth.
    \end{itemize}
Then \(O=r(X^{3,\delta_0}_{\varphi})\) is an sc-smooth retract and as a
  subset of a metrizable space metrizable.  
Further
  \begin{align}
    \oplus|O: O\rightarrow X^{3,\delta_0}_{\varphi}
    \end{align}
  is a bijection with inverse \(H\).  
Equip \(X^{3,\delta_0}_{\varphi}\) with the topology as well as the
  M-polyfold structure making \(\oplus|O\) a sc-diffeomorphism.
One easily proves that the topology is the quotient topology. 

Clearly the construction can be transferred to abstract nodal disk pairs
  instead of using the standard cylinders via holomorphic polar coordinates.

We have a canonical map \(p_{\mathbb B}:
  X^{3,\delta_0}_{\varphi}\rightarrow {\mathbb B}\) which for given \(u\)
  extracts the gluing parameter \(a(u)\) of the underlying domain.
We leave the following proposition to the reader. 

\begin{proposition}[properties of \(p_{\mathbb B}\)]
  \label{PROP_properties_of_pB}
  \hfill\\
The map \(p_{\mathbb B}\) is sc-smooth, surjective, and submersive.
\end{proposition}
%

%
\subsection{Proof of the Basic Theorems}
  \label{SSSEC2.3}\label{SEC_proofs_of_theorems}

We first prove Theorem \ref{THM_imprinting_method} and also provide some
  additional results.

\setcounter{CurrentSection}{\value{section}}
\setcounter{CurrentTheorem}{\value{theorem}}
\setcounter{section}{\value{CounterSectionImprintingMethod}}
\setcounter{theorem}{\value{CounterTheoremImprintingMethod}}
\begin{theorem}[imprinting method]
  \hfill\\
Let \(\oplus:X\rightarrow Y\) be an imprinting.
Then there exists a unique M-polyfold structure on \(Y\) determined by
  \(\oplus\) that is characterized by the following properties.
\begin{itemize}
  \item[(1)] 
  \(\oplus:X\rightarrow Y\) is sc-smooth.
  \item[(2)] 
  For each \(y\in Y\) there exists an open neighborhood \(V=V(y)\in
  \mathcal{T}_\oplus\) and a
  sc-smooth map \(H:V\rightarrow X\) satisfying \(\oplus\circ
  H=Id_{V}\).
  \end{itemize}
Moreover the M-polyfold structure has the following additional properties.
\begin{itemize}
  \item[(a)] 
  A  map \(f:Z\rightarrow Y\), where \(Z\) is a M-polyfold is sc-smooth if
  and only if \(f\) is continuous and for every \(z\in Z\) there exists an
  open neighborhood \(V\), \(y=f(z)\),  and a sc-smooth map
  \(H:V\rightarrow X\) such that \(H\circ f:f^{-1}(V)\rightarrow X\)
  is sc-smooth and \(\oplus\circ H=Id_{V}\).
  \item[(b)] 
  If \(g:Y\rightarrow Z\) is a map into a M-polyfold, then \(g\) is
  sc-smooth if and only if \(g\circ \oplus:X\rightarrow Z\) is sc-smooth.
  \end{itemize}
\end{theorem}
%
\begin{proof}
\setcounter{section}{\value{CurrentSection}}
\setcounter{theorem}{\value{CurrentTheorem}}
We begin by noting that we must establish the existence of an M-polyfold
  atlas, and that the associated equivalence class there of is unique.
To that end, we consider \(H:V\rightarrow X\) and define \(O=H(V)\), and
  outline our task as follows.
\begin{enumerate}                                                         
  \item Show \(H:V\to O\) is a bijection.
  \item Show \(H:V\to O\) is a homeomorphism.
  \item Use \(H\) and \(\oplus\) to construct a local sc-retraction \(r\)
    which yields the local sc-retract models.
  \item Deduce the remaining properties.
  \end{enumerate}
We now proceed through these steps.
Since \(V\in {\mathcal T}_\oplus\) the set \(U=\oplus^{-1}(V)\) is open in
  \(X\).
We define the sc-smooth  map \(r:U\rightarrow X\) by \(r= H\circ
  \oplus\).
One easily verifies 
  \begin{itemize}
    \item[(i)]  
    \(r(U)\subset U\).
    \item[(ii)] 
    \(r\circ r=r\).
    \item[(iii)] 
    \(r(U)=O\).
    \item[(iv)] 
    \(H:V\rightarrow O\) is a bijection. 
    \end{itemize}
This establishes \(O\subset X\) as a local sc-retract.
Then, equipping \(O\) with the subspace topology, we next show that
  \(H:V\rightarrow O\) is a homeomorphism.
An open subset of \(O\) has the form \(O\cap W\), where \(W\) is open in
  \(X\).
First, we show that \(H^{-1}(O\cap W)\) belongs to \({\mathcal
  T}_\oplus\) which establishes continuity of \(H\).
This is equivalent to showing that \(\oplus^{-1}(H^{-1}(O\cap W))\) is
  open in \(X\). 
We can write this set as follows
  \begin{align*}                                                          
    \oplus^{-1}(H^{-1}(O\cap W))
    &= (H\circ\oplus)^{-1}(O\cap W)\\
    &=\{x\in U\ :\ r(x)\in O\, \text{ and }\, r(x)\in W\}\\
    &=\{x\in U\ :\ r(x)\in W\},\qquad\text{because }r(U)=O,
    \end{align*}
  which is open since \(r\) is sc-smooth and in particular continuous.

To finish showing that \(H\) is a homeomorphism, we show that
  \(H:V\rightarrow O\) is an open map into \(O\).
Let \(Q\subset V\) be an open subset and define the open
  \(P=\oplus^{-1}(Q)\).
We note that \(P\subset U\) and now \(r(P)\subset P\).

We now show that \(r(P)\subset P\).
Let \(p\in P\).
Because \(p\in P=\oplus^{-1}(Q)\), we have
  \begin{align}\label{EQ_op_in_Q}                                         
    \oplus(p)\in Q.
    \end{align}
Because \(r=H\circ \oplus\), we assume
  \begin{align}\label{EQ_contra}                                          
    H\circ \oplus(p)\notin P
    \end{align}
  and derive a contradiction.
Because \(P=\oplus^{-1}(Q)\), we have from equation (\ref{EQ_contra}) that
  \begin{align*}                                                          
    \oplus\circ H\circ \oplus(p)\notin Q.
    \end{align*}
But \(\oplus\circ H = Id\).  
So \(\oplus(p)\notin Q\).
This contradicts equation (\ref{EQ_op_in_Q}) as desired.
Thus \(r(p) = H\circ\oplus(p)\in P\), and \(r(P)\subset P\).

We claim that \(H(Q)= O\cap P\), which would prove our assertion.
First we note that \(H(Q)= H\circ\oplus (\oplus^{-1}(Q)) =r(P)\subset
  O\cap P\), and hence \(H(Q)\subset O\cap P.\)
To show that \(O\cap P\subset H(Q)\), we fix \(x\in O\cap P\) and we find
  a unique \(z\in O\) with \(H(z)=x\) and we must have
  \(\oplus(x)=\oplus\circ H(z)=z\) so that \(z\in \oplus(P)=Q\).
Hence \(x\in H(Q)\), and thus \(O\cap P\subset H(Q)\) so that indeed
  \(H(Q) = O\cap P\).

At this point we have shown that \(H:V\rightarrow O\) is a
  homeomorphism.
Assume that \(H':V'\rightarrow O'\) is a second construction of this
  kind.
We consider
  \begin{align*}                                                          
    H'\circ H^{-1} : H(V\cap V')\rightarrow H'(V\cap V'), 
    \end{align*}
  which is a homeomorphism between open subsets of \(O\) and \(O'\),
  respectively.
These open subsets are also sc-smooth retracts as discussed above.
The transition map \(H'\circ H^{-1}\) can be computed as 
  \begin{align*}                                                          
    H'\circ H^{-1}(o)
    &=  H'\circ H^{-1} \circ r(o)\\
    &=  H'\circ H^{-1}\circ H\circ\oplus(o)\\
    &=  H'\circ \oplus(o)\\
    &=  r'(o)
    \end{align*}
  which is sc-smooth.   
Hence we have shown that all occurring transitions are sc-smooth.
Hence \((Y,{\mathcal T}_\oplus)\) is a metrizable topological space
  equipped with an sc-smooth M-polyfold atlas.

At this point it is worth mentioning that given a map \(\oplus:X\to Y\)
  for which \(\mathcal{T}_\oplus\) is metrizable, and given any collection
  \(\{(V_y, H_y)\}_{y\in Y}\) for which the following hold
  \begin{enumerate}                                                       
    \item \(\{V_y\}_{y\in Y}\) is an open cover for \(Y\)
    \item \(\oplus\circ H_y = Id\) for each \(y\in Y\)
    \item \(H_y\circ \oplus\) is sc-smooth,
    \end{enumerate}
  the above argument shows that the \(\{(V_y, H_y)\}_{y\in Y}\) yields an
  sc-smooth atlas.
More importantly however, is that any additional \((V', H')\) which is not
  an element of \(\{(V_y, H_y)\}_{y\in Y}\), but satisfies \(\oplus\circ
  H' = Id\) and \(H'\circ \oplus\) is sc-smooth will yield and additional
  chart which is compatible with the previous atlas.
In this way, an imprinting \(\oplus:X\to Y\) determines a \emph{unique}
  M-polyfold structure on \(Y\), and in fact an imprinting is not depend
  upon the collection \(\{(V_y, H_y)\}_{y\in Y}\) provided that such a
  collection exists.
We also note that the above construction immediately guarantees that the
  maps \(H:V\to H(V)\) are sc-diffeomorphisms.

That \(\oplus:X\rightarrow Y\) is sc-smooth follows by taking a point
  \(x_0\in X\) and defining \(y_0=\oplus(x_0)\).
For \(V=V(y_0)\) we then find \(H:V\rightarrow X\) with the usual
  properties so that
  \begin{align*}                                                          
    H\circ\oplus :\oplus^{-1}(V)\rightarrow X  
    \end{align*}
  is sc-smooth. 
Since \(x_0\in \oplus^{-1}(V)\) and \(H\) is a local
  sc-diffeomorphism we see that \(\oplus\) is sc-smooth near \(x_0\).

We have now proved the first part of Theorem \ref{THM_imprinting_method},
  and it remains to prove properties (a) and (b).
The proof of the assertions in (a) is just the usual definition of being an
  sc-smooth map into the M-polyfold \(Y\) using the fact that the \(H\)
  are sc-diffeomorphisms, i.e. essentially charts.
Concerning the proof of (b) it is clear that the sc-smoothness of \(g\circ
  \oplus:X\rightarrow Z\) implies the sc-smoothness of \(g:Y\rightarrow Z\).
If the latter is sc-smooth we use the fact, that
  \(\oplus:X\rightarrow Y\) for the M-polyfold structure on \(Y\) is
  sc-smoothness and the assertion follows from the chain rule.

This completes the proof of Theorem \ref{THM_imprinting_method}.
\end{proof}

Next we shall proof Theorem \ref{THM_metrizability_conditions}.

\setcounter{CurrentSection}{\value{section}}
\setcounter{CurrentTheorem}{\value{theorem}}
\setcounter{section}{\value{CounterSectionMetrizabilityConditions}}
\setcounter{theorem}{\value{CounterTheoremMetrizabilityConditions}}
\begin{theorem}[metrizabliity conditions]
  \hfill\\
Let \(X\) be a M-polyfold and \(\oplus:X\rightarrow Y\) a surjective map
  onto a set.
Assume the following.
\begin{itemize}
  \item[(1)] 
  The topology \({\mathcal T}_X\) on \(X\) is second countable.
  \item[(2)] 
  The quotient topology \({\mathcal T}_\oplus\) has the following
    property.
  For each  \(W\in {\mathcal T}_\oplus\) and \(z\in W\) there exists an
    open neighborhood \(L\) of \(z\) such that \(\text{cl}_Y(L)\subset
    W\).
  \item[(3)] 
  For every \(x\in X\) there exists an open neighborhood
  \(V=V(\oplus(x))\in{\mathcal T}_\oplus\) and a map
  \(H:(V,\oplus(x))\rightarrow (X,x)\) such that
  \begin{itemize}
    \item[(a)] 
    \(\oplus\circ H=Id_{V}\).
    \item[(b)] 
    \(H\circ\oplus: \oplus^{-1}(V)\rightarrow X\) is continuous.
    \end{itemize}
  \item[(4)] 
  For every \(y\in Y\) there exists \(V=V(y)\in {\mathcal T}_\oplus\) and
  a map \(H:V\rightarrow X\) such that
  \begin{itemize}
    \item[(a)] 
    \(\oplus\circ H=Id_V\).
    \item[(b)] 
    The map \(H\circ\oplus:\oplus^{-1}(V)\rightarrow X\) is sc-smooth.
    \end{itemize}
  \end{itemize}
Then the quotient topology \({\mathcal T}_\oplus\) is metrizable and
  consequently \(\oplus:X\rightarrow Y\) is an imprinting.
\end{theorem}
%
\begin{proof}
\setcounter{section}{\value{CurrentSection}}
\setcounter{theorem}{\value{CurrentTheorem}}
The key assertion, which is basic for our claim, is the following 
  \begin{itemize}
    \item[(i)] 
    The maps \(H:(V,\oplus(x))\rightarrow (H(V),x)\)
    are homeomorphisms, where \(H(V)\) is equipped with the
    topology induced from \(X\).
    \end{itemize}
The proof of this is precisely as in the proof of Theorem
  \ref{THM_imprinting_method}.
An immediate consequence of this is that 
  \begin{itemize}
    \item[(ii)] 
    \(\oplus:(X,{\mathcal T}_X)\rightarrow (Y,{\mathcal T}_\oplus)\) is
      open.
    \end{itemize}
To see this let \(U\subset X\) be open and pick \(y\in \oplus(U)\). 
Pick \(x\in U\) with \(\oplus(x)=y\) and for \(V=V(y)\) take
  \(H:(V,y)\rightarrow (X,x)\).
This establishes (ii).
We now claim the following.
  \begin{itemize}
    \item[(iii)] 
    \(\mathcal{T}_\oplus\) is second-countable.
    \end{itemize}
To see this, first note that since \(H\) is a homeomorphism onto its image
  we may assume that \(V=V(y)\) is so small that \(H(V)\subset U\).
Then
  \begin{align}
    V(y) =\oplus\circ H(V(y))\subset \oplus(U).
    \end{align}
Since \(X\) is metrizable and second countable we can take a countable
  basis \({(U_i)}_{i=1}^\infty\) for \({\mathcal T}_X\) and define
  \({(O_i)}_{i=1}^{\infty} \subset {\mathcal T}_\oplus\) by
  \begin{align}
    O_i=\oplus(U_i).
  \end{align}
Then by the continuity of \(\oplus\) we see that
  \({(O_i)}_{i=1}^{\infty}\) is a basis for \({\mathcal T}_\oplus\). 
Hence \({\mathcal T}_\oplus\) is second countable.
  \begin{itemize}
    \item[(iv)]  \({\mathcal T}_\oplus\) is Hausdorff.
    \end{itemize}
For this let \(y,y'\in Y\) be two different points. 
Pick \(x,x'\in X\) with \(\oplus(x)=y\) and \(\oplus(x')=y'\). 
Since \(X\) is metrizable we can fix a metric. 
We claim that for \(\varepsilon>0\) small enough
  \begin{align}
    \oplus(B_{\varepsilon}(x))\cap \oplus(B_{\varepsilon}(x'))=\emptyset.
    \end{align}
Of course, otherwise we would conclude by continuity of \(\oplus\) that
  \(y=\oplus(x)=\oplus(x')=y'\).
Thus \(\mathcal{T}_\oplus\) is Hausdorff.
\begin{itemize}
  \item[(v)] 
  The topology \({\mathcal T}_\oplus\) is completely regular. 
  \end{itemize}
For this take a closed \(A\subset Y\) and a point \(a\not\in A\).  
Define \(B=\oplus^{-1}(A)\) and take \(b\in X\) with \(\oplus(b)=a\). 
The set \(B\) is closed.
Take an open neighborhood \(V=V(a)\) such that \(V\cap A=\emptyset\). 
We have the map \(H:(V,a)\rightarrow (X,b)\).
Then \(H:(V,a)\rightarrow (O,b)\) is a homeomorphism, where \(O=r(U)\),
  \(U=\oplus^{-1}(V)\), and \(r=H\circ\oplus :U\rightarrow U\).
We note that \(O\cap B=\emptyset\). By assumption we find an open
  neighborhood \(L\) of \(a\) such that \(\text{cl}_Y(L)\subset V\).
Since \(V\) is metrizable and homeomorphic to \(O\), we find a
  continuous map \(f':V\rightarrow [0,1]\) which on \(a\) takes the
  valued \(0\) and which on \(V\setminus L\) takes the value \(1\).
Finally we define  a map \(f:Y\rightarrow [0,1]\) as follows.
\begin{align}
  f(z)=\left[ \begin{array}{cc}
  1 & \text{for}\ z\in Y\setminus V\\
  f'(z)& \text{for}\ z\in V
  \end{array}
  \right.
  \end{align}
\begin{itemize}
  \item[(vi)] 
  The map \(f\) is continuous.
  \end{itemize}

Recall that we have shown that \(\mathcal{T}_\oplus\) is second countable.
Hence \(\mathcal{T}_\oplus\) is first countable, and hence
  \(\mathcal{T}_\oplus\) is a sequential space.
Consequently, to prove \(f\) is continuous it is sufficient to prove \(f\)
  is sequentially continuous.
To that end, let \(y\in Y\) and assume \(y_k\rightarrow y\).   
If \(y\in V\) it follows that \(f(y_k)\rightarrow f(y)\) since \(f'\)
  is continuous.
If \(y\in Y\setminus \text{cl}_Y(V)\) continuity follows trivially as
  well.
The final case is \(y\in \text{cl}_Y(V)\setminus V\).
Arguing by contradiction and perhaps taking a subsequence, we may assume 
  \begin{align}
    f(y)-\varepsilon=1-\varepsilon \geq f(y_k).
    \end{align}
This implies that \(y_k\in L\). Hence \(y\in \text{cl}_Y(L)\subset V\)
  which gives a contradiction.

Finally, using that \({\mathcal T}\) is second countable, Hausdorff, and
  completely regular,  we apply Urysohn's metrization theorem and the proof
  is complete.
\end{proof}
%

%
\subsection{Additional Results about the Imprinting Method}
The construction via the imprinting method has certain naturality
  properties.

\begin{theorem}[equivalence of composing imprintings]
  \label{THMP1.4}\label{THM_equiv_composing_imprintings}
  \hfill\\
Assume that \(X\) is a M-polyfold and \(Y\) and \(Z\) are sets and the
  maps  \(\oplus_1\) and \(\oplus_2\) in the diagram
  \begin{align}
    X\xrightarrow{\oplus_1} Y\xrightarrow{\oplus_2} Z
    \end{align}
  are surjective. 
Define \(\oplus:X\rightarrow Z\) by \(\oplus =\oplus_2\circ\oplus_1\).
Assume further that \(\oplus_1:X\rightarrow Y\) is an imprinting and
  assume that \(Y\) is equipped with the associated M-polyfold structure.
Then the following two statements are equivalent.
\begin{itemize}
  \item[(1)] 
  \(\oplus:X\rightarrow Z\) is an imprinting.
  \item[(2)] 
  \(\oplus_2:Y\rightarrow Z\) is an imprinting.
  \end{itemize}
Moreover the induced M-polyfold structures (and topology) on \(Z\) by both
  imprintings coincide.
\end{theorem}
%
\begin{proof}
The quotient topology \({\mathcal T}_{\oplus}\) associated to \(\oplus\)
  and the quotient topology \({\mathcal T}_{\oplus_2}\) on \(Z\) associated
  to \(\oplus_2\) coincide.
To see this recall that \(Y\) is equipped with \({\mathcal
  T}_{\oplus_1}\), which by assumption is metrizable.
Let \(U\subset Z\) and note that 
  \begin{align}
    \oplus^{-1}(U) =\oplus_1^{-1}(\oplus_2^{-1}(U)).
    \end{align}
  implying that  \(\oplus^{-1}(U)\in {\mathcal T}_X\) if and only if
  \(\oplus_2^{-1}(U)\in {\mathcal T}_{\oplus_1}\).  
This shows that \({\mathcal T}_{\oplus}={\mathcal T}_{\oplus_2}\).

If (1) is an imprinting, then we find for every \(z\in Z\)
  an open neighborhood \(W=W(z)\) and a map \(H:W\rightarrow X\)
  such that \(H\circ\oplus:\oplus^{-1}(W)\rightarrow X\) is sc-smooth
  and \(\oplus\circ H= Id_{W}\).
Define \(H_2:W\rightarrow Y\) by
  \begin{align}
    H_2 = \oplus_1\circ H.
    \end{align}
Then \(\oplus_2\circ H_2=\oplus_2\circ\oplus_1\circ H=\oplus\circ H
  =Id_{W}\).
Moreover
  \begin{align}
    H_2\circ \oplus_2:\oplus_2^{-1}(W)\rightarrow Y
    \end{align}
  can be written as \(\oplus_1\circ H\circ \oplus_2\). 
By the definition of the M-polyfold structure on \(Y\) this map is
  sc-smooth precisely if \(\oplus_1\circ H\circ \oplus_2\circ\oplus_1=
  \oplus_1\circ H\circ\oplus\) defined on
  \(\oplus_1^{-1}(\oplus_2^{-1}(W))=\oplus^{-1}(W)\)
  is sc-smooth. 
The maps \(H\circ\oplus\) and \(\oplus_1\) are  sc-smooth as a consequence
  of the assumption.
Hence \(\oplus_2\) is an imprinting.

Next assume (2) holds.  
Then for each fixed \(z\in Z\) we find a map  \(H_2:W\rightarrow Y\) such
  that \(\oplus_2\circ H_2=Id_{W}\) and \(H_2\circ \oplus_2\) is
  sc-smooth.
Define \(y=H_2(z)\). 
By assumption for \(U=U(y)\) there exists \(H_1:U\rightarrow X\) with the
  obvious properties. 
By adjusting \(W\) we may assume that \(H_2(W)\subset U\).
We define 
  \begin{align*}
    &H:W\rightarrow X\\
    &H=H_1\circ H_2.
    \end{align*}
We compute \(\oplus\circ H=\oplus_2\circ\oplus_1\circ H_1\circ H_2 =
  Id_{W}\).
Further \(H\circ\oplus:\oplus^{-1}(W)\rightarrow X\) can be written as
  \begin{align}
    H_1\circ H_2\circ \oplus_2\circ \oplus_1.
    \end{align}
For the structure on \(Y\) the maps \(\oplus_1\) and \(H_1\) are sc-smooth
  and by assumption \(H_2\circ\oplus_2\) is sc-smooth.
Hence our expression is sc-smooth by the chain rule.  
At this point we have proved the equivalence of (1) and (2).

Assume that one and then both imprintings define a M-polyfold structure
  on \(Z\).
Denote \(Z\) equipped with the M-polyfold structures by \(Z_\oplus\) and
  \(Z_{\oplus_2}\).
We need to show that the identity maps \(Id:Z_\oplus\rightarrow
  Z_{\oplus_2}\) and \(Id:Z_{\oplus_2}\rightarrow Z_{\oplus}\) are
  sc-smooth.
For the first assertion we have to consider expressions \(H_2\circ Id\circ
  \oplus=H_2\circ\oplus\) and for the second assertion \(H\circ Id
  \circ\oplus_2=H\circ \oplus_2\).
We note that \(H_2\circ \oplus\)  has the form \(\oplus_1\circ
  H\circ\oplus\) which is sc-smooth by the chain rule since
  \(H\circ\oplus \) and \(\oplus_1\) are.  
The map \(H\circ\oplus_2\) is sc-smooth precisely when \(H\circ
  \oplus_2\circ\oplus_1=H\circ\oplus \) is sc-smooth, by Theorem
  \ref{THM_imprinting_method}(b).
This completes the proof.
\end{proof}

This  \(\oplus\)-method will be frequently employed in \cite{FH-poly}.
Here is another result which is  very important for our approach. 

\begin{theorem}[embeddings and induced imprintings]
  \label{THMOP10.2}\label{THM_embeddings_induced_imprintings}
  \hfill\\
Assume that $\oplus:X\rightarrow Y$ is $\oplus$-polyfold construction.
Let $X'$ be another M-polyfold and $e:X'\rightarrow X$ an
  sc-diffeomorphism onto a sub-M-polyfold, and $\phi:Y'\rightarrow Y$ an
  injection defined on the set $Y'$.
Further suppose that $\oplus':X'\rightarrow Y'$  is a surjective map and
  the data fits into the commutative diagram
  \begin{align*}                                                          
   \begin{CD}
    X @>\oplus >> Y\\
    @A e AA        @A \phi AA\\
    X'@>\oplus' >>  Y'
    \end{CD}
    \end{align*}
If we have the identity
  \begin{align*}                                                          
    e(X')=\oplus^{-1}(\phi(Y'))  
    \end{align*}
  then $\oplus'$ is a M-polyfold construction and for the
  $\oplus'$-structure on $Y'$ and the $\oplus$-structure on $Y$ the map $
  \phi:Y'\rightarrow Y$ is sc-smooth and a sc-diffeomorphism onto a
  sub-M-polyfold.
\end{theorem}
%
\begin{proof}
Since $\oplus:X\rightarrow Y$ is a M-polyfold construction the quotient
  topology ${\mathcal T}_{\oplus}$ is metrizable.
Equip $Y'$ with the topology ${\mathcal T}_{\oplus'}$.
We now break the proof in to several separate claims.

\noindent $\bullet$ \textbf{Claim: }For the topology ${\mathcal
  T}_{\oplus'}$ on $Y'$ the map $\phi:Y'\rightarrow Y$ is continuous.\\

Assume that $V\subset Y$ is open. 
By definition this means that the set $U=\oplus^{-1}(V)$ is open in $X$,
  which implies that $e^{-1}(U)$ is open in $X'$ by assumption.
From  $e^{-1}(U)= e^{-1}(\oplus^{-1}(V))= \oplus'^{-1}(\phi^{-1}(V))$ we
  infer that $\phi^{-1}(V)\in {\mathcal T}_{\oplus'}$.
This proves continuity.\\

\noindent $\bullet$ \textbf{Claim: } The map  $\phi$ is an open map onto
  its image.\\

For this let $B'\subset Y'$ be open and define $B=\phi(B')$.
Let  $v_0\in B$ and pick a suitable open neighborhood $V=V(v_0)$ in $Y$
  together with the map $H:V\rightarrow X$ such that $\oplus\circ H =
  Id_V$ and $H\circ \oplus:\oplus^{-1}(V)\rightarrow X$ is sc-smooth.
Consider the map $H: V\cap \phi(Y') \rightarrow X$. 
Applying $\oplus$ we obtain that
  \begin{align*}                                                          
    \oplus\circ H(V\cap \phi(Y'))= V\cap \phi(Y')  
    \end{align*}
  which implies that $H(V\cap\phi(Y')))\subset \oplus^{-1}(V\cap
  \phi(Y'))\subset \oplus^{-1}(\phi(Y'))= e(X')$. Hence we can consider
  the map
  \begin{align*}                                                          
    \Gamma:=\oplus'\circ e^{-1}\circ H :V\cap \phi(Y')\rightarrow Y'.
    \end{align*}
We note that this map is continuous (for the topology induced from
  ${\mathcal T}_{\oplus}$) by the assumption that $e:X'\rightarrow e(X')$
  is an sc-diffeomorphism onto the sub-M-polyfold $e(X')$.
Let $v'_0\in B'$ be the point with $\phi(v_0')=v_0$. 
Then
  \begin{align*}                                                          
    \Gamma(v_0)=\oplus'\circ e^{-1}\circ H(v_0) = \phi^{-1}\circ\oplus
    \circ H(v_0)= \phi^{-1}(v_0)=v_0'.
    \end{align*}
Since this map is continuous we find an open neighborhood $Q$ of $v_0$ in
  $Y$ such that $\Gamma(Q\cap \phi(Y'))\subset B'$.  This implies
  \begin{align*}                                                          
    \phi(B')& \supset \phi(\Gamma(Q\cap \phi(Y')))\\
    &=\phi \circ \oplus'\circ e^{-1}\circ H(Q\cap\phi(Y'))\\
    &=\phi \circ\phi^{-1}\circ \oplus\circ H(Q\cap\phi(Y'))\\
    &= Q\cap\phi(Y').
    \end{align*}
Hence the map $\phi$ is open onto its image. \\

\noindent $\bullet$ \textbf{Claim: }The topology ${\mathcal T}_{\oplus'}$
  is metrizable.\\

In view of the previous discussion the map $\phi:Y'\rightarrow \phi(Y)$,
  where $Y'$ is equipped with ${\mathcal T}_{\oplus'}$ and $\phi(Y)$ with
  the topology induced from ${\mathcal T}_{\oplus}$, is a homeomorphism.
Since $\phi(Y)$ is a subset of metrizable space and $\phi:Y'\rightarrow
  \phi(Y')$ a homeomorphism we deduce that $(Y',{\mathcal T}_{\oplus'})$ is
  metrizable.\\

At this point we know that $\oplus':X'\rightarrow Y'$ is a surjective map
  for which the quotient topology on $Y'$ is metrizable.
In order to show that $\oplus':X'\rightarrow Y'$ is an imprinting we need
  to find for given $y'\in Y'$ and open neighborhood $V'$ a map
  $H':V'\rightarrow X'$ such that $\oplus'\circ H'=Id_{V'}$ so that
  $H'\circ\oplus':\oplus'^{-1}(V')\rightarrow X'$ is sc-smooth.  
The verification of this statement will occupy the rest of the proof.\\

\noindent $\bullet$ \textbf{Claim: } $\oplus' :X'\rightarrow Y'$ is an
  imprinting.\\

Start with the point $y'\in Y'$ and consider $y=\phi(y')$. For this point
  we find $V=V(y)\in {\mathcal T}_{\oplus}$ and $H:V\rightarrow X$ with the
  usual properties.
Define $V':=\phi^{-1}(V)$ which belongs to ${\mathcal T}_{\oplus'}$.  
For $u'\in V'$ we consider $H\circ\phi(u')$ and show that it belongs to
  $e(X')$.  By the  theorem hypotheses $e(X')=\oplus^{-1}(\phi(Y'))$.
Hence
  \begin{align*}                                                          
    H\circ \phi(u') \in \oplus^{-1}(\oplus\circ H)\circ
    \phi(u')=\oplus^{-1}(\phi(u'))=e(X').
    \end{align*}
Consequently we can define the map
  \begin{align*}                                                          
    H':V'\rightarrow X': H':=e^{-1}\circ H\circ \phi.  
    \end{align*}
From the identity  $\oplus\circ e= \phi\circ\oplus'$ we infer on $e(X')$
  the identity $\phi^{-1}\circ \oplus = \oplus'\circ e^{-1}$.
We compute $\oplus'\circ H' = \oplus'\circ e^{-1}\circ H\circ \phi=
\phi^{-1}\circ\oplus\circ H\circ\phi =\text{Id}_{V'}$.
The other composition $H'\circ \oplus':{(\oplus')}^{-1}(V')\rightarrow X'$
  is sc-smooth.
In order to see this we compute  for $x'\in {(\oplus')}^{-1}(V')$
  \begin{align*}                                                          
    H'\circ \oplus'(x')= e^{-1}\circ H\circ \phi\circ\oplus'(x')=
    e^{-1}\circ H\circ \oplus\circ e(x'),
    \end{align*}
  which is the composition of sc-smooth maps. This shows that
  $\oplus':X'\rightarrow Y'$ is an imprinting.\\

\noindent $\bullet$ \textbf{Claim: } $\phi(Y')$ is a sub-M-polyfold of $Y$
  and $\phi:Y'\rightarrow \phi(Y')$ is a sc-diffeomorphism.\\

Equip $Y$ and $Y'$ with the M-polyfold structures from the imprinting and
  consider $\phi:Y'\rightarrow Y$.
This map is sc-smooth if and only if $\phi\circ\oplus'$ is sc-smooth. 
Since $\phi\circ\oplus'=\oplus\circ e$, where the right-hand side is
  sc-smooth this assertion follows.

Next take $y\in \phi(Y')$ and take the usual map $H:V(y)\rightarrow X$.
Define $x=H(y)$ which belongs to the sub-M-polyfold $e(X)$. We find an
  open neighborhood $W=W(x)$ in $X$ and a sc-smooth map $r:W\rightarrow W$
  such that $r\circ r=r$ and $r(W)=W\cap e(X)$.
By possibly shrinking $V=V(y)$ we may assume that
  \begin{align*}                                                            
    H(V)\subset W.  
    \end{align*}
Define $R:V\rightarrow Y$ by  $R(v) =  \oplus\circ r\circ H(v)$. 
This is an sc-smooth map for the M-polyfold structure on $Y$.
We first note that $R(V)\subset \phi(Y')$. 
Indeed for $u\in V$ we have that $r\circ H(v)\in e(X')$ implying
  $\oplus\circ r\circ H(v)\in e(X')\subset \oplus(e(X'))= \phi(Y')$.
Moreover, if $u\in V\cap \phi(Y')$, we compute using that $H(v)\in e(X')$
  \begin{align*}                                                          
    R(v) =\oplus\circ r\circ H(v) =\oplus\circ H(v)=v.  
    \end{align*}
We define  $\wt{V}=\{u\in V\ |\ R(v)\in V\}$ and note that
  $R:\wt{V}\rightarrow \wt{V}$. Hence we may assume without loss of
  generality that we already have the property
  $R:V\rightarrow V$.
Then $R$ is a sc-smooth retraction with $R(V)=V\cap e(Y')$.  
This shows that $\phi(Y')$ is a sub-M-polyfold of $Y$.  
It remains to show that $\phi^{-1}:\phi(Y')\rightarrow Y'$ is sc-smooth.  
For this it suffices to show that for every $y\in \phi(Y')$ the sc-smooth
  retraction $R$ as just constructed has the property that $\phi^{-1}\circ
  R:V\rightarrow Y'$ is sc-smooth.
We have the identity
 \begin{eqnarray*}
 \phi^{-1}\circ R(v)&=& \phi^{-1}\circ\oplus\circ r\circ H(v)\\
 &=&\oplus'\circ e^{-1}\circ r\circ H(v),
 \end{eqnarray*}
  which is a composition of sc-smooth maps. 
Hence $\phi^{-1}$ is sc-smooth.
The proof of Theorem \ref{THM_embeddings_induced_imprintings} is complete.
\end{proof}

There is an obvious corollary to the theorem.

\begin{corollary}[imprintings and sub-M-polyfolds]
  \label{CORTY2.8}\label{COR_imprtintings_sub_polyfolds}
  \hfill\\
Assume that \(\oplus:X\rightarrow Y\) is an imprinting, and that
  \(Y'\subset Y\).
If \(X':=\oplus^{-1}(Y')\) is a sub-M-polyfold of \(X\) then
  \(\oplus':X'\rightarrow Y'\) is an imprinting; here
  \(\oplus':=\oplus\big|_{X'}\).
Additionally, \(Y'\) is a sub-M-polyfold of \(Y\), and moreover the
  M-polyfold structure on \(Y'\) induced from the ambient space \(Y\)
  agrees with the M-polyfold structure on \(Y'\) induced from \(\oplus'\).
\end{corollary}
%

The imprinting method is well-behaved with certain operations. 

\begin{theorem}[Product]
  \label{product12.3}\label{THM_mpoly_product}
  \hfill\\
Let \(\oplus:X\rightarrow Y\) and \(\oplus':X'\rightarrow Y'\) be two
  imprintings.
Then \(\oplus\times\oplus':X\times X'\rightarrow Y\times Y'\) is an
  imprinting.
The induced M-polyfold structure on \(Y\times Y'\) is the product
  structure.
In particular the quotient topology \({\mathcal T}_{\oplus\times
  \oplus'}\) on \(Y\times Y'\) is the product topology \({\mathcal
  T}_{\oplus}\times {\mathcal T}_{\oplus'}\), which is the topology having
  as basis the products of open sets.
\end{theorem}
%
\begin{proof}
The product \(\oplus\times \oplus':X\times X'\rightarrow Y\times
Y'\)\index{$\oplus\times \oplus'$} is a surjective map defined on a
  M-polyfold  with image a set.  
Let us first show that the quotient topology \({\mathcal
  T}_{\oplus\times\oplus'}\) is the product topology \({\mathcal
  T}_{\oplus}\times {\mathcal T}_{\oplus'}\).
It is clear that \({\mathcal T}_{\oplus\times\oplus'}\supset {\mathcal
  T}_{\oplus}\times {\mathcal T}_{\oplus'}\).  
Pick \(W\subset Y\times Y'\) belonging to \({\mathcal
  T}_{\oplus\times\oplus'}\) and pick \((y,y')\in W\).
There exist open neighborhoods and maps
  \begin{align}
    H:V(y)\rightarrow X\ \ \text{and}\ \ H':V(y')\rightarrow X'
    \end{align}
  with the usual properties with respect to \(\oplus\) and \(\oplus'\). 
By perhaps taking smaller \(V\) and \(V'\) we may assume that 
\begin{align}
  H(V)\times H'(V')\subset {(\oplus\times\oplus')}^{-1}(W)
  \end{align}
  since the latter expression is open for the product topology on
  \(X\times X'\).  
Then \(V\times V'\in{\mathcal T}_{\oplus}\times {\mathcal T}_{\oplus'}\)
  and we show that \(V\times V'\subset W\) which will prove that \(W\) can
  be written as a union of open product sets.

For  \((v,v')\in V\times V'\) it holds that  \((H(v),H'(v'))\in {
  (\oplus\times \oplus')}^{-1}(W)\).
Applying \(\oplus\times \oplus'\) we obtain
  \begin{align}
    (v,v')=(\oplus\times\oplus')(H(v),H'(v'))\subset
    (\oplus\times\oplus')({(\oplus\times \oplus')}^{-1}(W))= W.
    \end{align}
This completes the proof.
\end{proof}
%

\begin{theorem}[Disjoint Union]
  \label{sum12.3}\label{THM_mpoly_coproduct}
  \hfill\\
The disjoint union \(\oplus\sqcup\oplus'\) \index{$\oplus\sqcup\oplus'$}
  of two imprintings is an imprinting.
\end{theorem}
%
\begin{proof}
The proof is straightforward.
\end{proof}
%

\begin{remark}
  \hfill\\
In what follows, we will only be interested in the disjoint union of an
  \emph{ordered} pair of sets.
In this way, we will define
  \begin{align*}                                                          
    A \sqcup B = \{0\}\times A \; \bigcup \; \{1\}\times B.
    \end{align*}
A downside of this definition is that, strictly speaking, \(A\sqcup B\neq
  B\sqcup A\), however the upside is that we have a rigorous definition of
  \(A\sqcup B\) without indexing the set \(\{A, B\}\), which is normally a
  requirement of the definition.
\end{remark}
%

Finally we consider the imprinting method in relationship to the
  submersion property.
\begin{definition}[imprinting-submersion]
  \label{DEF_imprinting-submersion}
  \hfill\\
Given a set \(Y\) together with M-polyfolds \(X\) and \(Z\), we say a pair
  of maps \(\oplus:X\to Y\) and \(f:Y\to Z\) is an
  \emph{imprinting-submersion}\index{imprinting-submersion}
  provided the following hold.
  \begin{align}
    \oplus: X\rightarrow Y\ \ \text{and}\ \ f:Y\rightarrow Z
    \end{align}
  so that the following holds.
\begin{itemize}
  \item[(1)]
  Each of \(\oplus\) and \(f\) are surjective.
  \item[(2)] 
  \(\oplus:X\rightarrow Y\) is an imprinting as in
  Definition \ref{DEF_imprinting_m_poly}.
  \item[(3)] 
  \(f\circ \oplus : X\rightarrow Z\) is sc-smooth and submersive as in
  Definition \ref{DEF_submersion_property}.
  \end{itemize}
\end{definition}
%

The definition is justified by the following theorem.
\begin{theorem}[imprinting-submersion yields submersions]
  \label{THmX2.15}\label{THM_imprinting_submersion_yield_submersions}
  \hfill\\
Let \(X\xrightarrow{\oplus} Y\xrightarrow{f} Z\) be an
  imprinting-submersion.
Then for the \(\oplus\)-structure on \(Y\) the map \(f:Y\rightarrow Z\)
  is sc-smooth and submersive.
\end{theorem}
%
\begin{proof}
Consider \(y_0\in Y\) and define \(z_0=f(y_0)\).  
We find an open neighborhood \(V(y)\in {\mathcal T}_\oplus\) and
  \(H:V\rightarrow X\) such that \(\oplus\circ H =Id_V\) and
  \(H\circ\oplus: \oplus^{-1}(V)\rightarrow X\) is sc-smooth.
We define \(U:=\oplus^{-1}(V)\) and  know that the image of
  \(\tau:=H\circ\oplus:U\rightarrow X\) is contained in \(U\).
Hence 
  \begin{align}
    \tau:U\rightarrow U\ \ \text{is sc-smooth and}\ \  \tau=\tau\circ\tau.
    \end{align}
Define the sc-smooth retract \(O=\tau(U)\).  
Then \(\oplus\big|_O:O\rightarrow V\) is a sc-diffeomorphism inverse to
  \(H:V\rightarrow O\).
Define \(x_0:=H(y_0)\) and \(\bar{f}:=f\circ\oplus\). 
For \(x\in U\) we compute
  \begin{align}
    \bar{f}(\tau(x))=f\circ\oplus\circ H\circ \oplus(x)=f\circ\oplus(x)=
    \bar{f}(x),
    \end{align}
  which shows that \(\bar{f}\circ \tau=\bar{f}\) on \(U\).

Since by assumption  \(\bar{f}: X\rightarrow Z\) is sc-smooth and
  submersive and the point \((x_0,z_0)\in X\times Z\) belongs to
  \(\text{Gr}(\bar{f})\) we find an open neighborhood  \(W=W(x_0,z_0)\subset
  X\times Z\) and a sc-smooth map \(\rho:W\rightarrow W\) of the form
  \(\rho(x,z)=(\bar{\rho}(x,z),z)\) such that \(\rho\circ\rho=\rho\) and
  \(\rho(W) = W\cap \text{Gr}(\bar{f})\).

Define an open neighborhood \({\mathcal U}\) of \((y_0,z_0)\) in \(Y\times
  Z\) by the following requirements
  \begin{itemize}
    \item[(1)] \(y\in V\) and \((H(y),z)\in W\).
    \item[(2)] \(\bar{\rho}(H(y),z)\in U \).
    \item[(3)] \((\tau (\bar{\rho}(H(y),z)),z)\in W\).
    \end{itemize}
By the continuity of the maps involved the set \({\mathcal U}\) is open
  and one easily checks that \((y_0,z_0)\in {\mathcal U}\). 
Indeed \((H(y_0),z_0)=(x_0,z_0)\in W\) which shows (1). 
Since \(\bar{f}(H(y_0)) =\bar{f}(x_0)= f(\oplus(x_0))= f(y_0)=z_0\) it
  follows that \(\bar{\rho}(H(y_0),z_0)=H(y_0)\) showing (2).
Since 
  \begin{align}
   \tau(\bar{\rho}(H(y_0),z_0))=\tau(\bar{\rho}(x_0,z_0))=\tau(x_0)=H\circ
   \oplus(x_0)= H(y_0)=x_0
   \end{align} 
  it follows that \((\tau(\bar{\rho}(H(y_0),z_0)),z_0)\in W\). 

If \((x,z)\in W\) with \(\bar{f}(x)=z\) it follows that  \((x,z)\in W\cap
  \text{Gr}(\bar{f})\) and therefore \(\bar{\rho}(x,z)=x\).
If \((y,z)\in {\mathcal U}\) it follows that \((H(y),z)\in W\) so that
  \(\rho(H(y),z)\) is defined and belongs to \(W\).
Moreover, in view of (2) we see that \(\oplus\circ\bar{\rho}(H(y),z)\) is
  defined and belongs to \(V\). Hence we obtain the sc-smooth map
  \begin{align*}
    &\delta:{\mathcal U}\rightarrow V\times Z\\
    &\delta(y,z)= (\oplus(\bar{\rho}(H(y),z)),z),
    \end{align*}
  which can be written as \(\delta(y,z)=(\bar{\delta}(y,z),z)\). 
We now aim to show that \(\delta:\mathcal{U}\to \mathcal{U}\subset V\times
  Z\)
Define \((y_1,z_1)=\delta(y,z)\) and note that \(z_1=z\) and
  \(y_1=\oplus(\bar{\rho}(H(y),z))\).
By construction \(y_1\in V\) and we see that
  \((H(y_1),z)=(\tau(\bar{\rho}(H(y),z)),z) \in W\) in view of (3) which
  implies property (1) for \((y_1,z)\), i.e.
  \begin{align}
    (H(y_1),z)\in W.
    \end{align}
With \(y_1= \oplus(\bar{\rho}(H(y),z))\) applying \(f\) we obtain 
  \begin{align}\label{EQ_bar_f_stuff}
    \bar{f}(H(y_1))= f(y_1)= \bar{f}(\bar{\rho}(H(y),z))=z,
    \end{align}
  where the final equality follows from the fact that
  \begin{align*}                                                          
    \rho(x,z) = (\bar{\rho}(x, z), z) \in W\cap {\rm Gr}(\bar{f}).
    \end{align*}
Since we have already shown that \((H(y_1), z)\in W\), it then follows
  from equation (\ref{EQ_bar_f_stuff}) that \((H(y_1), z)\in W\cap {\rm
  Gr}(\bar{f})\).
Hence 
  \begin{align}
    \bar{\rho}(H(y_1),z)=H(y_1)\in U,
    \end{align}
  which shows (2). 
Finally using again that  \(\bar{f}(H(y_1))=z\) it follows that
  \((\bar{\rho}(H(y_1),z),z)=(H(y_1),z)\in W\).
Further
  \begin{align}
    \tau(\bar{\rho}(H(y_1),z))=\tau(H(y_1))= H\circ \oplus\circ H(y_1)
    =H(y_1),
    \end{align}
from which we deduce that \((\tau(\bar{\rho}(H(y_1),z)),z)\in W\) proving
  (3).

Thus we have shown that \(\delta:{\mathcal U}\rightarrow {\mathcal U}\).
Moreover 
  \begin{eqnarray}
    \delta\circ \delta (y,z)& =&\delta(\oplus\circ\bar{\rho}(H(y),z),z)\\
    &=& (\oplus\circ \bar{\rho}(H\circ \oplus\circ
    \bar{\rho}(H(y),z),z),z)\nonumber\\
    &=& (\oplus\circ \bar{\rho}(\tau(\bar{\rho}(H(y),z)),z),z).\nonumber
    \end{eqnarray}
The element \((\tau(\bar{\rho}(H(y),z)),z)\) belongs to the graph of
  \(\bar{f}\).
Indeed, we compute 
  \begin{align}
    \bar{f}(\tau(\bar{\rho}(H(y),z)))\notag
    &=f\circ \oplus\circ \tau(\bar{\rho}(H(y),z)))\notag\\
    &=f\circ \oplus(\bar{\rho}(H(y),z)))\notag\\
    &=z\label{EQ_f_bar_z}.
    \end{align}
Hence we obtain the equality
  \begin{align}
    \delta\circ \delta(y,z) = (\oplus\circ \tau (\bar{\rho}(H(y),z)),z)
    =(\oplus(\bar{\rho}(H(y),z)),z)=\delta(y,z).
    \end{align}
If \((y,z)\in {\mathcal U}\) belongs to \(\text{Gr}(f)\) it follows that
  \(z=f(y)= f(\oplus\circ H(y))=\bar{f}(H(y))\) so that
  \(\bar{\rho}(H(y),z)=H(y)\) implying that \(\delta(y,z)=(\oplus\circ
  H(y),z)=(y,z)\).
Hence \(\delta\) fixes \({\mathcal U}\cap \text{Gr}(f)\). 
Given any \((y,z)\in {\mathcal U}\) we define \((y_1,z):=\delta(y,z)\). 
We compute
  \begin{align}
    f(y_1) = f\circ \oplus (\bar{\rho}(H(y),z))=\bar{f}
    (\bar{\rho}(H(y),z)) = z,
    \end{align}
  where the final equality follows from equation (\ref{EQ_f_bar_z}).
Thus we have shown that \((y_1,z)\in {\mathcal U}\cap \text{Gr}(f)\).
Summarizing, we have shown that 
  \begin{align*}                                                          
    \delta:\mathcal{U}\to \mathcal{U}\cap {\rm Gr}(f) \qquad \text{and}
    \qquad \delta\big|_{\mathcal{U}\cap {\rm Gr}(f)} = Id,
    \end{align*}
  and hence \(f\) is submersive.
This completes the proof of Theorem
  \ref{THM_imprinting_submersion_yield_submersions}.
\end{proof}

A type of situation which arises quite frequently in applications is the
  following.
We are given a smooth manifold with boundary and corners \(V\), a
  ssc-manifold \(X\) and a set \(Y\) fitting into the commutative diagram
  \begin{align}
    \begin{CD}
    V\times X   @> \oplus >>  Y\\
    @V P VV        @V p VV\\
    V  @=    V,
    \end{CD}
    \end{align}
  where \(P(v,x)=v\), and we would like to show that \(p:Y\to V\) is
  submersive.
As a consequence of Theorem
  \ref{THM_imprinting_submersion_yield_submersions}, to show that \(V\times
  X\xrightarrow{\oplus} Y\xrightarrow {p} V\) is submersive it is sufficient
  to show that \(p\circ\oplus \) is sc-smooth which is
  obviously the case since \(P=p\circ\oplus\).
Moreover we need \(P:V\times X\rightarrow V\) to be submersive, which
  follows from Proposition \ref{PROP_natural_projection_submersive}.
%
%

The final result of this section is concerned with imprintings and the
  question of tameness.
First however, the following elementary result will be useful.

\begin{exercise} 
  \hfill\\
Assume that \(X\) is a tame M-polyfold and \(V\) a smooth
  finite-dimensional manifold with boundary with corners, and suppose that
  \(p:X\rightarrow V\) is a sc-smooth submersive map. 
Prove that \(p\) is an imprinting.
\end{exercise}
%

\begin{theorem}[tame imprintings and manifold submersions]
  \label{THMXX-2-12}\label{THM_tame_imprintings_manifold_submersions}
  \hfill\\
Assume that \(X\) is a tame M-polyfold and \(V\) a smooth
  finite-dimensional manifold with boundary with corners. 
Suppose that \(p:X\rightarrow V\) is a sc-smooth submersive map 
  and assume that the equality \(d_V(p(x))=d_X(x)\) holds for all \(x\in
  X\) and \(\oplus:X\rightarrow Y\) is an imprinting and
  \(p':Y\rightarrow V\) a surjective map fitting into the commutative
  diagram
  \begin{align}
    \begin{CD}
    X  @> \oplus >> Y\\
    @V p VV   @V p' VV\\
    V @= V
    \end{CD}
    \end{align}
Then the following are true.
\begin{enumerate}                                                         
  \item the induced M-polyfold structure on \(Y\) is tame, 
  \item \(p'\) is sc-smooth and submersive,
  \item \(d_V(p'(y))=d_Y(y)\) for all \(y\in Y\).
  \end{enumerate}
\end{theorem}
%
\begin{proof}
We already know that \(p':Y\rightarrow V\) is submersive, where \(Y\) is
  equipped with the M-polyfold structure induced from the imprinting
  \(\oplus\); see Theorem
  \ref{THM_imprinting_submersion_yield_submersions}.
We also know that by Corollary
  \ref{PROP_degeneracy_inequality_submersions} that the following
  inequality holds.
  \begin{align}
    d_Y(y)\geq d_V(p'(y))\ \ \text{for all}\ y\in Y.
    \end{align}
Pick \(y_0\in Y\) and define \(v_0=p'(y_0)\). 
We find \(H:U(y_0)\rightarrow X\) with the usual properties and define the
  open subset  \(Q= \oplus^{-1}(U(y_0))\) of \(X\) the sc-smooth retraction
  \begin{align}
    \tau: Q\rightarrow Q : \tau:=H\circ \oplus.
    \end{align}
We observe that for \(x\in Q\) we have the identity \(p\circ\tau=p\big|_Q\)
  which follows from the calculation
  \begin{align}
    p\circ\tau(x)= p'\circ \oplus \circ H\circ\oplus(x) = p'\circ
    \oplus(x)=p(x).
    \end{align}
For \(x\in Q\) we therefore compute using \(p\circ \tau= p\big|_Q\) that
  \begin{eqnarray}\label{REQUEST2.6}\label{EQ_degeneracy_equals}
    d_X(x) = d_V(p(x)) = d_V\big(p(\tau(x))\big) =   d_X(\tau(x));
    \end{eqnarray}
  note that the first and last equalities each follow from the hypotheses
  that \(d_V(p(x))=d_X(x)\) for all \(x\in X\).
In other words, in \(Q\subset X\) the retraction \(\tau\) preserves the
  degeneracy index on \(X\).
Define the sc-smooth retract \(O:=\tau(Q)\subset X\).  
Then 
  \begin{eqnarray}\label{meqn2.7}\label{EQ_H_diffeo}
    H:U(y_0)\rightarrow O\ \ \text{and}\ \ \oplus\big|_O:O\rightarrow U(y_0)
    \end{eqnarray}
  are  sc-diffeomorphisms, inverse to each other.  
Consequently, we find that for \(x\in X\) we have
  \begin{align*}
    d_X(x)&\geq d_O(x) &\text{by Proposition \ref{PROP_degen_ind_ineq_sub_m_poly}}
    \\
    &= d_Y(\oplus(x)) &\text{by Proposition \ref{PROP_sc_diffeo_preserves_degen_ind}}
    \\
    &\geq d_V(p'(\oplus(x))) &\text{by Corollary \ref{COR_degen_inequality_submersions_manifolds}}
    \\
    &= d_V(p(x)) &\text{because }p'\circ \oplus = p
    \\
    &=d_X(x), &\text{by equation \ref{EQ_degeneracy_equals}}
    \end{align*}
  which implies that all listed expressions are equal. 
Since the choice of \(y_0\in Y\) was arbitrary we have proved the
  following facts:
  \begin{itemize}
    \item[(1)] 
    \(d_Y(y)=d_V(p'(y))\ \ \text{for all}\ y\in Y.\)
    \item[(2)] 
    \(d_{V}(p(x))=d_X(x)=d_Y(\oplus(x))\) for all \(x\in X.\)
    \end{itemize}
With these facts established, we return to  (\ref{EQ_H_diffeo}), and
  consider \(\tau:Q\rightarrow Q\), \(H\), \(O\), and \(\oplus\big|_O\),
  where \(Q\) is the open neighborhood in \(X\) given by
  \(\oplus^{-1}(U(y_0))\).
Next, we fix \(x\in H(U(y_0)) = \tau(Q)\), and let \(Q'\subset Q\) be a
  open neighborhood of \(x\) for which their exists a chart \(\psi:Q'\to P\)
  with \((P,C, E)\) a tame sc-smooth retract.
Without loss of generality, we shall adjust the \(Q'\) to be smaller if
  needed.
Since \((P,C,E)\) is a tame sc-smooth retract we can pick an
  sc-smooth retraction
  \begin{eqnarray}
    \gamma:W\rightarrow W,
    \end{eqnarray}
  where \(W\subset C\) is open in \(C\) such that \(\gamma(W)=P\) and
  \(d_C(\gamma(w))=d_C(w)\) for \(w\in W\).
Recall that since \((P,C,E)\) is tame, there is an additional property our
  retraction can be chosen to have ( regarding sc-complements of \(T_oP\) at
  smooth points), but we shall not state it presently.
The basic result about tame retractions gives us the
  following nontrivial equality, see Proposition
  \ref{PROP_eq_of_degen_index}
  \begin{eqnarray}
    d_P(x)=d_C(x)\ \text{for all}\ x\in P.
    \end{eqnarray}
  We define \( \varepsilon:W\rightarrow W\) by 
  \begin{eqnarray}\label{EQ_eps}
    \varepsilon:= \psi \circ  \tau  \circ     \psi^{-1}\circ \gamma
    \end{eqnarray}
We note that \(\varepsilon\circ\varepsilon=\varepsilon\)  and
  \(\varepsilon(W) = \psi(O) \).
The data we produced so far fits into the following commutative diagram
  \begin{align}
    \begin{CD}
    Y @.\supset U(y_0) @>   \psi\circ H>>  \varepsilon(W)@.\subset W\\
     @ VV p' V      @.  @. @   V  p\circ\psi^{-1} VV\\
    V @. @=    @. V
    \end{CD}
    \end{align}
  where \(\psi\circ H\) is a sc-diffeomorphism and  \(\psi:Q'\rightarrow
  P\) maps the retract \(O\subset Q'\) sc-diffeomorphically onto the
  retract \(\epsilon(W)\).
Hence, we have the tame retract \((P,C,E)\) with associated sc-smooth
  retraction \(\tau:W\rightarrow W\) satisfying \(\tau(W)=P\) and for
  \(\varepsilon:W\rightarrow W\) with \(R= \varepsilon(W)\) we obtain the
  sc-smooth  retract \((R,C,E)\), where we note that \(R\subset P\).
We shall show that \(\varepsilon\) is an sc-smooth tame retraction.

We first show that \(\varepsilon:W\rightarrow W\) preserves \(d_C\). 
Indeed, using that \((P,C,W)\) is tame, \(\varepsilon(w)\in R\subset P\)
  and using (\ref{EQ_degeneracy_equals}) 
  \begin{align*}
    d_C(\varepsilon(w))&= d_P(\varepsilon(w)) &\text{by Proposition
    \ref{PROP_eq_of_degen_index}}
    \\
    &= d_Q( \tau  \circ \psi^{-1}\circ \gamma(w))&\text{by Proposition
    \ref{PROP_sc_diffeo_preserves_degen_ind}}
    \\
    &= d_X( \tau  \circ     \psi^{-1}\circ \gamma(w))
    \\
    &=d_X(\psi^{-1}\circ\gamma(w)) &\text{by equation
    (\ref{EQ_degeneracy_equals})}
    \\
    &= d_Q(\psi^{-1}\circ\gamma(w))
    \\
    &=d_P(\gamma(w))&\text{by Proposition
    \ref{PROP_sc_diffeo_preserves_degen_ind}}
    \\
    &= d_C(\gamma(w)) &\text{by Proposition \ref{PROP_eq_of_degen_index}}
    \\
    &= d_C(w) &\text{ by tameness}.
    \end{align*}
Hence we have shown that the sc-smooth retraction
  \(\varepsilon:W\rightarrow W\) preserves \(d_C\).
In particular, 
  \begin{align*}                                                          
    d_v(p\circ \oplus(x)) = d_v(p(x)) = d_X(x).
    \end{align*}
In order that \((R,C,E)\) is a tame retract we need to verify the
  additional property which distinguishes a tame retract.
The submersion property of \(p\) and \(p'\), respectively,  will be
  crucial for this argument.
We have the tame retract \((P,C,E)\) and the retract \((R,C,E)\) where the
  latter has the properties that \(R\subset P\) and that it admits
  a sc-smooth retraction \(\varepsilon:W\rightarrow W\) preserving \(d_C\).

The following considerations being local we assume that \(E={\mathbb
  R}^d\oplus W\) and \(C=[0,\infty)^d\oplus W\). 

For each \(o\in C\) we let  \(E_o\) denote the subspace introduced in
  Definition \ref{DEF_Ex}.
Let  \(o\in R\) be a smooth point and consider the tangent space \(T_oR\).
We need to find a sc-complement  \(A\) contained in \(E_o\)  satisfying
  \(A\oplus T_oR=E\).

The consideration being local we may consider the following diagram
  \begin{align}
    \begin{CD}
    P  @> \varepsilon >> R\\
    @V p VV    @ V p' VV\\
    V@= V,
    \end{CD}
    \end{align}
  where \(V\) is an open neighborhood of \(0\in [0,\infty)^d\)  and \(p\)
  and \(p'\) are submersive sc-smooth maps.
We have the already established property that 
  \begin{align}
    d_C(o)=d_P(o)= d_V(p'(o))= d_R(o)\ \ \text{for}\ o\in R.
    \end{align}
Pick a smooth \(o_0\in R\) and define \(v_0=p'(o_0)\).
Since \(p'\) is submersive we find an open neighborhood \(W\) of
  \((o_0,v_0)\in R\times [0,\infty)^d\) and an associated sc-smooth
  retraction \(\rho:W\rightarrow W\) of the usual form
  \(\rho(o,v)=(\bar{\rho}(o,v),v)\) such that \(\rho(W) =W\cap
  \text{Gr}(p')\).
For a suitable open neighborhood \(V'\) of \(v_o\in [0,\infty)^d\) we
  consider the map
  \begin{align*}
    &f:V'\rightarrow R\\
    &f(v)=\bar{\rho}(o_0,v).
    \end{align*}
As was shown in the proof of Proposition
  \ref{PROP_submersions_implied_diffeo}, the image \(\Sigma=f(V')\) is a
  submanifold of \(R\), and \(f:V'\to f(V')\) is a diffeomorphism.
We obtain the equalities
  \begin{align}
    d_{[0,\infty)^d}(v) = d_{[0,\infty)^d}(p'(f(v)))= d_R(f(v))=
    d_C(f(v)),
    \end{align}
  and we shall study \(f\) near \(v_0\).  
Since \(p'\circ f(v)=v\) for \(v\in V'\) we see that the differential
  \(Df(v_0)\) of \(f\) at \(v_0\) is a  injective  sc-operator
  \begin{eqnarray}
    Df(v_0):{\mathbb R}^d\rightarrow E.
    \end{eqnarray}
However, we can say more about \(Df(v_0)\). 
With \(o_0=f(v_0)\) we can write \(o_0=(r^0_1,...,r^0_d,w_0)\), where
  \(w_0\) is smooth.
With \(\bar{d}=d_C(f(v_0))\) we may assume without loss of generality that
  \(r^0_1=..=r^0_{\bar{d}}=0\) and \(0 <r^0_{i}\) for \(i\in
  \{\bar{d}+1,..,d\}\). Since \(d_V(v_0)=d_C(o_0)\)
  we may also assume without loss of generality that
  \(v_0=(v_1^0,...,v^0_d)\) satisfies \(v^0_1=...=v^0_{\bar{d}}=0\) and
  \(v^0_i>0\) for \(i\in \{\bar{d}_1,....d\}\).

For the following we shall introduce the partial quadrants \(\bar{C}:=
  [0,\infty)^{\bar{d}}\times{\mathbb R}^{d-\bar{d}}\) and
  \(\wt{C}=[0,\infty)^{\bar{d}}\oplus {\mathbb R}^{d-\bar{d}}\oplus W\).  
We take \(h\in \bar{C}=[0,\infty)^{\bar{d}}\times {\mathbb R}^{d-\bar{d}}\)
  and observe that for \(0<t\) small \(v_0+th\in [0,\infty)^d\).  
We see that the degeneracy index \(d_{\bar{C}}\) gives us 
  \begin{align}
    d_{\bar{C}}(v_0+th) = \sharp\{i\in \{1,...,\bar{d}\}\ |\ h_i=0\}.
    \end{align}
Then  \(f(v_0+th)\in C\).
We also observe that for the point \(f(v_0+th)\in C\), the coordinates
  with indices  \(i\in \{\bar{d}+1,d\}\) must be positive if \(|t|\) is
  small enough.
Hence for \(0<t\) small
  \begin{align}
    f(v_0+th)-f(v_0)\in \wt{C}.
    \end{align}
Moreover we compute using that \(f\) preserve the degeneracy indices
  (\(d_V\rightarrow d_C\))
  \begin{eqnarray}
    d_{\wt{C}}(f(v_0+th)-f(v_0)) = d_{\wt{C}}(h) = d_{\bar{C}}(h).
    \end{eqnarray}
Since for \(t>0\) 
  \begin{align*}
    d_{\wt{C}}(f(v_0+th)-f(v_0)) =
    d_{\wt{C}}\left(\frac{f(v_0+th)-f(v_0)}{t}\right)
    \end{align*}
  we find taking the limit  
  \begin{align*}
    d_{\wt{C}}(Df(v_0)h)\geq \text{lim}_{t\rightarrow }
    d_{\wt{C}}(f(v_0+th)-f(v_0)),
    \end{align*}
  which implies
  \begin{eqnarray} \label{RESULt1}\label{EQ_result_one}
    d_{\bar{C}}(h) =d_{\wt{C}}(h)\leq d_{\wt{C}}(Df(v_0)h)\ \text{for all}\
    h\in \wt{C}.
    \end{eqnarray}
Here we used the fact that a nonzero vector \(u_0\) in any partial
  quadrant \(D\) has an open neighborhood \(U(u_0)\) in this partial
  quadrant for which \(d_D(u)\leq d_D(u_0)\) for all \(u\in U(u_0)\).

Next consider  the smooth point \(o_0\in R\subset P\) which satisfies
  \(p(o_0)=p'(o_0)=v_0\).   We know that  \(d_C(o_0)=d_V(v_0)\).
We carry out the previous consideration for \(p:P\rightarrow
  [0,\infty)^d\).
Again we use that \(o_0=(0,...,0,r^0_{\bar{d}+1},..,r^0_d,w_0)\), where
  \(r^0_i>0\) for \(i\in \{\bar{d}+1,...,d\}\) and
  \(v_0=(0,...,0,v^0_{\bar{d}+1},..,v^0_d)\), where \(v^0_i>0\) for \(i\in
  \{\bar{d}_1,..,d\}\). 
We take a smooth \(k\in C\) with \(k_i\geq 0\) for \(i\in
  \{1,...,\bar{d}\}\).
Then \(o_0+tk\in \wt{C}\) for \(0<t\) small and \(p(o_0+tk) -p(o_0)\in
  \bar{C}\).
By the same limit argument as before and using that \(p\) preserves
  \((d_C,d_{[0,\infty)^d})\) we obtain  the following inequality first  for
  smooth \(k\) from which we conclude it holds for all \(k\in \wt{C}\).
  \begin{eqnarray}\label{RESULt2}\label{EQ_result_two}
    d_{\wt{C}}(k) \leq d_{\bar{C}}(Dp(o_0)(k))\ \text{for}\ \ k\in \wt{C}.
    \end{eqnarray}
Since \(p\circ f(v)=v\) we find that \(Dp(o_0)\circ Df(v_0)=
  \text{Id}_{{\mathbb R}^d}\).
Combining (\ref{RESULt1}) and (\ref{RESULt2}) we obtain
  \begin{align}
    d_{\bar{C}}(h)=d_{\bar{C}}(Dp(o_0)\circ Df(v_0)(h))\geq  d_{\wt{C}}
    (Df(v_0)h)\geq d_{\bar{C}}(h)
    \end{align}
  implying that all terms are equal. Summarizing we obtain 
  \begin{itemize}
    \item[(1)] 
    \(Df(v_0):{\mathbb R}^d\rightarrow C\) is injective and maps
    \(\bar{C}\) into \(\wt{C}\).
    \item[(2)] 
    \(d_{\wt{C}} (Df(v_0)h)= d_{\bar{C}}(h)\) for all \(h\in \bar{C}\).
    \end{itemize}
In order to show that \((R,C,E)\) is a tame retract we need to show that
  \(T_{o_0}R\) has a sc-complement in \(E_{o_0}= \{0\}^{\bar{d}}\times
  {\mathbb R}^{d-\bar{d}}\times W\).

Let us first show that the image of \(Df(v_0)\) together with \(E_{o_0}\)
  generates \(E\) in the sense
  \begin{eqnarray}\label{GOalP1}\label{EQ_goal_one}
    (Df(o_0){\mathbb R}^d)+ E_{o_0} =E.
    \end{eqnarray}
Denote by \(e^1,...,e^{\bar{d}}\) the first  \(\bar{d}\)-many standard
  basis vectors in \({\mathbb R}^d\).
We define for \(i\in \{1,...,\bar{d}\}\) the vector \(b^i=Df(v_0)e^i\).
Using that \(Df(v_0)\) is \((d_{\bar{C}},d_{\wt{C}})\)-preserving we see
  that all the coordinates of \(b^j\) with \(i\in \{1,..,\bar{d}\}\) are
  vanishing with the exception of one which has to be positive.  
Hence given a vector of the form \(((c_1,....,c_{\bar{d}},0,...,0),0)\in
  E\) there exists \(h\in {\mathbb R}^d\) such that
  \begin{align}
    Df(v_0)(h)-((c_1,....,c_{\bar{d}},0,...,0),0)
    \end{align}
  has the first \(\bar{d}\) components vanishing. 
However, this precisely means that the above element belongs to \(E_{o_0}\).  
We have established the intermediate step (\ref{GOalP1}).
In order to complete the proof we note that we have the obvious
  sc-decomposition
  \begin{align}
    E= (D\varepsilon(o_0)E)\oplus (Id-D\varepsilon(o_0))E = T_{o_0}R\oplus
    (Id-D\varepsilon(o_0))E,
    \end{align}
  where \(\varepsilon\) is provided in equation (\ref{EQ_eps}). 
We claim that \((Id-D\varepsilon(o_0)E\subset E_{o_0}\), which would
  complete the proof.
We have proved already that  \(E= (Df(v_0){\mathbb R}^d)+ E_{o_0}\) and
  know that \(Df(o_0)({\mathbb R}^d)\subset T_{o_0}R\).
Applying \(Id-D\varepsilon(o_0)\)  to the equality
  \begin{align}
    (D\varepsilon(o_0)E)\oplus (Id-D\varepsilon(o_0))E= (Df(o_0){\mathbb
    R}^d)+ E_{o_0}
    \end{align}
  we obtain
  \begin{align}
    (Id-D\varepsilon(o_0))(E) = (Id-D\varepsilon(o_0))(E_{o_0}).
    \end{align}
Hence it suffices to prove that \(D\varepsilon(o_0)(E_{o_0})\subset
  E_{o_0}\).
By assumption \(\varepsilon(o_0)=o_0\) and 
  \begin{align}
    o_0= (0,...,0,r^0_{\bar{d}+1},..,r^0_{d},w_0).
    \end{align}
Pick \(k=(0,...,0,k_{\bar{d}+1},..,k_d,w)\in E_{o_0}\).  
Then for small \(|t|\), we have \(o_0+tk\in C\) and 
  \begin{eqnarray}\label{inXXput-1}
    d_C(\varepsilon(o_0+tk))=d_C(o_0+tk)=d_C(o_0). 
    \end{eqnarray}
Since for \(|t|\) small the components of \(\varepsilon(o_0+tk)\) numbered
  \(\bar{d}+1,...d\) have to be positive by continuity, it follows
  from(\ref{inXXput-1}) that the components numbered \(1,..,\bar{d}\) have
  to vanish.
However this implies that \(\varepsilon(o_0+tk)\) for \(|t|\) belongs to
  \(E_{o_0}\). 
Consequently we have proved that
  \begin{align}
    D\varepsilon(o_0)E_{o_0}\subset E_{o_0}.
    \end{align}
The proof is complete.
\end{proof}
%

%
\section{Operations and the Imprinting Method}
  \label{Operations-oplus}\label{SEC_operations_and_imprinting}
We introduce ideas and results which are very useful in establishing
  a building bloc system which allows analytical results to be recycled.  
This even extends to the sc-Fredholm theory which will be discussed later.

%
\subsection{Operations}\label{qsec3.1} \label{SEC_operations}
Using the previously established results we can take the {\bf product}
  \(\oplus\times\oplus'\) and the {\bf disjoint union}
  \(\oplus\sqcup\oplus'\) of two given imprintings.
Consider an imprinting \(\oplus:X\to Y\), and an
  injective map \(\phi:Y'\to Y\).
Suppose that \(X'=\oplus^{-1}(\phi(Y'))\) is a sub-M-polyfold of \(X\).
Then the following diagram is commutative, and moreover by Theorem 3.12
  the map \(\oplus':=\phi^{-1}\circ \oplus \big|_{X'}\) is an imprinting.
  \begin{align}
    \begin{CD}
    X @>\oplus >>  Y\\
    @A\text{incl} AA   @A\phi AA\\
    X' @> \phi^{-1}\circ \oplus|X' >>   Y'
    \end{CD}
    \end{align}

\begin{definition}[admissible maps]
  \label{DEF_admissible_map}
  \hfill\\
Given an imprinting  \(\oplus:X\rightarrow Y\) and an injective
  map between sets \(\phi:Y'\rightarrow Y\), we say that \(\phi\) is
  \emph{admissible} provided
  \begin{align}
    X':=\oplus^{-1}(\phi(Y))
    \end{align}
  is a sub-M-polyfold of \(X\). In this case we define the {\bf pull-back}
  \(\phi^\ast\oplus \) by
  \begin{align*}
    &\phi^*\oplus: \oplus^{-1}(\phi(Y'))\rightarrow Y'\\
    &\phi^*\oplus=\phi^{-1}\circ\oplus\big|_{(\oplus^{-1}(\phi(Y')))}
    \end{align*}
\end{definition}
%
The following is an easy exercise.

\begin{lemma}[composition of admissible maps]
  \label{LEM_composition_admissible_maps}
  \hfill\\
Assume that \(\oplus:X\rightarrow Y\) is an imprinting and
  \(\phi:Y'\rightarrow Y\) is admissible so that
  \(\oplus'=\phi^{\ast}\oplus\)\index{$\oplus'=\phi^{\ast}\oplus$} is
  defined.
Suppose further \(\psi:Y''\rightarrow Y'\) is admissible for \(\oplus'\)
  defining \(\psi^\ast\oplus'\).
Then \(\phi\circ\psi\) is admissible for \(\oplus\) and naturally
  \((\phi\circ\psi)^\ast\oplus = \psi^\ast(\phi^\ast\oplus)\).
\end{lemma}
%
These are some basic operations which one can carry out to stay within the
  scope of imprintings. 
The playing field can be vastly extended by adding what we call
  restriction maps. 
This is done in the next subsection.

%
\subsection{Restrictions}  \label{qsec3.2}\label{SEC_restrictions}
We start with a definition adding an additional piece of structure to the
  imprinting method.

\begin{definition}[imprinting with restrictions]
  \label{DEF_imprinting_with_restrictions}
  \hfill\\
An \(\oplus\)-\emph{construction with restriction} is a pair
  \((\oplus,\bm{p})\), \index{\((\oplus,\bm{p})\)} where
  \(\oplus:X\rightarrow Y\) is a M-polyfold construction by the
  \(\oplus\)-method, and \(\bm{p}\) is a finite family of maps
 \(p_i:Y\rightarrow A_i\), \(i\in I\), where the \(A_i\) are M-polyfolds
 and the compositions  \(p_i\circ\oplus:X\rightarrow A_i\) are sc-smooth.
\end{definition}
%
We can use \(\oplus :{\mathbb B}\times E\rightarrow
  X^{3,\delta_0}_{\varphi}\) to construct an example.
For a nonzero \(a\in {\mathbb B}\) we have that \(\varphi(|a|)>51\).
Define \(\Sigma^+=[0,R]\times S^1\) and \(\Sigma^-=[-R,0]\times S^1\). 
We have natural inclusions
  \begin{eqnarray*}
    \Sigma^+\xrightarrow{i^+_a} Z_a\xleftarrow{i_a^-} \Sigma^-
    \end{eqnarray*}
  which for \(a=0\) are  obvious and for \(a\neq 0\) are given by
  \(i^+_a(s,t) =\{(s,t),(s',t')\}_a\) and
  \(i^-_a(s',t')=\{(s,t),(s',t')\}_a\).
Denote by 
  \begin{align}
    A^\pm= H^3(\Sigma^\pm,{\mathbb R}^N)
    \end{align}
  the sc-Hilbert spaces where level \(m\) corresponds to regularity
  \(H^{3+m}\).
We obtain the restriction maps
  \begin{align}
    A^+\xleftarrow{p^+} X^{3,\delta}_{\varphi}\xrightarrow{p^-} A^-.
    \end{align}

\begin{proposition}[cylinder gluing is an imprinting with restrictions]
  \label{PROP_cylinder_gluing_is_imprinting_w_rest}
  \hfill\\
The pair \((\oplus,\bm{p})\), where \(\oplus:{\mathbb B}\times
  E\rightarrow X^{3,\delta}_{\varphi}\) is gluing and \(\bm{p}=\{p^x,p^y\}\)
  is a \(\oplus\)-construction with restriction.
\end{proposition}
%
This can be viewed as one of the building-block pieces.  
Shortly we shall discuss the important fact that \(\bm{p}\) is compatible
  with the submersive \(p_{\mathbb B}\).
The following definition will be crucial for fibered product
  constructions.

\begin{definition}[plumable]
  \label{DEF_plumable}
  \hfill\\
Assume that \((\oplus,\bm{p})\) and \((\oplus',\bm{p}')\) are
  imprintings with restrictions, and \(i_0\in I\) and \(i_0'\in
  I'\) are given so  that \(A_{i_0}=A_{i_0'}'\).
We say that \((\oplus,\bm{p})\) and \((\oplus',\bm{p}')\) are
  \((i_0,i_0')\)-\emph{plumbable} provided the subset
  \begin{align}
    X{_{i_0}\times_{i_0'}} X':=\{(x,x')\in X\times X'\ |\
    p_{i_0}\circ\oplus(x)=p_{i_0}'\circ\oplus'(x')\}
    \end{align}
  of \(X\times X'\) is a sub-M-polyfold.  
\end{definition}
%

\begin{remark}
There is an obvious generalization, where instead of \(A_{i_0}=A_{i_0'}'\)
  we just assume that we are given an sc-diffeomorphism
  \begin{align}
    \sigma:A_{i_0}\rightarrow A'_{i_0'}.
    \end{align}
\end{remark}
%

Define  \(Y{_{i_0}\times_{i_0'}} Y'=\{(y,y')\in Y\times Y'\ |\
  p_{i_0}(y)=p_{i_0}'(y')\}\)  and denote by
  \begin{align}
    \phi: Y{_{i_0}\times_{i_0'}} Y'\rightarrow Y\times Y'
    \end{align}
  the inclusion map.  
Take the product \(\oplus\)-construction \(\oplus\times\oplus'\) and
  observe that
  \begin{align}
    (\oplus\times\oplus')^{-1}(\phi( Y{_{i_0}\times_{i_0'}} Y')) =
    X{_{i_0}\times_{i_0'}} X'
    \end{align}
  which by assumption is a sub-M-polyfold.   
Next we aim to define restrictions \(\mathbf{p}''\) for the disjoint union.
To that end, we first define a new index set \(I\) via the following.
  \begin{align*}                                                          
    I'' = \{1\}\times (I\setminus \{i_0\})\; \; \bigcup \; \; \{2\}\times
    (I'\setminus \{i_0'\})
    \end{align*}
Next, for each \(i'' = (i_1'', i_2'') \in I''\), we define
  \begin{align*}                                                          
    p_{i''}'' = 
    \begin{cases}
    p_{i_2''}\circ \pi_1  &\text{if }i_1'' = 1
    \\
    p_{i_2''}' \circ \pi_2 &\text{if }i_1'' = 2
    \end{cases}
    \end{align*}
  where \(\pi_1\) and \(\pi_2\) are the projections from the fibered
  product onto the first and second factor respectively, and hence we
  can define the indexed set of restrictions as follows.
  \begin{align*}                                                          
    \mathbf{p}'' = \{p_{i''}''\}_{i''\in I''}
    \end{align*}

\begin{definition}[plumbing]
  \label{DEF_plumbing}
  \hfill\\
If \((\oplus,\bm{p})\) and \((\oplus',\bm{p}')\) are
  \((i_0,i_0')\)-{plumbable}\index{plumbable} we define the
  imprinting with restriction
  \((\oplus,\bm{p}){_{i_0}\times_{i_0'}}(\oplus',\bm{p}')\) where
  \begin{align}
    \oplus{_{i_0}\times_{i_0'}}\oplus' : =\phi^{\ast}(\oplus\times\oplus')
    \end{align}
  and call it the \((i_0,i_0')\)-\emph{plumbing}\index{plumbing} of
  \((\oplus,\bm{p})\) and \((\oplus',\bm{p}')\).
\end{definition}
%

\begin{remark}
  \hfill\\
Assume that we have three imprintings with restriction, say
  \((\oplus,\bm{p})\), \((\oplus',\bm{p}')\) and \((\oplus'',\bm{p}'')\).
Let \(i_0\in I\), \(i_0', i_1'\in I'\), \(i'_1\neq i'_0\),  and
  \(i''_1\in I''\).
Suppose that \(A_{i_0}=A'_{i_0'}\), \(A'_{i'_1}=A''_{i_1''}\) so that we
  obtain the following diagram
  \begin{align}
    \begin{CD}
    Y@> p_{i_0}>> A_{i_0}\equiv 'A_{i'_0}@< p'_{i_0'}<< Y' @> p'_{i'_1}>>
    A'_{i_1'}\equiv A''_{i''_1} @< p''_{i''_0} <<  Y''
    \end{CD}
    \end{align}
Assume that \((\oplus,\bm{p})\) and \((\oplus',\bm{p}')\) are
  \((i_0,i'_0)\)-plumbable and \((\oplus',\bm{p}')\) and
  \((\oplus'',\bm{p}'')\) are \((i'_1,i_1'')\)-plumbable.
In this case  one can show that 
  \begin{itemize}
  \item[]
  \((\oplus,\bm{p}){_{i_0}\times_{i_0'}}(\oplus',\bm{p}')\) and
  \((\oplus'',\bm{p}'')\) are \((i'_1,i_1'')\)-plumbable.
  \item[]  
  \((\oplus,\bm{p})\) and
  \((\oplus',\bm{p}'){_{i'_1}\times_{i_1''}}(\oplus'',\bm{p}'')\) are
  \((i_0,i_0')\)-plumbable.
  \end{itemize}
Considering either way we obtain the same sets and in fact the same spaces
  due to the fact that the fibered product operation is well-behaved, see
  the next subsection.
\end{remark}
%
For imprintings with restrictions \((\oplus,\bm{p})\) and
  \((\oplus',\bm{p}')\) we can define first the disjoint union
  \(\oplus\sqcup \oplus'\).
For a pair \((i,i')\in I\times I'\) we define \(p''_{(i,i')}: Y\sqcup Y'\)
  by \(p''_{(i,i')}(y)=p_i(y)\) for \(y\in Y\) and
  \(p''_{(i,i')}(y')=p_{i'}'(y')\). 
Then we call \((\oplus \sqcup,\oplus',\bm{p}'')\) the disjoint union of
  \((\oplus,\bm{p})\) and \((\oplus',\bm{p}')\) and write it as
  \((\oplus,\bm{p})\sqcup(\oplus',\bm{p}')\).
The following is obvious.

\begin{theorem}[disjoint union of imprintings]
  \label{THM_disjoint_union_of_imprintings}
  \hfill\\
The disjoint union of two imprintings with restrictions is
  an imprinting with restrictions.
\end{theorem}
%

%
\subsection{Imprinting-Submersions with Restrictions}
  \label{SEC_imprinting_submersions_w_restr}
In our upcoming constructions we shall frequently see
  imprinting-submersions of the form
  \begin{align}
    X\xrightarrow{\oplus} Y\xrightarrow{f} Z
    \end{align}
In order to plumb such triples, these objects will also come with
  restrictions, i.e.
  \begin{align}
    p_i:Y\rightarrow A_I,\ \ i\in I.
    \end{align}
When we carry out a plumbing of two such gadgets we would like to make
  sure that the plumbed \(\oplus\)-construction is again submersive.
To be able to conclude this we need some compatibility conditions.  
This is formalized in the following definition

\begin{definition}[imprinting-submersion with restrictions]
  \label{DEF_imprinting-submersion_w_rest}
  \hfill\\
Consider a triple \((\oplus,\bm{p},f)\), where \((\oplus,\bm{p})\) is an
  imprinting with restrictions \((\oplus,\bm{p})\),
  say \(\oplus :X\rightarrow Y\) and \(p_i:Y\rightarrow A_i\) for \(i\in
  I\), and \(f:Y\rightarrow Z\) a surjective map onto a M-polyfold such that
  \(f\circ\oplus :X\rightarrow Z\) is submersive.
We shall say that \((\oplus,\bm{M},f)\) is a \emph{imprinting-submersion
   with restrictions}\index{imprinting-submersion with restrictions}
   provided with \(\bar{f}=f\circ\oplus\), \(\bar{p}_i=p_i\circ \oplus\)
   the following additional compatibility condition holds.
For every \(x_0\in X\) with \(z_0=\bar{f}(x_0)\) there exists an open
  neighborhood \(W\) of \((x_0,z_0)\)  in \(X\times Z\) and a sc-smooth
  retraction \(\rho:W\rightarrow \rho\) of the form
  \(\rho(x,z)=(\bar{\rho}(x,z),z)\) it holds
  \begin{align}
    \bar{p}_i\circ \bar{\rho}(x,z) = \bar{p_i}(x),\ \ i\in I,\ (x,z)\in W.
    \end{align}
\end{definition}
%

Going back to our gluing example one can show the following.

\begin{proposition}[gluing is an imprinting-submersion with restrictions]
  \label{PROP_gluing_is imprinting_submersion_w_restr}
The triple \((\oplus,\bm{p},p_{\mathbb B})\), where \(\oplus:{\mathbb
  B}\times E\rightarrow X^{3,\delta}_{\varphi}\) is the previously given
  imprinting, \(\bm{p}=\{p^x,p^y\}\) and \(p_{\mathbb
  B}:X^{3,\delta_0}_{\varphi}\rightarrow {\mathbb B}\) is the domain
  parameter extraction,  is an imprinting-submersion  with restrictions.
\end{proposition}
%
Assume that \((\oplus,\bm{p},f)\) and \((\oplus',\bm{p}',f')\) are given
  and \((\oplus,\bm{p})\) and \((\oplus',\bm{p}')\) are
  \((i_0,i_0')\)-plumbable.
Then we obtain the new \(\oplus\)-construction with restrictions
  \((\oplus,\bm{p}){_{i_0}\times_{i_0'}}(\oplus',\bm{p}')\) as already
  shown.
We can define the map \(f'':  Y{_{i_0}\times_{i_0'}}Y'\rightarrow Z\times
  Z'\) by
  \begin{align}
    f''(y,y')=(f(y),f'(y')) = f{_{i_0}\times_{i_0'}}f'.
    \end{align}
The compatibility condition from the previous definition implies that
  \(f''\) is submersive for the induced M-polyfold structure or
  alternatively \(\bar{f}'':= f''\circ
  (\oplus{_{i_0}\times_{i_0'}}\oplus')\) is submersive. 
In fact, it also satisfies the compatibility condition.  
Hence we can make the following definition.

\begin{definition}[fiber product of imprinting-submersions with restrictions]
  \label{DEF_fiber_product_of_imp_submer}
Given two imprinting-submersions with restrictions
  \((\oplus,\bm{p},f)\) and \((\oplus',\bm{p}',f')\) and \(i_0\in I\),
  \(i_0'\in I'\) the \((i_0,i_0')\)-fibered product is defined by
  \begin{align}
     (\oplus,\bm{p},f){_{i_0}\times_{i_0'}}(\oplus',\bm{p}',f'
     :=(\oplus{_{i_0}\times_{i_0'}}\oplus',
     \bm{p}{_{i_0}\sqcup_{i_0'}}\bm{p}',f{_{i_0}\times_{i_0'}}f').
    \end{align}
It is again an imprinting-submersion  with restrictions.
\end{definition}
%
This result allows one to plug imprinting-submersions together. 
Another important piece for SFT is concerned with splitting of cylinders
  along periodic orbits.

%
\subsection{Some Results About Fibered Products}
\label{qsec3.3}\label{SEC_fiber_products}
Given a diagram \(X\xrightarrow{q} A\xleftarrow{p} Y\) of sc-smooth maps
  between M-polyfolds,  we can build (as a set)  the fibered product
  \(X{_{q}\times_p}Y\subset X\times Y\), which consists of all pairs
  \((x,y)\) with \(q(x)=p(y)\).

\begin{definition}[M-polyfold fiber product]
  \label{DEF_m_polyfold_fiber_product}
  \hfill\\
We say \(X{_{q}\times_p}X'\) is a M-polyfold provided the subset
  \(X{_{q}\times_p}X'\) of the product M-polyfold \(X\times X'\)
  is a sub-M-polyfold and hence has an induced M-polyfold structure, i.e.
  the one we take.
\end{definition}
%
We sometimes write \(X\times_A X'\) instead of \(X{_{q}\times_p}X'\).
The following is the expected result about iterated fibered product. 
For example given
  \begin{align}
    X\xrightarrow{q} A\xleftarrow{p} X'\xrightarrow{q'}
    A'\xleftarrow{p'}X''
    \end{align}
  we obtain the subset \(X{_{q}\times_p} X'{_{q'}\times_{p'}} X''\subset
  X\times X'\times X''\) and this may or may not be a M-polyfold.
There are two interesting situations in the case of inductive
  constructions.
For example one might show first that \(X{_{q}\times_p} X'\) is a
  M-polyfold and view the set \(X{_{q}\times_p} X'{_{q'}\times_{p'}} X''\)
  as a subset of the M-polyfold \((X{_{q}\times_p} X')\times X''\).
A priori it seems possible our subset might be a sub-M-polyfold of the
  latter but not of the first and even if it is, the induced structures
  might be different. 
However, this is not case as the following result shows.  
The other case is that we consider our fibered product set to be a subset
  of \(X\times(X'{_{q'}\times_{p'}}X'')\).  

\begin{proposition}[fiber products of fiber products]
  \label{propu3.10}\label{PROP_fiber_products_of_fiber_products}
  \hfill\\
Assume that \(X,X',X''\) and \(A,A'\) are M-polyfolds and
  \begin{align}
    X\xrightarrow{q} A\xleftarrow{p} X'\xrightarrow{q'}
    A'\xleftarrow{p'}X''
    \end{align}
  is a diagram of sc-smooth maps.  
\begin{itemize}
  \item[(1)] 
  If the fibered product \(X\times_A X'\) is a M-polyfold and \((X\times_A
  X') {_{q'\circ \pi_2}\times}_{p'} X''\) has an induced M-polyfold
  structure from \((X\times_A X')\times X''\) then \(X\times_A
  X'\times_{A'} X''\) has an induced M-polyfold structure from \(X\times
  X'\times X''\) and the two M-polyfold structures coincide.
  \item[(2)]   
  If the fibered product \(X'\times_{A'} X''\) is a M-polyfold and the
  associated diagram \(X{_{q}\times_{p\circ\pi_1}} (X'\times_{A'} X'') \)
  has an induced M-polyfold structure from \((X\times_A X')\times X''\)
  then \(X\times_A X'\times_{A'} X''\) has an induced M-polyfold structure
  from \(X\times X'\times X''\) and the two M-polyfold structures coincide.
  \end{itemize}
\end{proposition}
%
\begin{proof}
Since the proofs of (1) and (2) are essentially the same, we just prove
  (1).
Let \((x,x')\in X\times_A X'\). 
By assumption there exists an open neighborhood \(U\) of \((x,x')\) in
  \(X\times X'\) and an sc-smooth \(r:U\rightarrow U\) with \(r\circ r=r\)
  and \(r(U)=U\cap (X\times_A X')\).
Moreover, by assumption, there exists for \(((x,x'),x'')\in (X\times_A
  X'){_{q'\circ\pi_2}\times_{p'}} X''\) an open neighborhood \(W\) in \(
  (X\times_A X')\times X''\) and an sc-smooth retraction \(s:W\rightarrow
  W\) with \(s(W)= W\cap ((X\times_A X'){_{q'\circ\pi_2}\times_{p'}}
  X'')\). 
We define an sc-smooth map \(\tau\) near \((x,x',x'')\) by
  \begin{align}
    (z,z',z'')\rightarrow (r(z,z'),z'')\rightarrow s(r(z,z'),z'').
    \end{align}
Then \(\tau\) is an sc-smooth retraction defined near \((x,x',x'')\). 
Consider the inclusion map
  \begin{eqnarray}\label{p:"}\label{EQ_double_product}
   (X\times_A X')\times X''\rightarrow X\times X'\times X''.
\end{eqnarray}
It the product of the inclusion \(X\times_A X'\rightarrow X\times X'\) and
  \(Id_{X''}\).
The first map is sc-smooth since it is the restriction of \(Id_{X\times
  X'}\) to a sub-M-polyfold.
This shows that (\ref{p:"}) is sc-smooth.
 
Next consider the M-polyfold \(X\times_A X'\times_{A'} X''\) of \(X\times
  X'\times X''\).
Consider the inclusion \(X\times_A X'\times_{A'} X''\rightarrow (X\times_A
  X') {_{q'\circ \pi_2}\times}_{p'} X''\).
The target is a sub-M-polyfold of \((X\times X')\times X''=X\times
  X'\times X''\) and therefore this map is sc-smooth.
This competes the proof.
\end{proof}

The discussion in this subsection applies equally well to strong bundles
  over M-polyfolds \(q:V\rightarrow X\).
Starting with a commutative diagram of strong bundle maps
  \begin{align}
    \begin{CD}
    V @> \Phi >> W @ <\Psi <<V'\\
    @V q VV @V p VV   @V q' VV\\
    X @>\phi>> Y @<\psi << X'\,
    \end{CD}
    \end{align}
  which we sometimes just write as
  \begin{align}
    q\xrightarrow{\Phi} p \xleftarrow{\Psi} q',
    \end{align}
  we can build  he fibered products (as just sets) and obtain
  \begin{align}
    V\times_W V'\xrightarrow {\Phi\times_W\Psi} X\times_Y X'.
    \end{align}
Clearly \(V\times_W V'\) is a subset of the strong bundle \(V\times V'\)
  and we say that the fibered product is defined provided \(V\times_W V'\)
  is a strong sub-bundle.
This also will imply that \(X\times_Y X'\) is a sub-M-polyfold of
  \(X\times X'\) and we obtain the strong bundle written as \(V\times_W
  V'\xrightarrow {\Phi\times_W\Psi} X\times_Y X'\).
We also note that taking the fibered product involving more strong
bundles, the procedure is associative similarly as described in the
M-polyfold case treated in Proposition \ref{propu3.10}.

%
\section{Strong Bundles and the Imprinting Method}
The results in Section \ref{SEC_imprinting_method} and Section
  \ref{SEC_operations_and_imprinting} can be easily generalized to strong
bundle constructions following a few basic principles.
We note however, that because the objects of interest have more structure
  (they are strong (vector) bundles with projections to the base, etc), it
  will be convenient to move to a more categorical framework.

%
\subsection{The category \texorpdfstring{\(
  \mathbf{VBL}_{\mathbb{K}}\)}{VBL} and the Imprinting Method}
  \label{qsec4.1}\label{SEC_vbl_imprinting}

Here and throughout, we let \({\mathbb K}\) denote either the real or
  complex numbers.
That is, all the following results are to be understood with the same
  choice of either \(\mathbb{R}\) or \(\mathbb{C}\) for \(\mathbb{K}\). 
Our next task is to define the category \(\mathbf{VBL}_{\mathbb{K}}\),
  which is meant to be an acronym for ``vector bundle like'' sets; see
  Definitions \ref{DEF_vbl_object}, \ref{DEF_vbl_morphism}, and
  \ref{DEF_category_vbl}.


\begin{definition}[vector-bundle-like set]
  \label{DEF_vbl_set}\label{DEF_vbl_object}
  \hfill\\
A \emph{vector-bundle-like set}\index{vector-bundle-like set}
  \(\underline{\mathsf{Y}}\) consists of a triple \(\underline{\mathsf{Y}} =
  (\mathsf{Y}, Y, {\rm pr}_{\mathsf{Y}})\), where \(\mathsf{Y}\) and \(Y\)
  are sets, \({\rm pr}_{\mathsf{Y}}: \mathsf{Y}\to Y\) is a surjective map,
  and every fiber \({\rm pr}_{\mathsf{Y}}^{-1}(y)\) is equipped with the
  structure of a vector space over \({\mathbb K}\).
\end{definition}
%


\begin{remark}
A strong M-polyfold bundle over an M-polyfold, for example \({\rm
  pr}_{\mathsf{X}}:\mathsf{X}\rightarrow X\), has an underlying structure of
  a vector-bundle-like set as \((\mathsf{X}, X, {\rm pr}_{\mathsf{X}})\).  
\end{remark}
%




With the objects established, we now aim to define the morphisms.

\begin{definition}[morphism between vector-bundle-like sets]
  \label{DEF_morphism_btwn_vbl_sets}\label{DEF_vbl_morphism}
  \hfill\\
  \index{morphisms between vector-bundle-like sets}
Let \(\underline{\mathsf{Y}}= (\mathsf{Y}, Y, {\rm pr}_{\mathsf{Y}})\) and
  \(\underline{\mathsf{Y}}= (\mathsf{Y}', Y', {\rm pr}_{\mathsf{Y}'})\)
  denote vector-bundle-like sets over \(\mathbb{K}\).
A \emph{morphism between} \(\underline{\mathsf{Y}}\) and
  \(\underline{\mathsf{Y}}'\) is a map \(\Phi:\mathsf{Y}\rightarrow
  \mathsf{Y}'\) which \(\mathbb{K}\)-linearly maps fibers to
  fibers.
As such, it uniquely determines a map, denoted \(\phi:Y\to Y'\), which
  fits into the commutative diagram below
  \begin{eqnarray}\label{DIAgram1}\label{EQ_diagram_a}
    \begin{CD}
    \mathsf{Y} @>\Phi>> \mathsf{Y}'\\
    @V{\rm pr}_{\mathsf{Y}} VV   @V {\rm pr}_{\mathsf{Y}'} VV\\
    Y @>\phi>>  Y'.
    \end{CD}
    \end{eqnarray}
\end{definition}
%

With the objects and morphisms defined as above, we provide the natural
  definition of the associated category.
\begin{definition}[the category \(\mathbf{VBL}_{\mathbb{K}}\)]
  \label{DEF_category_vbl}
  \hfill\\
  \index{\(\mathbf{VBL}_{\mathbb{K}}\), the category}
The category with vector-bundle-like sets as objects, as provided in
  Definition \ref{DEF_vbl_object}, and with morphisms as fiber-wise
  \(\mathbb{K}\)-linear maps as in Definition
  \ref{DEF_vbl_morphism}, we shall call \(\mathbf{VBL}_{\mathbb{K}}\).
\end{definition}
%


Next we aim to define a type of product and coproduct structure on
  \(\mathbf{VBL}_{\mathbb{K}}\); we make these notions rigorous with the
  following definition.

\begin{definition}[product and coproduct in \(\mathbf{VBL}_{\mathbb{K}}\)]
  \label{DEF_product_coproduct_vbl}
  \hfill \\
Consider \(\underline{\mathsf{Y}}, \underline{\mathsf{Y}}'\in
  {\rm Ob}(\mathbf{VBL}_{\mathbb{K}})\), and define the triple 
  \begin{align*}                                                          
    (\mathsf{Y}\times \mathsf{Y}', Y\times Y' , {\rm
    pr}_{\mathsf{Y}}\times {\rm pr}_{\mathsf{Y}'})
    \end{align*}
  where \({\rm pr}_{\mathsf{Y}}\times {\rm pr}_{\mathsf{Y}'}
  \mathsf{Y}\times \mathsf{Y}'\to  Y\times Y' \) is the natural
  projection.
Observe that for each \((y,y')\in Y\times Y'\) the set 
  \(({\rm pr}_{\mathsf{Y}}\times {\rm pr}_{\mathsf{Y}'})^{-1}(y, y') \) has
  the structure of a vector space over \(\mathbb{K}\).
Consequently, the above triple is in \({\rm
  Ob}(\mathbf{VBL}_{\mathbb{K}})\), and hence we define the
  \(\mathbf{VBL}_{\mathbb{K}}\)-product via the following
  \begin{align}\label{EQ_vbl_product}                                     
    \underline{\mathsf{Y}}\times \underline{\mathsf{Y}}':=
    (\mathsf{Y}\times \mathsf{Y}', Y\times Y' , {\rm
    pr}_{\mathsf{Y}}\times {\rm pr}_{\mathsf{Y}'})\in {\rm
    Ob}(\mathbf{VBL}_{\mathbb{K}}).
    \end{align}
Similarly we define a \(\mathbf{VBL}_{\mathbb{K}}\)-coproduct structure via
  \begin{align}\label{EQ_vbl_coproduct}                                   
    \underline{\mathsf{Y}}\sqcup \underline{\mathsf{Y}}':=
    (\mathsf{Y}\sqcup \mathsf{Y}', Y\sqcup Y' , {\rm
    pr}_{\mathsf{Y}}\sqcup {\rm pr}_{\mathsf{Y}'})\in {\rm
    Ob}(\mathbf{VBL}_{\mathbb{K}});
    \end{align}
  here \(\sqcup\) denotes the disjoint union.
\end{definition}
%

Similarly, we define a fiber product.

\begin{definition}[fiber products in \(\mathbf{VBL}_{\mathbb{K}}\)]
  \label{DEF_fiber_product_vbl}
  \hfill\\
Given \(\underline{\mathsf{Y}}, \underline{\mathsf{Y}}',
  \underline{\mathsf{B}} \in \mathbf{VBL}_{\mathbb{K}}\) and \(\Phi,
  \Phi'\in {\rm Mor}(\mathbf{VBL}_{\mathbb{K}})\) with
  \(\Phi:\underline{\mathsf{Y}}\to \underline{\mathsf{B}}\) and
  \(\Phi':\underline{\mathsf{Y}}'\to \underline{\mathsf{B}}\), we define
  the associated fiber product to be the following triple
  \begin{align*}                                                          
    \underline{\mathsf{Y}}{_\Phi \times_{\Phi'}} \underline{\mathsf{Y}}' =
    ( \mathsf{Y}{_\Phi \times_{\Phi'}} \mathsf{Y}', Y{_\phi
    \times_{\phi'}} Y' , {\rm pr}_{\mathsf{Y}}{_\Phi \times_{\Phi'}} {\rm
    pr}_{\mathsf{Y}'})
    \end{align*}
  where
  \begin{align*}                                                          
    \mathsf{Y}{_\Phi \times_{\Phi'}} \mathsf{Y}'&= \big\{ (\mathsf{y},
    \mathsf{y}')\in \mathsf{Y}\times \mathsf{Y}' : \Phi(\mathsf{y})=
    \Phi'(\mathsf{y}')\big\}
    \\
    Y{_\phi \times_{\phi'}} Y' &=\big\{ (y, y')\in Y\times Y' : \phi(y)=
    \phi'(y')\big\}
    \\
    {\rm pr}_{\mathsf{Y}}{_\Phi \times_{\Phi'}} {\rm pr}_{\mathsf{Y}'}&=
    {\rm pr}_{\mathsf{Y}} \times {\rm pr}_{\mathsf{Y}'}\big|_{\mathsf{Y}
    {_\Phi \times_{\Phi'}} \mathsf{Y}'} : \mathsf{Y} {_\Phi
    \times_{\Phi'}} \mathsf{Y}' \to Y {_\phi \times_{\phi'}} Y'.
    \end{align*}
Here, as our notation suggests, we have 
\begin{align*}                                                            
  \underline{\mathsf{Y}} &= (\mathsf{Y}, Y, {\rm pr}_{\mathsf{Y}})\\
  \underline{\mathsf{Y}}' &= (\mathsf{Y}', Y', {\rm pr}_{\mathsf{Y}'})\\
  \underline{\mathsf{B}} &= (\mathsf{B}, B, {\rm pr}_{\mathsf{B}})
  \end{align*}
  and 
  \(\phi:Y\to B\) and \(\phi':Y'\to B'\) are the maps induced from
  \(\Phi:\mathsf{Y}\to \mathsf{B}\) and \(\Phi':\mathsf{Y}'\to
  \mathsf{B}'\) respectively.
We note that for each \((y, y')\in Y{_\phi \times_{\phi'}} Y'\), the set
  \(({\rm pr}_{\mathsf{Y}}{_\Phi \times_{\Phi'}} {\rm
  pr}_{\mathsf{Y}'})^{-1}(y, y')\)  has the structure of a vector space over
  \(\mathbb{K}\) and hence we have
  \begin{align*}                                                          
    \underline{\mathsf{Y}}{_\Phi \times_{\Phi'}} \underline{\mathsf{Y}}'
    \in {\rm Mor}(\mathbf{VBL}_{\mathbb{K}}).
    \end{align*}
\end{definition}
%

Suppose \(\Phi\in {\rm Mor}(\mathbf{VBL}_{\mathbb{K}})\) with \(\Phi:
  \underline{\mathsf{Y}} \to \underline{\mathsf{Y}}'\) .
In the following we shall say \(\Phi\) is {\bf surjective} provided the map
  \(\Phi:\mathsf{Y} \rightarrow \mathsf{Y}'\) is surjective, and we call it
  {\bf injective} provided \(\Phi:\mathsf{Y}\rightarrow \mathsf{Y}'\) is
  injective.
Given a vector-bundle-like set \(\underline{\mathsf{Y}} = (\mathsf{Y}, Y,
  {\rm pr}_{\mathsf{Y}}) \in {\rm Ob}(\mathbf{VBL}_{\mathbb{K}})\),
  and a subset \(V\subset Y\) we obtain a new vector-bundle-like set 
  \(\underline{\mathsf{V}} = (\mathsf{V}, V, {\rm pr}_{\mathsf{V}})\)
  defined by 
  \begin{align*}                                                          
    \mathsf{V}:={\rm pr}_{\mathsf{Y}}^{-1}(V)\qquad\text{and}\qquad {\rm
    pr}_{\mathsf{V}}:= {\rm pr}_{\mathsf{Y}}\big|_{\mathsf{V}}.
    \end{align*}
Finally, we note that given a vector-bundle-like set \((\mathsf{Y}, Y,
  {\rm pr}_{\mathsf{Y}})\), and a map \(f:Y'\rightarrow Y\), the associated
  {\bf pull-back}  \((f^*\mathsf{Y}, Y', f^*{\rm pr}_{\mathsf{Y}})\) is a
  well-defined vector-bundle-like set.

\begin{remark}
We note that terminology can get rather heavy with descriptions, so for
  concision, we establish two abbreviations which we shall henceforth use.
\begin{enumerate}                                                         
  \item VBL shall be short for ``vector-bundle-like,'' and
  \item SB shall be short for ``strong M-polyfold bundle.''
  \end{enumerate}
In this way, the objects of of \(\mathbf{VBL}_{\mathbb{K}}\) are VBL sets,
  which will sometimes have a SB structure.
\end{remark}
%


Next we define the imprinting method for strong M-polyfold bundles, which
  is a means to equip a VBL-set with a SB structure. 
There are two items to note, the first of which is that much of what
  follows is a straightfoward generalization of the imprinting method for
  M-polyfolds as provided in Section \ref{SEC_imprinting_method}.
Second is that, although we do not carry it out here, it may be natural to
  consider a category \(\mathbf{SB}\) whose objects are elements of
  \({\rm Ob}(\mathbf{VBL}_{\mathbb{K}})\) that have been equipped with a
  strong M-polyfold bundle (SB) structure, and whose morphisms are
  elements of \({\rm Mor}(\mathbf{VBL}_{\mathbb{K}})\) which are strong
  M-polyfold bundle maps.
In this way, one easily sees that there is a forgetful functor
  \(\mathbf{SB}\to \mathbf{VBL}_{\mathbb{K}}\).
Moreover, given \(\underline{\mathsf{X}}\in {\rm
  Ob}(\mathbf{VBL}_{\mathbb{K}})\cap {\rm Ob}(\mathbf{SB})\), and
  \(\underline{\mathsf{Y}}\in {\rm Ob}(\mathbf{VBL}_{\mathbb{K}})\), an SB
  imprinting (see Definition \ref{DEF_smpb_imprinting} below) is a morphism
  \(\oplus_\triangleleft\in {\rm Mor}(\mathbf{VBL}_{\mathbb{K}})\) of the
  form \(\oplus_\triangleleft: \underline{\mathsf{X}}\to
  \underline{\mathsf{Y}}\) which induces induces an SB structure on the
  target \(\underline{\mathsf{Y}}\) in such a way that
  \(\oplus_\triangleleft \in {\rm Mor}(\mathbf{SB})\).

\begin{definition}[SB imprinting]
  \label{DEF_smpb_imprinting}\label{DEFFF4.3}
  \hfill\\
Let \(\underline{\mathsf{X}}, \underline{\mathsf{Y}}\in {\rm
  Ob}(\mathbf{VBL}_{\mathbb{K}})\), and suppose that
  \(\underline{\mathsf{X}}\)
  has the structure of a strong M-polyfold bundle (SB).
Suppose further that \(\oplus_\triangleleft\in {\rm
  Mor}(\mathbf{VBL}_{\mathbb{K}})\) with \(\oplus_\triangleleft:
  \underline{\mathsf{X}}\to \underline{\mathsf{Y}} \).
We say \(\oplus_\triangleleft\) is an \emph{SB-imprinting}
  provided the following hold.
\begin{enumerate}                                                         
  \item
  \(\oplus_{\triangleleft}\) is surjective
  \item 
  The quotient topology \(\mathcal{T}_{\mathsf{Y}}\) on \(\mathsf{Y}\)
  with respect to \(\oplus_\triangleleft : \mathsf{X}\to \mathsf{Y}\) is
  metrizable.
  \item
  For each \(y\in Y\), there exists \(V\subset Y\) containing \(y\), such
  that \(\mathsf{V}:={\rm pr}_{\mathsf{Y}} ^{-1}(V)\in
  \mathcal{T}_{\mathsf{Y}}\), and there exists \(H_\triangleleft\in {\rm
  Mor}(\mathbf{VBL}_{\mathbb{K}})\) with \(H_\triangleleft:
  \underline{\mathsf{V}}\to \underline{\mathsf{X}}\)
  \begin{enumerate}                                                         
    \item 
    \(\oplus_{\triangleleft}\circ H_{\triangleleft}(v)=v\) for \(v\in
    \mathsf{V}={\rm pr}_{\mathsf{Y}} ^{-1}(V)\),
    \item 
    \(H_{\triangleleft}\circ\oplus_{\triangleleft}:
    \oplus_{\triangleleft}^{-1}(\mathsf{V})\rightarrow
    \mathsf{X}\) is an sc-smooth SB-map (sometimes called
    sc\(_\triangleleft\)-smooth map);
    \end{enumerate}
  \end{enumerate}
  here, as our notation suggests, \(\underline{\mathsf{V}} = (\mathsf{V},
  V, {\rm pr}_{\mathsf{V}}) \) where \({\rm pr}_{\mathsf{V}} := {\rm
  pr}_{\mathsf{Y}}\big|_{\mathsf{V}}\).
\end{definition}
%

Observe that for \(\oplus_\triangleleft\in {\rm
  Mor}(\mathbf{VBL}_{\mathbb{K}})\) with
  \(\oplus_\triangleleft:\underline{\mathsf{X}}\to
  \underline{\mathsf{Y}}\), we immediately obtain an underlying map
  \(\oplus:X\rightarrow Y\), which necessarily fits into the following
  commutative diagram:
  \begin{align*}                                                          
    \begin{CD}
      \mathsf{X} @>\oplus_\triangleleft >> \mathsf{Y}\\
      @V{\rm pr}_{\mathsf{X}} VV   @V {\rm pr}_{\mathsf{Y}} VV\\
      X @>\oplus >>  Y.
      \end{CD}
    \end{align*}
Denote by \({\mathcal T}_\oplus\) the quotient topology on \(Y\) with respect
  to \(\oplus:X\rightarrow Y\).
Note that \(V\in {\mathcal T}_\oplus\) if and only if \({\rm
  pr}_{\mathsf{Y}}^{-1}(V)\in {\mathcal T}_{\mathsf{Y}}\).
For the equivalence one needs to use the fact that \({\rm
  pr}_{\mathsf{X}}:\mathsf{X}\rightarrow X\) is a continuous and open map.
\\

\begin{exercise} 
  \label{EXRC10}\label{EX_10a}
  \hfill\\
Assume that \(\oplus_{\triangleleft}:\underline{\mathsf{X}}\rightarrow
  \underline{\mathsf{Y}}\) is a SB imprinting in the sense of
  Definition \ref{DEF_smpb_imprinting}.
  \begin{enumerate}                                                         
    \item 
    Show that the quotient topology \({\mathcal T}_\oplus\) on \(Y\)  with
      respect to the map \(\oplus:X\rightarrow Y\) is metrizable.
    Also show that \({\mathcal T}_\oplus\) equals the quotient topology
      associated to the map \({\rm pr}_{\mathsf{Y}} :\mathsf{Y}\rightarrow Y\).
    \item 
    Show further that  \({\rm pr}_{\mathsf{Y}}:(\mathsf{Y},{\mathcal
      T}_{\oplus_{_\triangleleft}})\rightarrow (Y,{\mathcal T}_\oplus)\)
      is continuous and open.
    \item
    Prove that \(\oplus:X\to Y\) is an M-polyfold imprinting.
    \end{enumerate}
\end{exercise}
%

The next result generalizes Theorem \ref{oplusmethod} to the \({\rm
  VBL}\)-set context.
Some of the conclusions follows from the Exercise \ref{EX_10a}.

\begin{theorem}[consequences of strong bundle imprinting]
  \label{THMVV4.4}\label{THM_consequences_strong_bundle_imprinting}
  \hfill\\
Let \(\oplus_{\triangleleft}:\underline{\mathsf{X}}\rightarrow
  \underline{\mathsf{Y}}\) be an SB-imprinting in the sense of
  Definition \ref{DEF_smpb_imprinting}, and let \(\oplus:X\to
  Y\) be the induced map.
Then the following hold.
\begin{enumerate}                                                         
  \item 
  The quotient topology \({\mathcal T}_\oplus\) on \(Y\) associated to
    \(\oplus\) is precisely the finest topology on \(Y\) for which \({\rm
    pr}_{\mathsf{Y}}:(\mathsf{Y},{\mathcal
    T}_{\oplus_{_\triangleleft}})\rightarrow (Y, \mathcal{T}_{\oplus})\)
    is continuous.  
  Moreover, the map \({\rm pr}_{\mathsf{Y}}:\mathsf{Y}\rightarrow
    Y\) for the quotient topologies is open.
  \item 
  The quotient topology \({\mathcal T}_\oplus\) is metrizable and
    \(\oplus:X\rightarrow Y\) is an M-polyfold imprinting in the sense of
    Definition \ref{DEF_imprinting_m_poly}.
  \item 
  The triple \(\underline{\mathsf{Y}} = (\mathsf{Y}, Y, {\rm
    pr}_{\mathsf{Y}}) \in {\rm Ob}(\mathbf{VBL}_{\mathbb{K}})\) has the
    structure of an SB over an M-polyfold which is
    uniquely determined by the properies that \(\oplus_{\triangleleft} :
    \mathsf{X}\to \mathsf{Y}\) is a SB-map, and that all the
    maps \(H_\triangleleft : {\rm pr}_{\mathsf{Y}}^{-1}(V)\to \mathsf{X}\)
    are sc-smooth SB-maps.
  \item 
  Let \(\underline{X}' =( \mathsf{X}', X', {\rm pr}_{\mathsf{X}'})\in {\rm
    Ob}(\mathbf{VBL}_{\mathbb{K}})\) be a strong M-polyfold bundle over an
    M-polyfold, and suppose \(\Phi\in {\rm Mor}(\mathbf{VBL}_{\mathbb{K}})\)
    with \(\Phi: \underline{\mathsf{Y}}\to \underline{\mathsf{X}}'\). 
  Then \(\Phi\) is an sc-smooth SB-map if and only if
    \(\Phi\circ \oplus_{\triangleleft}: \underline{\mathsf{X}}\to
    \underline{\mathsf{X}}'\) is an sc-smooth SB-map.
  \item 
  Let \(\underline{X}' =( \mathsf{X}', X', {\rm pr}_{\mathsf{X}'})\in {\rm
    Ob}(\mathbf{VBL}_{\mathbb{K}})\) be an SB over an
    M-polyfold, and suppose \(\Phi\in {\rm Mor}(\mathbf{VBL}_{\mathbb{K}})\)
    with \(\Phi': \underline{\mathsf{X}}'\to \underline{\mathsf{Y}}\).
  Then \(\Phi'\) is a sc-smooth SB-map provided that the
    following hold.
  \begin{enumerate}                                                       
    \item 
    \(\Phi'\) is continuous
    \item 
    for each \(x'\in X'\) and \(V\in \mathcal{T}_{Y}\) with
    \(\phi'(x')=y\in V\), and for each \(H_{\triangleleft}: {\rm
    pr}_{\mathsf{Y}}^{-1}(V)\to \mathsf{X}\) associated to \(y\) as above,
    the map
    \begin{align*}                                                        
      H_\triangleleft \circ \Phi':( \Phi' \circ {\rm
      pr}_{\mathsf{Y}})^{-1}(V) \to \mathsf{X}
      \end{align*}
    is an sc-smooth SB-map.
    \end{enumerate}
  \end{enumerate}
\end{theorem}
%

The points (1) and (2) already follow from Exercise \ref{EX_10a}. 
The verification of the remaining claims are  left to the reader in the
  next exercise.\\

\begin{exercise}\hfill\\
Prove Theorem \ref{THM_consequences_strong_bundle_imprinting} (3)--(5).
\end{exercise}
%

The process of equipping a VBL-set with the structure of a
  SB via the imprinting method above, will be referred to as {\bf
  strong M-polyfold bundle construction}\index{strong M-polyfold bundle
  construction} or simply an SB-construction.
  \index{SB-construction}.

%
\subsection{Additional Results}\label{qsec4.2}\label{SEC_additional_results}
We list results which generalize previously proved theorems for sets to
  \({\rm vbl}\)-sets.
The proofs are left to the reader.
\begin{theorem}[composition of SB-constructions]
  \label{THMVV4.6}\label{THM_composition_of_SB_cons}
  \hfill\\
Let \(\underline{\mathsf{X}}, \underline{\mathsf{Y}},
  \underline{\mathsf{Y}}'\in {\rm Ob}(\mathbf{VBL}_{\mathbb{K}})\), and let
  \(\underline{\mathsf{X}}\) be a strong M-polyfold bundle.
Let \(\oplus_\triangleleft, \oplus_\triangleleft' \in {\rm
  Mor}(\mathbf{VBL}_{\mathbb{K}})\) with \(\oplus_\triangleleft :
  \underline{\mathsf{X}}\to \underline{\mathsf{Y}}\) and
  \(\oplus_\triangleleft' : \underline{\mathsf{Y}}\to
  \underline{\mathsf{Y}}'\).  
Suppose that \(\oplus_\triangleleft\) is a SB-imprinting, and suppose
  \(\oplus_\triangleleft'\) is surjective.  
Then the following two statements are equivalent.
\begin{enumerate}                                                         
  \item 
  When \(\underline{\mathsf{Y}}\) is equipped with the SB structure
  associated to the imprinting \(\oplus_\triangleleft\), the morphism
  \(\oplus_\triangleleft':\underline{\mathsf{Y}}\to
  \underline{\mathsf{Y}}'\) is an SB-imprinting.
  \item 
  \(\oplus_{\triangleleft}'\circ
  \oplus_{\triangleleft}:\underline{\mathsf{X}} \rightarrow
  \underline{\mathsf{Y}}'\) is an SB-imprinting.
  \end{enumerate}
Moreover, if either (1) or (2) hold, then the so does the other, and the
  two SB structures on \(\underline{\mathsf{Y}}'\) are in fact the same.
\end{theorem}
%

\begin{exercise}
\hfill\\
Prove Theorem \ref{THM_composition_of_SB_cons}.
\end{exercise}
%

The next result is a generalization of Theorem \ref{THMOP10.2}.

\begin{theorem}[strong bundle embeddings and imprintings]
  \label{THMVV4.7}\label{THM_sb_embeddings_imprintings}
  \hfill\\
Let \(\underline{\mathsf{X}}, \underline{\mathsf{X}}',
  \underline{\mathsf{Y}}, \underline{\mathsf{Y}}'\in {\rm
  Ob}(\mathbf{VBL}_{\mathbb{K}})\), and let 
  \(\oplus_\triangleleft, \oplus_\triangleleft',
  E_\triangleleft, \Phi_\triangleleft \in {\rm
  Mor}(\mathbf{VBL}_{\mathbb{K}})\) so that the following commutative
  diagram holds.
\begin{align*}                                                            
  \begin{CD}
  \underline{\mathsf{X}} @>\oplus_{\triangleleft}>> \underline{\mathsf{Y}}\\
  @A E_\triangleleft AA     @A \Phi_\triangleleft AA\\
  \underline{\mathsf{X}}' @>\oplus'_{\triangleleft} >> \underline{\mathsf{Y}}'
  \end{CD}
  \end{align*}
Additionally, suppose the following hold.
\begin{enumerate}                                                         
  \item 
  \(\underline{\mathsf{X}}\) and \(\underline{\mathsf{X}}'\) each have an
    SB-structure
  \item 
  \(\oplus_{\triangleleft}:\underline{\mathsf{X}} \to
    \underline{\mathsf{Y}}\) an SB imprinting
  \item 
  \(E_\triangleleft:\underline{\mathsf{X}}'\rightarrow
    \underline{\mathsf{X}}\) is a SB isomorphism onto a strong M-polyfold
    sub-bundle of \(\underline{\mathsf{X}}\)
  \item 
  \(\Phi_\triangleleft:\underline{\mathsf{Y}}'\rightarrow
    \underline{\mathsf{Y}}\) is injective
  \item 
  \(\oplus'_{\triangleleft}:\underline{\mathsf{X}}'\rightarrow
    \underline{\mathsf{Y}}'\) is surjective
  \end{enumerate}
If the identity
  \begin{align*}                                                          
    E_\triangleleft(\mathsf{X}') =
    \oplus_\triangleleft^{-1}(\Phi_\triangleleft(\mathsf{Y}'))
    \end{align*}
  holds, then \(\oplus_\triangleleft'\) is an SB-imprinting, and for the
  induced structure on \(\underline{\mathsf{Y}}'\) the morphism
  \(\Phi_\triangleleft\) defines a SB-isomorphism onto a strong
  M-polyfold sub-bundle of \(\underline{\mathsf{Y}}\).
\end{theorem}
%

\begin{corollary}[SB-imprintings and sub-M-polyfold bundles]
  \label{COR_sb_imprtintings_sub_polyfold_bundles}
  \hfill\\
Suppose \(\underline{\mathsf{X}}, \underline{\mathsf{X}}',
  \underline{\mathsf{Y}}, \underline{\mathsf{Y}}' \in {\rm
  Ob}(\mathbf{VBL}_{\mathbb{K}})\) and \(\oplus_\triangleleft:
  \underline{\mathsf{X}}\to \underline{\mathsf{Y}}\) is an SB-imprinting.
Suppose further that \(\mathsf{Y}'\subset \mathsf{Y}\), \(Y'\subset Y\),
  \(\oplus_\triangleleft^{-1}(\mathsf{Y}') = \mathsf{X}'\), and
  \(\oplus^{-1}(Y') = X'\).
Finally, suppose that \(\underline{\mathsf{X}}'\) has an SB-structure
  induced from \(\mathsf{\mathsf{X}}\).
Then \(\oplus_\triangleleft'\in {\rm Mor}(\mathbf{VBL}_{\mathbb{K}})\) is
  an SB-imprinting, where \(\oplus_\triangleleft':
  \underline{\mathsf{X}}'\to \underline{\mathsf{Y}}'\), is determined by
  \begin{align*}                                                            
    &\oplus_\triangleleft':\mathsf{X}'\to \mathsf{Y}'\\
    &\oplus_\triangleleft':= \oplus_\triangleleft \big|_{\mathsf{X}'}.
    \end{align*}
Additionally, \(\underline{\mathsf{Y}}'\) has an SB-structure induced from
  both \(\oplus_\triangleleft'\) and separately from the ambient space
  \(\mathsf{Y}\), however these two SB-structures are in fact the same.
\end{corollary}
%

\begin{exercise}\hfill\\
Prove Theorem \ref{THM_sb_embeddings_imprintings} and Corollary
  \ref{COR_sb_imprtintings_sub_polyfold_bundles}.\\
\end{exercise}
%

As a consequence of Theorem \ref{THM_sb_embeddings_imprintings} there is a
  well-defined pull-back operation.

\begin{definition}[SB-admissible morphism]
  \label{DEF_sb_admissible_morphism} 
  \hfill\\
Let \(\oplus_{\triangleleft}:\underline{\mathsf{X}}\rightarrow
  \underline{\mathsf{Y}}\) be an SB-imprinting, and let \(\Phi\in {\rm
  Mor}(\mathbf{VBL}_{\mathbb{K}})\) with
  \(\Phi:\underline{\mathsf{Y}}'\rightarrow \underline{\mathsf{Y}}\)
  injective.
Define \(\underline{\mathsf{X}} = (\mathsf{X}', X', {\rm
  pr}_{\mathsf{X}'})\in {\rm Ob}(\mathbf{VBL}_{\mathbb{K}})\) by the
  following:
  \begin{align*}                                                          
    \mathsf{X}':=\{x \in X\, :\,  \oplus_{\triangleleft}(x)\in
    \Phi(\mathsf{Y}')\},
    \qquad
    X' := X,
    \qquad
    {\rm pr}_{\mathsf{X}'}:= {\rm pr}_{\mathsf{X}}\big|_{\mathsf{X}'}.
    \end{align*}
We say that \(\Phi\) is an \emph{SB-admissible morphism}\index{SB-admissible
  morphism} provided the subset \(\underline{\mathsf{X}}'\) is a strong
  M-polyfold sub-bundle of \(\underline{\mathsf{X}}\).
\end{definition}
%

If \(\oplus_{\triangleleft}:\underline{\mathsf{X}}\rightarrow
  \underline{\mathsf{Y}}\) is an SB-imprinting and \(\Phi\in {\rm
  Mor}(\mathbf{VBL}_{\mathbb{K}})\) with
  \(\Phi:\underline{\mathsf{Y}}'\rightarrow \underline{\mathsf{Y}}\) is an
  SB-admissible morphism, we can pull-back an SB structure to
  \(\underline{\mathsf{Y}}'\) as follows.
First,  we define \(\underline{\mathsf{X}}'\in {\rm
  Ob}(\mathbf{VBL}_{\mathbb{K}})\) as in Definition
  \ref{DEF_sb_admissible_morphism} by the following:
  \begin{align*}                                                          
    \mathsf{X}':=\{x \in X\, :\,  \oplus_{\triangleleft}(x)\in
    \Phi(\mathsf{Y}')\},
    \qquad
    X' := X,
    \qquad
    {\rm pr}_{\mathsf{X}'}:= {\rm pr}_{\mathsf{X}}\big|_{\mathsf{X}'},
    \end{align*}
  which, by assumption of \(\Phi\) being an SB-admissible morphism, is a
  strong M-polyfold sub-bundle of \(\underline{\mathsf{X}}\).
We then define \(\oplus_\triangleleft'\in {\rm
  Mor}(\mathbf{VBL}_{\mathbb{K}})\) by
  \begin{align*}                                                          
    &\oplus_\triangleleft': \mathsf{X}'\to \mathsf{Y}',\\
    &\oplus_\triangleleft':=\Phi^{-1}\circ
    \oplus_{\triangleleft}\big|_{\mathsf{X}'}
    \end{align*}
  which has naturally underlying \(\oplus': X'\to Y'\) given by
  \(\oplus':=\phi^{-1}\circ \oplus\).
In view of Theorem \ref{THM_sb_embeddings_imprintings}, it follows that
  \(\oplus_\triangleleft': \underline{\mathsf{X}}'\to
  \underline{\mathsf{Y}}'\) is an SB-imprinting.
We say \(\underline{\mathsf{Y}}'\) has a pulled-back SB structure from
  \(\underline{\mathsf{X}}\) via \(
  \oplus_\triangleleft' = \Phi^{-1}\circ
  \oplus_\triangleleft\big|_{\mathsf{X}'}\).


Recall that \(\mathbf{VBL}_{\mathbb{K}}\) has a product and coproduct
  structure as given in equations (\ref{EQ_vbl_product}) and
  (\ref{EQ_vbl_coproduct}).
Consequently, we generalize Theorem \ref{THM_mpoly_product} and Theorem
  \ref{THM_mpoly_coproduct} via the following result.

\begin{theorem}[SB-product and SB-coproduct imprintings]
  \label{THMBB4.8} \label{THM_sb_product_coproduct_imprintings}
  \hfill\\
Let \(\underline{\mathsf{X}}, \underline{\mathsf{X}}',
  \underline{\mathsf{Y}}, \underline{\mathsf{Y}}'\in {\rm
  Ob}(\mathbf{VBL}_{\mathbb{K}})\), and let \(\oplus_\triangleleft,
  \oplus_\triangleleft'\in {\rm Mor}(\mathbf{VBL}_{\mathbb{K}})\) so that
  \(\oplus_\triangleleft: \underline{\mathsf{X}}\to \underline{
  \mathsf{Y}}\) and \(\oplus_\triangleleft': \underline{\mathsf{X}}'\to
  \underline{ \mathsf{Y}}'\) are SB-imprintings.
Then \(\underline{\mathsf{X}}\times \underline{\mathsf{X}}'\) and
  \(\underline{\mathsf{X}}\sqcup \underline{\mathsf{X}}'\) each have an
  SB-structure, and moreover each of
  \begin{align*}                                                          
    \oplus_\triangleleft \times \oplus_\triangleleft ': 
    \underline{\mathsf{X}}\times \underline{\mathsf{X}}' \to 
    \underline{\mathsf{Y}}\times \underline{\mathsf{Y}}'
    \end{align*}
  and 
  \begin{align*}                                                          
    \oplus_\triangleleft \sqcup \oplus_\triangleleft ': 
    \underline{\mathsf{X}}\sqcup \underline{\mathsf{X}}' \to 
    \underline{\mathsf{Y}}\sqcup \underline{\mathsf{Y}}'
    \end{align*}
  are SB-imprintings.
\end{theorem}
%

\begin{exercise}\hfill\\
Prove Theorem \ref{THM_sb_product_coproduct_imprintings}.\\
\end{exercise}
%

\begin{definition}[SB-restrictions]
  \label{DFE_sb_restrictions}
  \hfill\\
Let \(\oplus_\triangleleft\in {\rm Mor}(\mathbf{VBL}_{\mathbb{K}})\) with
  \(\oplus_\triangleleft: \underline{\mathsf{X}}\to \underline{\mathsf{Y}}\)
  an SB-imprinting.
Suppose further that \(I\) is a finite set, and
  \(\{\underline{\mathsf{B}}_i\}_{i\in I} \subset {\rm
  Ob}(\mathbf{VBL}_{\mathbb{K}})\) so that each \(\underline{\mathsf{B}}_i\)
  has an SB-structure.
Finally, suppose \(\mathbf{p} = \{p_i\}_{i\in I}\subset {\rm
  Mor}(\mathbf{VBL}_{\mathbb{K}})\) are morphisms of the form \(p_i:
  \underline{\mathsf{Y}}\to \underline{\mathsf{B}}_i\) and have the property
  that the compositions
  \begin{align*}                                                          
    p_i\circ \oplus_\triangleleft :\underline{\mathsf{X}}\to
    \underline{\mathsf{B}}_i
    \end{align*}
  yield sc-smooth SB-maps \(p_i \circ \oplus_\triangleleft: \mathsf{X}\to
  \mathsf{B}_i\) for each \(i\in I\).  
In this case, we call \(\oplus_\triangleleft\) an SB-imprinting with
  SB-restrictions \(\mathbf{p}\), or the pair \((\oplus_\triangleleft,
  \mathbf{p})\)\index{\((\oplus_{\triangleleft},\bm{p}_{\triangleleft})\)}
  an SB-imprinting with restrictions.
\end{definition}
%

The following definition will be crucial for fibered product constructions.

\begin{definition}[SB-plumbable]
  \label{DEF_sb_plum}\index{SB-plumbable}
  \hfill\\
Let \((\oplus_\triangleleft, \mathbf{p})\) and \((\oplus_\triangleleft',
  \mathbf{p}')\) be SB-imprintings with restrictions.
Here \(\oplus_\triangleleft:\underline{\mathsf{X}}\to
  \underline{\mathsf{Y}}\) and
  \(\oplus_\triangleleft':\underline{\mathsf{X}}'\to
  \underline{\mathsf{Y}}'\), and \(\mathbf{p} = \{p_i\}_{i\in I}\) and
  \(\mathbf{p}' = \{p_{i'}\}_{i'\in I'}\)  with \(p_i:
  \underline{\mathsf{Y}}\to \underline{\mathsf{B}}_i\) and \(p_i':
  \underline{\mathsf{Y}}'\to \underline{\mathsf{B}}_i'\).
For a fixed pair \((i_0, i_0')\in I\times I'\), we say that
  \((\oplus_\triangleleft, \mathbf{p})\) and \((\oplus_\triangleleft',
  \mathbf{p}')\) are \((i_0, i_0')\)-plumbable provided that
  \(\underline{\mathsf{B}}_{i_0} = \underline{\mathsf{B}}_{i_0'}'\) and
  \begin{align*}                                                          
    \mathsf{X}{_{i_0}\times_{i_0'}}
    \mathsf{X}':=\{(\mathsf{x},\mathsf{x}')\in \mathsf{X}\times \mathsf{X}':
    p_{i_0} \circ \oplus_\triangleleft(\mathsf{x}) = p_{i_0'}'\circ
    \oplus_\triangleleft'(\mathsf{x}')\}
    \end{align*}
  is a strong M-polyfold sub-bundle of \(\mathsf{X}\times \mathsf{X}'\).
\end{definition}
%

\begin{definition}[\((i_0,i_0')\)-plumbing]
  \label{DEF_ii_plumbing}
  \hfill\\
Let \((\oplus_{\triangleleft},\bm{p})\) and
  \((\oplus'_{\triangleleft},\bm{p}')\) be SB-imprintings with
  restrictions, which are also are \((i_0,i_0')\)-{plumbable}.
Here \(\oplus_\triangleleft:\underline{\mathsf{X}}\to
  \underline{\mathsf{Y}}\) and
  \(\oplus_\triangleleft':\underline{\mathsf{X}}'\to
  \underline{\mathsf{Y}}'\), and \(\mathbf{p} = \{p_i\}_{i\in I}\) and
  \(\mathbf{p}' = \{p_{i'}\}_{i'\in I'}\)  with \(p_i:
  \underline{\mathsf{Y}}\to \underline{\mathsf{B}}_i\) and \(p_i':
  \underline{\mathsf{Y}}'\to \underline{\mathsf{B}}_i'\).
We define the \((i_0, i_0')\)-plumbing of this pair via
  \begin{align*}                                                          
    (\oplus_\triangleleft'', \mathbf{p}''):= (\oplus_\triangleleft,
    \mathbf{p}) \; {_{i_0}\times_{i_0'}} \; (\oplus_\triangleleft',
    \mathbf{p}')
    \end{align*}
  where
  \begin{align*}                                                          
    &\oplus_\triangleleft'': \underline{\mathsf{X}}{_{i_0}\times_{i_0'}}
    \underline{\mathsf{X}}'\to
    \underline{\mathsf{Y}}{_{i_0}\times_{i_0'}} \underline{\mathsf{Y}}'
    \\
    &\oplus_\triangleleft'':= \oplus_\triangleleft\times
    \oplus_\triangleleft'\big|_{\mathsf{X}{_{i_0}\times_{i_0'}}
    \mathsf{X}'}
    \end{align*}
  and we have abused notation a bit by abbreviating 
  \begin{align*}                                                          
    \underline{\mathsf{Y}}{_{i_0}\times_{i_0'}} \underline{\mathsf{Y}}' :=
    \underline{\mathsf{Y}}{_{p_{i_0}}\times_{p_{i_0'}'}}
    \underline{\mathsf{Y}}'. 
    \end{align*}
Because \((\oplus_{\triangleleft},\bm{p})\) and
  \((\oplus'_{\triangleleft},\bm{p}')\) are \((i_0, i_0')\)-plumbable, it
  follows that the fiber-priduct \(
  \underline{\mathsf{X}}{_{i_0}\times_{i_0'}} \underline{\mathsf{X}}'\)
  has an SB-structure induced from \(\underline{\mathsf{X}}\).
Consequently, by Corollary \ref{COR_sb_imprtintings_sub_polyfold_bundles},
  it follows that \(\oplus_\triangleleft'':
  \underline{\mathsf{X}}{_{i_0}\times_{i_0'}}\underline{\mathsf{X}}\to
  \underline{\mathsf{Y}}{_{i_0}\times_{i_0'}}\underline{\mathsf{Y}}'\) is an
  SB-imprinting, and the SB-structure on 
  \(\underline{\mathsf{Y}}{_{i_0}\times_{i_0'}}\underline{\mathsf{Y}}'\)
  induced from \(\underline{\mathsf{Y}}\times\underline{\mathsf{Y}}'\)
  agrees with the SB-structure induced from \(\oplus_\triangleleft''\).

Next we aim to define restrictions \(\mathbf{p}''\).
To that end, we first define a new index set \(I\) via the following.
  \begin{align*}                                                          
    I'' = \{1\}\times (I\setminus \{i_0\})\; \; \bigcup \; \; \{2\}\times
    (I'\setminus \{i_0'\})
    \end{align*}
Next, for each \(i'' = (i_1'', i_2'') \in I''\), we define
  \begin{align*}                                                          
    p_{i''}'' = 
    \begin{cases}
    p_{i_2''}\circ \pi_1  &\text{if }i_1'' = 1
    \\
    p_{i_2''}' \circ \pi_2 &\text{if }i_1'' = 2
    \end{cases}
    \end{align*}
  where \(\pi_1\) and \(\pi_2\) are the projections from the fibered
  product onto the first and second factor respectively, and hence we
  can define the indexed set of restrictions as follows.
  \begin{align*}                                                          
    \mathbf{p}'' = \{p_{i''}''\}_{i''\in I''}
    \end{align*}
We then refer to the pair \((\oplus_\triangleleft'', \mathbf{p}'')\) as
  the \((i_0, i_0')\)-plumbing of \((\oplus_\triangleleft, \mathbf{p})\)
  with \((\oplus_\triangleleft', \mathbf{p}')\).

\end{definition}
%

We leave it to the reader to work out results and statements 
dealing with restrictions.

%
\section{Functorial Method for  M-Polyfold Constructions}
  \label{CONST-F}\label{SEC_functorial_method}
The following results are very useful for the construction of M-polyfolds
  and strong bundles over them. 
In particular, the results show that certain constructions which can be
  carried out in a sufficiently functorial way for the standard vector
  spaces \({\mathbb R}^n\) have natural extensions to smooth manifolds.
We shall refer to the method presented below in the space as well as
  strong bundle context as the {\bf embedding method}.

%
\subsection{M-Polyfold Construction Functors}
  \label{qsec5.1}\label{SEC_construction_functors}
We consider the category \({\mathcal N}\)\index{\({\mathcal N}\)}, which
  has non-negative integers as objects, and has morphisms given as follows.
For \(N, L\in {\rm Ob}(\mathcal{N})\), we define the morphisms
  \({\rm Mor}(\mathcal{N})\) to consist of maps \(f:\mathbb{R}^N\to
  \mathbb{R}^L\); to be clear, here \(N\) is the
  source of \(f\) and \(L\) is the target of \(f\).

\begin{definition}[M-polyfold construction functor]
  \label{def:poly construction}
  \label{DEF_poly_construction_functor}
  \hfill\\
An \emph{M-polyfold construction functor}\index{M-polyfold construction
  functor} over \({\mathcal N}\) consists of a covariant functor \(X\),
  which does the following.
To each each natural number \(N\) it associates an M-polyfold which we
  denote \(X(N)\), and to each morphism \(f:N\rightarrow L\) it associates
  an sc-smooth map \(X(f):X(N)\rightarrow X(L)\).
Moreover we require that the M-polyfolds are equipped with an additional
  structure, namely that for each object \(N\) we have a map which
  associates to a point \(u\in X(N)\) a subset \({\rm im}(u)\subset
  {\mathbb R}^N\), which we call the image of \(u\).
The following is assumed to hold:
  \begin{itemize}
    \item[(1)] 
      Given an open subset \(U\) of \({\mathbb R}^N\) the subset of
      \(X(N)\) which consists of all \(u\) with \({\rm im}(u)\subset U\)
      is open in \(X(N)\).
    \item[(2)] 
      We have \({\rm im}(X(f)(u))=f({\rm im}(u))\).
    \item[(3)] 
      If \(f,g:N\rightarrow L\) are morphisms and \(u\in X(N)\), then 
      \begin{equation*}                                                   
	f\big|_{{\rm im}(u)}=g	\big|_{{\rm
	im}(u)}\qquad\text{implies}\qquad X(f)(u)=X(g)(u)
	\end{equation*}
    \end{itemize}
\end{definition}
%

\begin{remark} 
  \hfill\\
A typical situation is where \(X(N)\) consists of certain maps
  \(u:\Omega_u\rightarrow {\mathbb R}^N\), where the domain \(\Omega_u\)
  depends on \(u\).
Then we can take \({\rm im}(u)=u(\Omega_u)\) and \(X(f)(u)=f\circ u\).
However in the generality of our definition an element \(u\) does not
  need to be a map and \({\rm im}(u)\) can be an abstract construction
  which cannot be identified with an image of a map.
\end{remark}

%

We shall let \(\mathbf{MPoly}\)\index{\(\mathbf{MPoly}\)} denote the category of
  M-polyfolds with the sc-smooth maps between them.
We shall also let \({\mathcal P}\) denote the category consisting of
  smooth manifolds (without boundaries) as objects, and with morphisms given
  by smooth maps between them.
Given an object in \({\mathcal P}\), each connected component
  admits a proper smooth embedding into some \({\mathbb R}^N\). 
The first observation is given by the following proposition.

\begin{proposition}[extension of M-polyfold construction functors]
  \label{EXTEND}\label{PROP_extension_m_poly_construction_functors}
  \hfill\\
Given an M-polyfold construction functor \(X:{\mathcal N}\rightarrow
  \mathbf{MPoly} \), there exists a natural extension to a functor \(X:{\mathcal
  P}\rightarrow \mathbf{MPoly}\), which associates to a each smooth manifold
  \(M\in {\rm Ob}(\mathcal{P})\) an M-Polyfold \(X(M)\in {\rm
  Ob}(\mathbf{MPoly})\), and to each smooth map \(f:M\to M'\) it associates
  an sc-smooth map between M-polyfolds:
  \begin{equation*}                                                       
    X(f): X(M)\rightarrow X(M').
    \end{equation*}
Furthermore we have a natural sc-diffeomorphism \(X(N)\rightarrow
  X({\mathbb R}^N)\).
\end{proposition}
%
\begin{proof}
We assume that the manifold \(M\) in \({\mathcal P}\) is connected so that
  it has a proper embedding into some \({\mathbb R}^N\).
We shall first construct \(X(M)\) for such \(M\). 
In the case \(M\) has several connected components \(M=\coprod
  M_\lambda\) we shall define \(X(M)=\coprod X(M_\lambda)\).

Consider a smooth proper embedding \(\phi:M\rightarrow {\mathbb R}^N\). 
Then define the following set of functions.
\begin{align*}                                                            
  X(\phi) := \big\{ (u, \phi) \in X(N)\times \{\phi\}: {\rm im}(u)\subset
  \phi(M)\big\}
  \end{align*}
Given another proper embedding \(\psi:M\rightarrow {\mathbb R}^L\)
  we will likewise consider \(X(\psi)\).
We say \((u,\phi)\) is related to \((v,\psi)\) provided there exists a
  morphism \(f:N\rightarrow L\) such that
  \begin{equation}\label{eq:equiv rel X}                                  
    X(f)(u) = v\quad \text{and}\quad  f\big|_{\phi(M)}=\psi\circ
    \phi^{-1}:\phi(M)\rightarrow \psi(M).
    \end{equation}
We claim this is an equivalence relation.
To see this, we first must show reflexivity. 
Observe that by taking \(f=Id\), the second part of equation
  (\ref{eq:equiv rel X}) is trivially true, and the first part, namely
  that \(X(Id)(u)=u\), follows from the fact that \(X\) is a functor.
Next we show symmetry, and to that end we suppose \((u, \phi)\) is
  related to \((v, \psi)\) so that there exists a morphism \(f:N\to L\)
  such that equation (\ref{eq:equiv rel X}) holds.
We must establish the existence of a morphism \(h:L \to N\) for which
  \(X(h)(v) = u\) and \(h\big|_{\psi(M)}=\phi\circ \psi^{-1}\).
Note that since \(h\) need only be a smooth map from \(\mathbb{R}^L\to
  \mathbb{R}^N\) instead of an embedding, it is elementary to find an \(h\)
  which satisfies the latter condition.
Thus we must show this \(h\) satisfies \(X(h)(v) = u\).
To that end, we recall that since we assumed (\ref{eq:equiv rel X}),
  that means \(v=X(f)(u)\).
Applying \(X(h)\) to each side of this equation, we then compute
  \begin{align*}                                                          
    X(h)(v) 
    &=X(h)X(f)(u)\\
    &=X(h\circ f)(u)\\
    &=X(Id)(u)\\
    &=u
    \end{align*}
  where to obtain the third inequality we have made use of property (3) of
  Definition \ref{DEF_poly_construction_functor} applied to the morphism \(h\circ
  f\) and \(Id\); to apply property (3) we also make use of the fact that
  \begin{equation*}                                                       
    f\big|_{\phi(M)} = \psi\circ\phi^{-1}\qquad \text{and}\qquad
    h\big|_{\psi(M)}= \phi\circ\psi^{-1},
    \end{equation*}
  so that \(h\circ f\big|_{\phi(M)} = Id\), and hence \(h\circ
  f\big|_{{\rm im}(u)} = Id\) since \({\rm im}(u)\subset \phi(M)\).
Finally, transitivity is straightforward to verify, and hence the
  relation above is indeed an equivalence relation.
Denote an equivalence class by \([u,\phi]\) and denote by \(X(M)\) the
  collection of all these.

Next we observe that given smooth proper embeddings \mbox{\(\phi:M
  \to \mathbb{R}^N\)} and \(\psi:M \to \mathbb{R}^L\) and a morphism
  \(f:N\to L\) for which \mbox{\(f\big|_{\phi(M)}=\psi\circ\phi^{-1}\)},
  we obtain a bijection
  \begin{equation*}                                                       
    X(f)\big|_{X(\phi)}:X(\phi)\rightarrow X(\psi).
    \end{equation*}
Consequently, given any equivalence class \([v,\psi]\) and a proper
  embedding \(\phi:M\rightarrow {\mathbb R}^N\) there is a unique
  representative of the form \([u,\phi]\), and thus we have
  natural bijections
  \begin{equation*}                                                       
    \phi_\sharp:X(M)\rightarrow X(\phi):[u,\phi]\rightarrow (u,\phi).
    \end{equation*}
Observe that \(\psi_\sharp\circ \phi_\sharp^{-1}:X(\phi)\rightarrow
  X(\psi)\) is given by
  \begin{equation*}                                                       
    \psi_\sharp\circ \phi_\sharp^{-1}(u) = X(f)(u)
    \end{equation*}
  where \(X(f)\) is the restriction of an sc-smooth map.  
Finally, it remains to show that \(X(\phi)\) has a natural M-polyfold
  structure.
Note that once this is established, it immediately follows that the
  transition maps \(\psi_\sharp\circ \phi_\sharp^{-1}\) are sc-smooth, and
  hence \(X(M)\) has an M-polyfold structure which is sc-diffeomorphic to
  the \(X(\phi)\), and this M-polyfold structure is independent of the
  choices involved.
Thus we focus our attention on showing that \(X(\phi)\) has a natural
  M-polyfold structure.

We have a proper embedding \(\phi:M\rightarrow {\mathbb R}^N\) and can
  take a tubular neighborhood \(U\) around the image \(\phi(M)\) with a
  bundle projection \(r:U\rightarrow \phi(M)\subset U\).
Note that \(r\circ r=r\), i.e. \(r\) is a smooth retraction.
Perhaps making \(U\) somewhat smaller we may assume that \(r\) is the
  restriction of a morphism \(R:N\rightarrow N\).
Of course \(R\circ R=R\) will not hold in general.

By property (1) of Definition \ref{DEF_poly_construction_functor} we have that 
  \begin{equation*}                                                       
    X(U)=\{u\in X(N) : {\rm im}(u)\subset U\}  
    \end{equation*}
  is an open subset of \(X(N)\) and therefore has a M-polyfold structure.
The morphism \(R\) induces an sc-smooth map
  \begin{equation*}                                                       
    X(R)\big|_{X(U)}:X(U)\rightarrow X(U).   
    \end{equation*}
Since \(R\big|_{U}=r\) and \(r\circ r = r\), it follows that on \(U\)
  we have \mbox{\(R\circ R = R\)}.
Consequently, since for each \(u\in X(U)\) we have \({\rm im}
  (u)\subset U\), it trivially follows that \(R\circ R\big|_{{\rm im}(u)}
  = R\big|_{{\rm im}(u)}\), and hence by property (3) of Definition
  \ref{DEF_poly_construction_functor} we have
  \begin{equation*}                                                       
    \big(X(R)\big|_{X(U)}\big) \circ
    \big(X(R)\big|_{X(U)}\big)=X(R)\big|_{X(U)}.
    \end{equation*}
Hence \(X(R)\big|_{U}:X(U)\rightarrow X(U)\) is an sc-smooth retraction.
This implies that the image of \(X(R)\big|_{X(U)}\) has a M-polyfold
  structure.
By property (2) of Definition \ref{DEF_poly_construction_functor} we have that
  for \(u\in X(U)\)
  \begin{equation*}                                                       
    \text{im}\big(X(R)(u)\big) = R(\text{im}(u)) = r(\text{im}(u)) \subset
    \phi(M).
    \end{equation*}
Hence for all \(u\in X(U)\) we have
  \begin{equation*}                                                       
    \big(X(R)(u),\phi\big)\in X(\phi).
    \end{equation*}
Observe that for each \((u,\phi)\in X(\phi)\) we then immediately
  have \mbox{\(X(R)(u)=u\)}, since by property (3) of Definition
  \ref{DEF_poly_construction_functor} we have
  \begin{equation*}                                                       
    X(R)(u) = X(Id)(u) =u,
    \end{equation*}
  where we have used the fact that \(R=Id\) on \(\phi(M)\). 
Thus we have shown that  \(X(\phi)\) has a natural M-polyfold structure,
  and therefore \(X(M)\) has a well-defined M-polyfold structure, which
  is natural in the sense that it does not depend on choices involved.

Also observe that \(X(N)\) and \(X({\mathbb R}^N)\) are naturally
  sc-diffeomorphic via the map
  \begin{equation*}                                                       
    X({\mathbb R}^N)\rightarrow X(N):[u,Id_{{\mathbb R}^N}]\rightarrow u.
    \end{equation*}

Assume next \(f:M\rightarrow W\) is a smooth map between two manifolds
  in \({\mathcal P}\).
Then we obtain a natural sc-smooth map
  \begin{equation*}                                                       
    X(f):X(M)\rightarrow X(W)
    \end{equation*}
  defined by
  \begin{equation*}                                                       
    [u,\phi] \rightarrow [X(F)(u),\psi],
    \end{equation*}
  where we take any morphism \(F:N\rightarrow L\) such that
  \mbox{\(F\big|_{\phi(M)}= \psi\circ f\circ \phi^{-1}\)}.
The choice of \(F\) is irrelevant by property  (3) as long as it
  satisfies the this property over \(\phi(M)\).
\end{proof}

We note that in practice, it is generally easier to build M-polyfolds of
  maps into \(\mathbb{R}^N\) rather than a more general finite dimensional
  manifold.
Furthermore, often such a construction can be identified as an M-polyfold
  construction over \({\mathcal N}\), in which case the above abstract
  result guarantees that these constructions extend naturally to smooth
  manifolds without boundary.
We also note that there is an analogous result for the construction of
  strong bundles, and similar versions for sc-manifolds, ssc-manifolds,
  and their strong bundles, however we leave the sc-manifold and
  ssc-manifold formulations to the reader.
  
Let us view our gluing construction \(X^{3,\delta}_{\varphi}\), made
  precise in equation (\ref{EQ_example_construction}), as a
  functor which associates to \({\mathbb R}^N\) the M-polyfold
  \(X^{3,\delta}_{\varphi}({\mathbb R}^N)\) with its M-polyfold structure
  induced from the imprinting method associated to the pre-gluing map
  \(\oplus\).
Then to a smooth map \(f:{\mathbb R}^N\rightarrow {\mathbb R}^L\), we
  define the associated map
  \begin{align}
    f_\sharp: X^{3,\delta}_{\varphi}({\mathbb R}^N)\rightarrow
    X^{3,\delta}_{\varphi}({\mathbb R}^L)\qquad\text{by} \qquad
    f_\sharp(u) = f\circ u.
    \end{align}
By factoring \jwf{[How does this factor like this?]}
  \begin{align}
    f_{\sharp}(u) =     \oplus_L \circ (\text{Id}_{\mathbb B}\times
    (f\times f))\circ H_N(u)
    \end{align}
  for a suitable \(H_N\) with \(\oplus_N\circ H_N=Id\) and
  \(H_N\circ\oplus_N\) being sc-smooth the sc-smoothness of \(f_\sharp\)
  follows from the fact, proven in \cite{El}, that \(f\times
  f:E_N\rightarrow E_L:(u^+,u^-)\rightarrow (f\circ u^+,f\circ u^-)\) is
  ssc-smooth.

Hence \(X^{3,\delta}_{\varphi}\) is an M-polyfold construction functor in
  the sense of Definition \ref{DEF_poly_construction_functor}, and hence by
  Proposition \ref{PROP_extension_m_poly_construction_functors} admits an
  extension from \({\mathcal N}\) to \(\mathbf{MPoly}\).
One can verify that the previous discussion about construction functors
  applies and we have a natural extension of our functor to the category of
  smooth manifolds with smooth maps between them.

One can generalize this discussion also to cover the construction of
  strong bundles.

%
\subsection{Strong Bundle Constructions}
  \label{qsec5.2}\label{SEC_strong_bundle_constructions}

The above construction can be regarded as a certain prototype for
  M-polyfold constructions in general.
There will be variations that we will make use of later, like when the
  target manifold is changing, however we will address these scenarios
  as needed.
For now we focus on a similar construction for strong bundles.
In this case we first define a category \(\mathcal{N}^2\) which has
  objects given by pairs \((N,L)\) of non-negative integers, and the
  morphisms \((N,L)\rightarrow (K,M)\) are smooth maps
  \begin{equation*}                                                       
    F:{\mathbb N}^N\times {\mathbb R}^L\rightarrow {\mathbb R}^K\times
    {\mathbb R}^M
    \end{equation*}
  of the form
  \begin{equation*}                                                       
    F(u,h) = (f(u),A(u)h),
    \end{equation*}
  where \(f:{\mathbb R}^N\rightarrow {\mathbb R}^K\) and \(A:{\mathbb
  R}^N\rightarrow \mathcal{L}({\mathbb R}^L,{\mathbb R}^M)\) are smooth maps.

\begin{definition}[strong bundle (SB) construction funtor]
  \label{DEF_sb_construction_functor}
  \hfill\\
A strong bundle (SB) construction functor over \({\mathcal N}^2\) consists
  of a covariant functor \(Z\) which associates to a pair of natural
  numbers a strong bundle \(Z(N,L)\) with projection
  \begin{equation*}                                                       
    p:Z(N,L)\rightarrow X(N)
    \end{equation*}
  to the underlying base M-polyfold \(X(N)\), where \(X\) is a M-polyfold
  construction over \({\mathcal N}\).
The strong bundles come with an additional structure, namely for each
  \(z\in Z(N,L)\) we have an associated subset \({\rm im}(z)\subset
  \mathbb{R}^N\times \mathbb{R}^L \) so that \({\rm im}(p(z)) ={\rm
  pr}({\rm im}(z))\), where \({\rm pr}:{\mathbb R}^N\times {\mathbb
  R}^L\rightarrow {\mathbb R}^N\) is the canonical projection.
In addition the following holds.
\begin{itemize}
  \item[(1)]  
    Given an open subset \(U\subset {\mathbb R}^N\) the subset of \(Z(N,L)\)
    consisting of all \(z\) with \({\rm im}(p(z))\subset U\) is open.
  \item[(2)] 
    For each morphism \(F\in {\rm Mor}(\mathcal{N}^2)\) we have 
  \begin{equation*}                                                       
    {\rm im}(Z(F)(z))=F({\rm im}(z)).
    \end{equation*}
  \item[(3)]
    If \(F,G:(N,L)\rightarrow (K,M)\) are morphisms \(F, G\in {\rm
    Mor}(\mathcal{N}^2)\), and \(z\in Z(N,L)\), then
  \begin{equation*}                                                       
    F\big|_{{\rm im}(z)} =  G\big|_{{\rm
    im}(z)}\qquad\text{implies}\qquad Z(F)(z)=Z(G)(z).
    \end{equation*}
\end{itemize}
\end{definition}
%

As before we have the following extension result. 
Let \({\mathcal {SP}}\) be the category whose objects are smooth vector
  bundles over manifolds (without  boundary), \(E\rightarrow M\). 
If we consider \(E\) over a connected component of \(M\) there exists a
  smooth proper bundle embedding
  \begin{equation*}                                                       
    \Phi:E\rightarrow {\mathbb R}^N\times {\mathbb R}^P
    \end{equation*}
  for suitable \((N,P)\), which is linear on the fibers and covers a
  smooth proper embedding \mbox{\(\phi:M \to \mathbb{R}^N\)}.
For such a bundle embedding \(\Phi\) covering \(\phi\) we find a tubular
  neighborhood \(U\) around \(\phi(M)\), and a smooth map \(r:U\to U\) for
  which \(r\circ r = r\) and \(r(U) = \phi(M)\) given by local
  projection.
Then we can extend \(\Phi(E)\rightarrow \phi(M)\) to a bundle over \(U\)
  so that the fiber over \(p\in U\) is the fiber over \(r(p)\).
We then find a smooth map \(F:{\mathbb R}^N\times {\mathbb R}^P\rightarrow
  {\mathbb R}^N\times {\mathbb R}^P\) with the property that the
  restriction
  \begin{equation*}                                                       
    F:U'\times \mathbb{R}^P\to U'\times \mathbb{R}^P
    \end{equation*}
  satisfies
  \begin{equation*}                                                       
    F(p, h) = \big(r(p), \pi_{r(p)}(h)\big),    
    \end{equation*}
  where \(\phi(M)\subset U'\subset \bar{U}'\subset U\), and where \(\pi_u\)
  is the orthogonal projection onto the fiber over \(u\in \phi(M)\).
Then by construction, for each \((p,h)\in U'\times {\mathbb R}^P\) we have
  \(F\circ F=F\). 
Noting that our smooth map \(F\) can be regarded as the restriction of a
  morphism \(F\in {\rm Mor}(\mathcal{N}^2)\), we then claim the following.

\begin{proposition}[extension of SB construction functors]
  \label{PORP_extension_sb_construction_functor}
  \hfill\\
The functor \(Z\) from a strong bundle construction over  \({\mathcal N}^2\)
  has a natural associated functor \(Z\) defined on \({\mathcal {SP}}\), which
  associates to a vector bundle  \(E\rightarrow M\) in \({\mathcal {SP}}\) a
  strong bundle \(Z(E)\) and to a smooth vector bundle map \(F:E\rightarrow
  E'\) covering \mbox{\(f:M\rightarrow M'\)} an sc-smooth strong bundle map
  between  \(Z(E)\rightarrow Z(E')\) covering \(X(f):X(M)\rightarrow X(M')\).
Further we have a natural strong bundle isomorphism \(Z(N,P)\rightarrow
  Z({\mathbb R}^N\times {\mathbb R}^P)\).
\end{proposition}
%
\begin{proof}
We only sketch the setup.
The basic observation is that for two proper bundle embeddings \(\Phi\)
  and \(\Psi\) the composition \(\Phi\circ \Psi^{-1}\) is the restriction of
  a global morphism, \(\Phi\circ \Psi^{-1}= F\big|_{\Psi(E)}\).
Then \(Z(E)\) is  defined as the collection of
  equivalence classes \([u, \Phi]\) where \(\Phi\) is a smooth proper
  bundle embedding, and \(u\in Z(N, P)\) for which \({\rm im }(u)\subset
  \Phi(E)\).
The proof then closely parallels the proof of Proposition \ref{EXTEND};
  we leave the details to the reader.
\end{proof}

\vspace{10cm}

\bibliographystyle{\alpha}


\printindex
\end{document}